\documentclass[a4paper, final]{article}

\newcommand{\version}{3.2}

\author{A. Fornasiero}
\def\shorttitle{Tame structures, v. \version}
\title{Tame structures and open cores\\
\normalsize{Version \version}}

\usepackage{ifpdf}

\usepackage{amsmath, amsthm, amssymb}
\usepackage{xspace}

\usepackage[notref,notcite]{showkeys}

\usepackage{tocloft}

\usepackage{comment}

\usepackage{paralist}

\usepackage{fancyhdr}

\usepackage[T1]{fontenc}

\ifpdf
\usepackage[pdftex, linktocpage=true, pdfborder={0 0 0}, plainpages=false,%
 pdfpagelabels, pdfdisplaydoctitle=true, bookmarks=true, colorlinks=false]{hyperref}
\hypersetup{pdfauthor={A. Fornasiero},
  pdftitle={Tame structures},
  pdfkeywords={O-minimal, open core, ordered field, definably complete, 
    dense pair, d-minimal, locally o-minimal, Baire, tame},
  pdfsubject={Tameness in ordered structures, generalizations of o-minimality,
  d-minimality}}
\else
\usepackage[plainpages=false,linktocpage=true,pdfpagelabels]{hyperref}
\fi

\makeatletter
\renewcommand\@makefnmark{\@textsuperscript{\normalfont(\@thefnmark)}}
\makeatother

\newlength{\tocwidth}
\makeatletter
\providecommand{\cftdotfill}{\@cftdotfill}
\makeatother

\newcommand*{\intro}[1]{\textbf{#1}}
\newcommand*{\Pa}[1]{\bigl( #1 \bigr)}
\newcommand*{\set}[1]{\{#1\}}
\newcommand*{\abs}[1]{\lvert#1\rvert}
\newcommand*{\card}[1]{\lvert#1\rvert}
\newcommand{\Nat}{\mathbb{N}}
\newcommand{\Real}{\mathbb{R}}
\newcommand{\Raz}{\mathbb{Q}}
\newcommand{\Ralg}{\Real^{\mathrm{alg}}}
\newcommand{\Rbar}{\bar\Real}
\newcommand*{\inter}[1]{\mathring{#1}}
\DeclareMathOperator{\interior}{int}
\newcommand*{\cl}[1]{\overline{#1}}
\DeclareMathOperator{\cll}{cl}
\newcommand{\fr}{\partial}
\DeclareMathOperator{\bd}{bd}
\DeclareMathOperator{\fdim}{fdim}
\DeclareMathOperator{\lc}{lc}
\newcommand*{\nlc}[1]{\ulcorner #1 \urcorner}
\newcommand*{\inlc}[2]{{#1}^{\ulcorner #2 \urcorner}}
\newcommand{\Dis}{\mathcal{D}}
\DeclareMathOperator{\Fin}{Fin}
\newcommand{\rest}{\upharpoonright}
\newcommand{\bij}{\overset{\sim}{\longrightarrow}}
\newcommand{\clB}{\cl B}
\DeclareMathOperator{\id}{id}

\newcommand{\Zclosed}{Z\hyph closed\xspace}
\newcommand{\Zclosure}{Z\hyph closure\xspace}
\newcommand{\Zbasis}{Z-basis\xspace}
\newcommand{\Zgenerating}{Z-generating\xspace}
\newcommand{\Zindependent}{Z-independent\xspace}
\newcommand{\Zdimension}{Z-dimension\xspace}
\newcommand{\Zfree}{Z-free\xspace}
\newcommand{\Zapplication}{Z\hyph application\xspace}

\DeclareMathOperator{\dcl}{dcl}
\DeclareMathOperator{\acl}{acl}
\DeclareMathOperator{\zcl}{Zcl}
\newcommand{\zclK}{\zcl^{\K}}
\newcommand{\zclM}{\zcl^{\monster}}
\DeclareMathOperator{\scl}{Scl}
\DeclareMathOperator{\sdim}{Sdim}

\DeclareMathOperator{\lexmin}{lex min}
\DeclareMathOperator{\lexinf}{lex inf}
\newcommand*{\gen}[1]{\langle #1 \rangle}
\newcommand{\app}{\leadsto}

\newcommand*{\pair}[1]{\langle #1 \rangle}
\newcommand{\Am}{\mathbb A}
\newcommand{\Bm}{\mathbb B}

\newcommand{\monster}{\mathbb M}

\newcommand{\av}{\bar a}

\newcommand{\bv}{\bar b}

\newcommand{\cv}{\bar c}

\newcommand{\x}{\bar x}
\newcommand{\y}{\bar y}
\newcommand{\z}{\bar z}

\DeclareMathOperator{\tp}{tp}
\def\Ind#1#2{#1\setbox0=\hbox{$#1x$}\kern\wd0\hbox to
  0pt{\hss$#1\mid$\hss}\lower.9\ht0\hbox to 0pt{\hss$#1\smile$\hss}\kern\wd0}
\newcommand{\ind}[1][]{\mathop{\mathpalette\Ind{}^{\!\!\!\!\rlap{$\scriptscriptstyle\textnormal{#1}$}\,\,\,\,}}}
\def\mind{\ind[M]}
\def\zind{\ind[Z]}
\def\tind{\ind[\th]}

\newcommand{\notind}[1][]{\mathrel{\not\mkern-7mu{\ind[#1]}}}

\newcommand{\et}{\ \&\ }
\newcommand{\vel}{\ \vee\ }

\newcommand{\K}{\mathbb{K}}

\newcommand{\F}{\mathbb{F}}
\newcommand{\Kinf}{\K_\infty}
\newcommand{\KC}{\K^C}
\newcommand{\KbC}{{\Kb}^C}

\newcommand{\Kb}{\overline{\K}}

\newcommand{\good}{\mathcal G}
\newcommand{\bad}{\mathcal B}
\newcommand{\Normal}{\mathcal N}
\newcommand{\B}{\mathfrak B}
\newcommand{\Afam}{\mathcal A}

\newcommand{\Cfam}{\mathcal C}
\newcommand{\Sfam}{\mathcal S}
\newcommand{\Part}{\mathcal P}
\newcommand{\elem}{\equiv}
\DeclareMathOperator{\Th}{Th}
\newcommand{\Lang}{\mathcal L}
\newcommand{\Langf}{\mathcal L_{OF}}
\newcommand{\Ltwo}{\Lang(U)}

\newcommand{\Td}{T^d}
\newcommand{\Cp}{\mathcal C^p}
\newcommand{\Cone}{\mathcal C^1}
\newcommand{\identity}{\mathrm 1}
\newcommand{\idmatrix}{\mathbf 1}

\DeclareMathOperator{\reg}{reg}
\DeclareMathOperator{\isol}{isol}

\newcommand{\aminimal}{a\hyph minimal\xspace}

\newcommand{\aminimality}{a\hyph minimality\xspace}
\newcommand{\iminimal}{$i$\hyph minimal\xspace}
\newcommand{\iminimality}{$i$\hyph minimality\xspace}

\newcommand{\Iminimality}{$I$\hyph minimality\xspace}
\newcommand{\ipminimal}{constructible\xspace}
\newcommand{\ipminimality}{constructibility\xspace}

\newcommand{\dminimal}{d\hyph minimal\xspace}

\newcommand{\pN}{pseudo\hyph$\Nat$\xspace}

\newcommand{\xnarrow}{$x$-narrow\xspace}

\newcommand{\DSF}{DSF\xspace}

\DeclareMathOperator{\RK}{rk}
\newcommand{\rkCB}{\RK^{CB}}
\newcommand*{\CBd}[2]{#1^{(#2)}}
\newcommand{\rkP}{\RK^{P}}
\DeclareMathOperator{\zrk}{Zrk}
\newcommand{\rkM}{\RK}
\DeclareMathOperator{\Zrk}{Zrk}
\DeclareMathOperator{\Zdim}{Zdim}

\newcommand{\dcompact}{d-compact\xspace}
\newcommand{\Fs}{\mathcal F_\sigma}
\newcommand{\Gd}{\mathcal G_\delta}
\newcommand{\ao}{a.o\mbox{.}\xspace}
\newcommand{\sdiff}{\mathbin{\Delta}}

\def\hyph{\nobreakdash-\hspace{0pt}\relax}
\newcommand{\Wlog}{W.l.o.g\mbox{.}\xspace}
\newcommand{\wloG}{w.l.o.g\mbox{.}\xspace}
\newcommand{\eg}{e.g\mbox{.}\xspace}
\newcommand{\ie}{i.e\mbox{.}\xspace}
\newcommand{\Ie}{I.e\mbox{.}\xspace}
\newcommand{\wrt}{w.r.t\mbox{.}\xspace}
\newcommand{\st}{s.t\mbox{.}\xspace}
\newcommand{\cf}{cf\mbox{.}\xspace}

\newcommand{\tfae}{t.f.a.e\mbox{.}\xspace}
\newcommand{\Tfae}{T.f.a.e\mbox{.}\xspace}
\newcommand{\aka}{a.k.a\mbox{.}\xspace}
\newcommand{\resp}{resp\mbox{.}\xspace}

\newtheorem{lemma}{Lemma}[section]
\newtheorem{thm}[lemma]{Theorem}
\newtheorem{corollary}[lemma]{Corollary}
\newtheorem{conjecture}[lemma]{Conjecture}
\newtheorem{proposition}[lemma]{Proposition}
\newtheorem{open problem}[lemma]{Open problem}

\newtheorem*{proviso}{Proviso}
\newtheorem*{fact*}{Fact}

\theoremstyle{remark}

\newtheorem{claim}{Claim}
\newtheorem*{claim*}{Claim}

\theoremstyle{definition}
\newtheorem{definizione}[lemma]{Definition}

\newtheorem{remark}[lemma]{Remark}
\newtheorem{final remark}[lemma]{Final remark}
\newtheorem{example}[lemma]{Example}
\newtheorem{examples}[lemma]{Examples}
\newtheorem{question}[lemma]{Question}

\begin{document}
\maketitle

\begin{abstract}
We study various notions of ``tameness'' for definably complete expansions of
ordered fields.
We mainly study structures with locally o-minimal open core, d-minimal
structures, and dense pairs of d-minimal structures.
\end{abstract}


{\small
\tableofcontents
}

\section{Introduction}
We will study various notion of ``tameness'', which generalize the notion of
o-minimality.
We will be interested only in definably complete structure: ``tame''  but not
definable complete structures (\eg, weakly o-minimal structures), while
important and interesting, are outside the scope of this article.

The first natural generalization of o-minimality
(for definably complete structures) is asking for o-minimality only around
each point of the structure (see Def.~\ref{def:lmin}).
Much of the theory of o-minimal structures can be generalized without
difficulty to locally o-minimal ones: see \S\ref{sec:lmin},
and from \S\ref{sec:i0min} to the end.

There is a dichotomy in further generalizing local o-minimality;
let $\K$ be a definably complete expansion of an ordered field:
\begin{enumerate}
\item
Either we ask that the open core of $\K$ is locally o-minimal, obtaining
\aminimal structures;
a stronger version is the requirement that the open core of $\K$ is o-minimal
(see \S\ref{sec:amin}).
\item
Or we ask that every definable subset of $\K'$ is a union of an open set and
finitely many discrete sets, for every $\K' \succeq \K$,
obtaining d-minimal structures (see from \S\ref{sec:dmin} to the end).
\end{enumerate}

Definably complete structures were explicitely defined and studied
in \cite{miller}.
The open core of $\K$ was defined already in \cite{MS99}, where they study the
case when $\K$ is an expansion of~$\Real$.
Structures with o-minimal open core are one of the main topics of \cite{DMS};
here, instead, they are only a side remark, because
we show that many of the results (and some of the techniques) of \cite{DMS}
can be generalized to \aminimal structures;
moreover, we answer some questions left open there.
One of the natural examples of 
structures with o-minimal open core is given by elementary
pairs of o-minimal structures $A \preceq B$, studied in \cite{vdd-dense};
we show that the main results of \cite{vdd-dense} can be generalized to
elementary pairs of \dminimal structures, with a very similar proof
(see \S\ref{sec:dense}).

D-minimal structures were the main theme of~\cite{miller05}, where he also
examines other notion of tameness, in the case when $\K$ expands~$\Real$:
many of the definitions and proofs of this article are either inspired
by \cite{miller05}, or a direct reference to it.

While o-minimal structures are geometric (that is, the algebraic closure $\acl$
satisfies the Exchange Principle), and therefore they rosy of \th-rank~1, no
such result is true for \dminimal structures:
more precisely, if $\K$ is a sufficiently saturated \dminimal non o-minimal
structure, then $\acl$ does \emph{not} satisfy the Exchange Principle, and
$\K$ is not rosy (Lemma~\ref{lem:asymmetry}).
However, we have a notion of dimension for \dminimal structures, given by teh
topology, (which for o-minimal structures coincides with the usual o-minimal dimension),
which can be usefully employed in the study of \dminimal structures,
and in particular of dense elementary pairs of such structures, in the same
way as the o-minimal dimension is used in o-minimal structures.

All structure considered will moreover be definably Baire (see~\cite{FS}):
this will not be explicit in the definitions, but will follow quite easily
from them.
Also, many of the proofs will rely on consequences of this Baire property
(much in the same way as in~\cite{miller05} many theorems relied on Baire's
category theorem).
Therefore, a preliminary study of definably complete and of definably
Baire structures is essential in order to understand tame structures, and it
will be carried out  Sections~\ref{sec:preliminary} and~\ref{sec:definably-complete}.

Some general results can be obtained under even weaker conditions: see
Sections~\ref{sec:i0min} and~\ref{sec:DSF}.

\section{Conventions, basic definitions, and notation}
Definable will always mean ``definable with parameters''.

$\Rbar = \pair{\Real, 0, 1,+, \cdot, <}$ is the ordered field of real numbers.
$\Ralg$ is the subset of $\Real$ given by the real algebraic numbers.

A linearly ordered structure $\pair{\K, <}$ is \intro{definably complete} if
every definable subset of $\K$ has a supremum in $\K \sqcup \set{\pm \infty}$.
\begin{proviso}
$\K$~will always be a definably complete structure expanding an ordered field.
\end{proviso}

$d: \K^n \times \K^n \to \K$ is the \intro{distance} function 
$d(x,y) := \abs(x - y)$.
For every $\x \in \K^n$ and $0 < r \in \K$, $B(\x,r)$ is the \intro{open ball}
of center $\x$ and radius~$r$, while $\clB(x,r)$ is the \intro{closed ball}.

Let $X \subseteq \K^n$.
$\cl X$, also denoted by $\cll(X)$, 
is the topological \intro{closure} of $X$ inside~$\K^n$, while
$\inter X$, also denoted by $\interior(X)$, is the \intro{interior} of~$X$;
$\fr X := \cl X \setminus X$ is the \intro{frontier} of~$X$;
$\bd(X) := \cl X \setminus \inter X$ is the \intro{boundary} of $X$.
$X$ is \intro{nowhere dense} if $\inter{\cl X}$ is empty.

$\K$ is \intro{definably Baire} (or simply ``Baire'' for short) if $\K$ is not
the union of a definable increasing family of nowhere dense subsets~\cite{FS}.

$X$~is an \intro{$\Fs$-set} if $X$ is definable and is the union of a
definable increasing  family of closed subsets of~$\K^n$, and is a 
$\Gd$-set if its complement is an $\Fs$-set.
$X$~is \intro{meager} if is the union of a
definable increasing  family of nowhere dense sets.
$X$~is \intro{almost open} (or \ao for short) if there exists a definable open
sets $U$ such that $X \sdiff U$ is meager~\cite{FS}, where $\sdiff$ is the
\intro{symmetric difference} of sets.

$X$~is \intro{constructible} if it is a finite Boolean
combination of open sets.
$X$~is \intro{locally closed} if for every $x \in X$ there exists a
neighbourhood $U \ni x$ such that $X \cap U$ is closed in~$U$.

\begin{fact*}
If $X$ is definable, then $X$ is constructible iff it is a finite Boolean
combination of \emph{definable} open sets~\cite{allouche96, DM}.  
$X$~is locally closed iff it is of the form $C \cap U$, for some closed set
$C$ and some open set~$U$.  
$X$~is constructible iff it is a finite union of locally closed sets.
\end{fact*}

\begin{definizione}
Let $X \subseteq \K^n$ be definable.
$X$ is \intro{\dcompact{}} if it is  closed and bounded.
$X$ is \intro{pseudo-finite} if it is \dcompact and discrete.
\end{definizione}

\begin{definizione}
The \intro{open core} of $\K$ is the reduct of $\K$ generated by all definable
open subsets of~$\K^n$, for every~$n \in \Nat$.
\end{definizione}

\begin{definizione}
$\K$ is \intro{locally o-minimal} if, for every definable function 
$f: \K \to \K$, the sign of $f$ is eventually constant.
\end{definizione}
See also \S\ref{sec:lmin} for more details on local o-minimality.

Given $d \leq n \in \Nat$, we denote by $\Pi^n_d: \K^n \to \K^d$ the projection
onto the first $d$ coordinates.
Given $X \subseteq \K^{n+m}$ and $\cv \in \K^n$, we define
$X_{\cv} := \set{\y \in \K^m: \pair{\cv, \y} \in X}$,
and $X_{[\cv]} := \set{\cv} \times X_{\cv} = X \cap \set{\cv} \times X^m$.

\begin{remark}
The open core of $\K$ includes all definable constructible sets and, more
generally, all $\Fs$-sets.
In fact, if $X$ is an $\Fs$-set, then $X$ is the projection of a closed
definable set: if $X = \bigcup_{t \in \K} X_t$, where $\Pa{X_t: t \in \K}$ is
a definable increasing family of closed subsets of~$\K^n$,
then $X = \Pi^{n+1}_n\cll(\bigcup_{t \in \K} X_t \times \set t)$.
\end{remark}

I will now discuss briefly the proviso that $\K$ expands a field.
This assumption is often convenient for notational purposes and to simplify
the statements of the theorems (compare \eg our definition of $\Fs$-sets with
the corresponding definition of $D_\Sigma$ sets in~\cite{DMS}); 
in those cases, a reader that is interested in definably complete structures 
that may not expand a field can easily modify definitions, proofs, 
and statements to his situation. 
However, sometimes the field assumption is used in an essential way
(\eg, in \S\ref{sec:enumerable} and \S\ref{subsec:dense}), and the reader
assumed above should be more careful.

\section[Locally o-minimal open core]{Structures with o-minimal and locally o-minimal open core}\label{sec:amin}

\begin{definizione}
Let $P$ be a property of definable sets.
We say that $P$ is \intro{definable} (for~$\K$), if for every definable family
$\Pa{X_y}_{y \in   A}$, the set $d_P(X) := \set{y \in A: P(X_y)}$ is
definable.
If $T$ is a theory, we say that $P$ is definable for $T$ if $P$ is definable
for every model of~$T$.
\end{definizione}

For instance, ``being closed'' and ``being pseudo-finite'' are definable
properties.
A type $p$ over $\K$ is definable iff the corresponding property ``$X \in p$''
is definable~\cite[\S11.b]{poizat85}.
We do not know if ``being constructible'' is definable (because a
constructible set $X$ is a finite union of locally closed sets; if we do not
have a bound on the number of locally closed sets $C_i$ such that $X =
\bigcup_i C_i$, we are not able to express the constructibility of $X$ in a
definable way).
Notice also that a property might be definable for $\K$ without being
definable for the theory of~$\K$: for instance, if $\K$ is \emph{not}
$\omega$-saturated, then ``being finite'' might be definable for~$\K$,
without being definable for some $\K' \succ \K$.

\begin{definizione}\label{def:lmin}
$K$ is \intro{locally o-minimal} if, for every definable function $f: \K \to
\K$ and every $x \in \K$, there exists $y > x$ such that either $f > 0$ on
$(x,y)$, or $f = 0$ on $(x,y)$, or $f < 0$ on $(x,y)$.
\end{definizione}

\begin{thm}\label{thm:a-minimal}
The following are equivalent:
\begin{enumerate}
\item\label{en:a-dcf} Every definable discrete closed subset of $\K$ is pseudo-finite;
\item Every definable discrete closed subset of $\K^n$ is pseudo-finite for
every $n \in \Nat$;
\item\label{en:a-df} Every definable discrete subset of $\K$ is pseudo-finite;
\item Every definable discrete subset of $\K^n$ is pseudo-finite for every $n \in \Nat$;
\item Every definable nowhere-dense subset of $\K$ is pseudo-finite;
\item\label{en:a-mf} Every definable meager subset of $\K$ is pseudo-finite;
\item\label{en:a-ndd} Every definable nowhere-dense subset of $\K$ is discrete;
\item The \textbf{open core} of $\K$ is \textbf{locally o-minimal};
\item\label{en:a-uf} The union of any definable increasing family of pseudo-finite subsets of $\K$ is pseudo-finite;
\item\label{en:a-ufn} The union of any definable increasing family of pseudo-finite subsets of $\K^n$ is pseudo-finite, for every $n \in \Nat$.
\end{enumerate}
Moreover, if any of the above equivalent conditions is satisfied, then every
$\Fs$ subset of $\K^n$ is constructible, $\K$ is Baire, and every meager
subset of $\K^n$ is nowhere dense.
\end{thm}

\begin{definizione}
We say that $\K$ is \intro{\aminimal{}} if it satisfies any  of the equivalent
conditions in Thm.~\ref{thm:a-minimal} (that is, if it has locally o-minimal
open core).
\end{definizione}

\begin{definizione}
Let $X \subseteq \K^{n + m}$.
Define $\Fin_n(X) := \set{y \in \K^n: X_Y \text{ is finite}}$.
\end{definizione}

``Being finite'' is not definable a definable property in general.
The following lemma characterises when it \emph{is} a definable property.

\begin{lemma}\label{lem:definable-finite}
The following are equivalent:
\begin{enumerate}
\item Every pseudo-finite set is finite;
\item For every $X \subseteq \K^2$ definable, $\Fin_1(X)$ is also definable;
\item 
``Being finite'' is a definable property.
\end{enumerate}
\end{lemma}
\begin{proof}
$(3 \Rightarrow 2)$ and $(1 \Rightarrow 3)$ are clear.
Assume (2). Let $Z \subset \K^n$ be pseudo-finite; we have to prove that $Z$
is finite; for simplicity, assume $n = 1$.
Let $r := \sup\set{r \in \K: Z \cap [-r, r] \text{ is finite}}$.
Clearly, $r = + \infty$, and therefore, since $Z$ is bounded, $Z$~is finite.
\end{proof}

\begin{corollary}\label{cor:b-minimal}
The following are equivalent:
\begin{enumerate}
\item 
$\K$ is \aminimal, and every pseudo-finite subset of $\K$ is finite;
\item 
$\K$ is \aminimal, and every pseudo-finite subset of $\K^n$ is finite, for
every $n \in \Nat$; 
\item 
$\K$ is \aminimal, and for every $X \subseteq \K^{2}$ definable,
$\Fin_1(X)$ is definable; 
\item 
$\K$ is \aminimal, and  ``being finite'' is a definable property for $\K$;
\item 
$\K$ has \textbf{o-minimal open core};
\item\label{en:b-min-closed-discrete} 
Every definable closed discrete subset of $\K$ is finite;
\item 
Every definable closed discrete subset of $\K^n$ is finite,
for every $n \in \Nat$; 
\item 
Every definable discrete subset of $\K$ is finite;
\item 
Every definable discrete subset of $\K^n$ is finite, for every $n \in \Nat$.
\end{enumerate}
\end{corollary}

The above corollary is a converse of \cite[Theorem~A]{DMS}.

\begin{definizione}
We say that $\K$ 
has \intro{o-minimal open core}
if it satisfies any  of the equivalent
conditions in Corollary~\ref{cor:b-minimal}.
\end{definizione}

Notice that if $\K$ expands~$\Real$, then ``being finite'' \emph{is} a
definable property.
Notice also that  ``Uniform Finiteness'' (UF), as defined
in~\cite{DMS} (also know as ``elimination of the quantifier
$\exists^\infty$'') is a stronger property than ``being finite is definable in
$\K$'', because the former says that ``being finite'' is definable
in the theory of~$\K$ (for instance, if $\K$ expands~$\Real$, then ``being
finite'' is definable for~$\K$, but $\K$ does not necessarily satisfy UF).

\begin{corollary}[\cite{DMS}]
If $\K$  satisfies UF, then $\K$ has o-minimal open core.
\end{corollary}
\begin{proof}
Condition~\ref{cor:b-minimal}-\ref{en:b-min-closed-discrete} follows easily from the hypothesis.
\end{proof}

\begin{remark}
If $\K$ is an \aminimal expansion of $\Real$, then it has o-minimal open core.
\end{remark}

\section{Preliminaries}
\label{sec:preliminary}
Let $X \subseteq \K^n$ be definable.
Let $\fr A := \cl A \setminus A$, and $\bd(A) := \cl A \setminus \inter A$.

\begin{lemma}
If $Y \subseteq \K^n$ is definable and definably connected (\eg, $n = 1$ and
$Y$ is an interval), then the following are equivalent:
\begin{enumerate}
\item $Y \cap \bd(X) = \emptyset$;
\item either $Y \subseteq \inter X$, or $Y \cap \cl X = \emptyset$.
\end{enumerate}
\end{lemma}

\begin{remark}\label{rem:boundary}
\begin{itemize}
\item $\fr X$ has empty interior.
\item $\bd(X) = \fr X \sqcup \fr(\K^n \setminus X)$.
\item $\bd(X)$ is closed.
\item $\bd(X \cup X') \subseteq \bd(X) \cup \bd(X')$.
\item If $X$ and $X'$ are nowhere dense, then $X \cup X'$ is also nowhere
      dense.
\item If $X$ is locally closed, then $\bd(X)$ is nowhere dense.
      Therefore, if $X$ is constructible, then $\bd(X)$ is nowhere dense.
\end{itemize}
\end{remark}

\begin{corollary}\label{cor:constructible-nowhere-dense}
If $X$ is constructible and $\inter X = \emptyset$, then $X$ is nowhere dense.
\end{corollary}


\begin{lemma}\label{lem:discrete-boundary}
$X$ is locally closed iff $\fr X$ is closed.
If $\bd(X)$ is discrete, then $X$ is locally closed.
\end{lemma}
\begin{proof}
For the first part, if $\fr X$ is closed, then $X$ is open in $\cl X$, and
hence locally closed.
Conversely, if $X = \cl X \cap U$ for some open set~$U$, then $\fr X = X
\setminus U$.

For the second part, if $\bd(X)$ is discrete, then,
since $\bd(X)$ is also closed, $\bd(X)$ has no accumulation points in~$\K^n$.
Since $\fr X \subseteq \bd(X)$, we have that $\fr X$ is also closed (in~$\K^n$).
\end{proof}

\begin{lemma}\label{lem:uniform-cont}
Let $f : [0,1] \to \K$ be definable and continuous.
Then, $f$ is uniformly continuous.
\end{lemma}
\begin{proof}
Assume not; then, there exists $\varepsilon > 0$ such that, for every $\delta
> 0$, the set
\[
X(\delta) := \set{(x,y) \in [0,1]^2: \abs{x -y} \leq \delta \ \& \ \abs{f(x) -
  f(y)} \geq \varepsilon}
\]
is non-empty.
Since $[0,1]^2$ is \dcompact, we have $X := \bigcap_{\delta > 0} X(\delta) \neq
\emptyset$. 
If $(x,y) \in X$, then $x = y$ and $f$ is not continuous at~$x$, absurd.
\end{proof}

\begin{definizione}
Let $P$ be a property of definable sets.
We say that $P$ is \intro{monotone} if $X \subseteq Y$ and $P(Y)$ imply $P(X)$.
We say that $P$ is \intro{additive} if $P(X)$ and $P(Y)$ imply $P(X \cup Y)$.
\end{definizione}

\begin{lemma}
Let $C \subset \K^n$ be \dcompact, $f: C \to \K^m$ be definable, and $P$ be a
property of definable sets.
Assume that $P$ is definable, monotone, and additive, and that, for every 
$c \in C$, there exists $U_c$ definable neighbourhood of~$c$, such that 
$P(f(U_c \cap C))$.
Then, $P(f(C))$.
\end{lemma}
\begin{proof}
\Wlog, $f$~is the inclusion function.
Proceed as in~\cite{FS}.
\end{proof}

For instance, we can apply the above lemma to the property 
``being nowhere dense''.

\subsection{Dimension}
\begin{definizione}
Let $X \subseteq \K^n$ be definable and non-empty.
The \intro{dimension} of $X$ is
\begin{multline*}
\dim X := \max \set{d \leq n: \text{ there exists a coordinate space } L
\text{ of dimension } d,\\ 
\text{s.t. } \Pi^n_L(X) \text{ has non-empty interior}},
\end{multline*}
where $\Pi^n_L$ is the projection from $\K^n$ onto~$L$.
By convention, we say that $\dim \emptyset = -1$.
The \intro{full dimension} of $X$ is the pair $\pair{d,k}$, where
$d = \dim X$
and $1 \leq k$ is the number of coordinate spaces $L$ of dimension~$d$,
\st $\Pi^n_L(X)$ has non-empty interior.
\end{definizione}
The set of full dimensions is ordered lexicographically, with the dimension
component more important.
Therefore, by induction on the full dimension we mean induction first on $d$ and then on~$k$.
Dimension and full dimension were already defined in~\cite{DMS}.

\begin{lemma}\label{lem:dim-fiber}
Let $X \subseteq \K^{n+m}$ be a definable, of dimension~$n$.
Let $A := \set{a \in \K^n: \dim(X_a) > 0}$.
If $X$ is an~$\Fs$, then $A$ is an $\Fs$.
If moreover $\K$ is Baire, then $\inter A = \emptyset$
(\ie, $A$~has dimension less than~$n$).
\end{lemma}

\begin{proof}
Same as~\cite[2.8(3) and 3.4]{DMS}.
Since
\[
A = \bigcup_{\rho \in \Pi(n,1)} \bigcup_{s > 0} 
\set{a \in \K^n: \rho(X_a) \text{ contains a closed  interval of length } s},
\]
and $X$ is an~$\Fs$, $A$~is also an~$\Fs$.
If, for contradiction, $A$~has non-empty interior, then, 
\wloG, forevery $a \in A$, $\Pi^m_1(X_a)$ contains an open interval.
Thus, by Kuratowski-Ulam's Theorem~\cite{FS}, 
$\Pi^{n + m}_{n + 1}(X)$ is non-meager, and thus has non-empty interior, 
contradicting $\dim X = n$.
\end{proof}

In the above lemma we cannot drop the assumption that $X$ is an~$\Fs$;
for instance, let $X$ be the set of pairs $\pair{x,y} \in \Real^2$ such that:
\[
x \in \Raz \et 0 < y < 1
\vel
x \notin \Raz \et 1 < y < 2,
\]
in the structure
$\Real(\Nat) := \pair{\Real, +, \cdot, \Nat}$.

\begin{lemma}\label{lem:dim-union}
Assume that $\K$ is Baire.
Let $X_1, X_2 \subseteq \K^n$ be definable.
If $X_1 $ and $X_2$ are both $\Fs$,
then $\fdim(X_1 \cup X_2) = \max\Pa{\fdim(X_1), \fdim(x_2)}$
and in particular $\dim(X_1 \cup X_2) = \max\Pa{\dim(X_1), \dim(X_2)}$.
\end{lemma}
\begin{proof}
Assume, for contradiction, that, for some $m \leq n$,
$\Pi^n_m(X_1 \cup X_2)$ has non-empty interior, while $\Pi^n_m(X_i)$ has empty
interior, for $i = 1, 2$.
However, since $\Pi^n_m(X_i)$ is an~$\Fs$, this means that $\Pi^n_m(X_i)$ is
meager, for $i = 1, 2$,
and therefore $\Pi^n_m(X_1 \cup X_2) = \Pi^n_m(X_1) \cup \Pi^n_m(X_2)$
is also meager, absurd.
\end{proof}

In the above lemma we cannot drop the assumptions that the $X_i$ are~$\Fs$:
for instance, let $X_1 = \Raz$ and $X_2 = \Real \setminus \Raz$
in the structure $\Real(\Nat)$.

\begin{example}\label{exa:dim-constructible}
It is not true that, if $X \subseteq \K^n$ is definable and constructible, 
then $\dim \cl X = \dim X$;
\cf Thm.~\ref{thm:imin}(\ref{en:i-dim-n}).
In fact, let $\K :=\Real(\Nat)$, and $X \subset \Real^2$ defined by:
\[
X := \set{\pair{x,y}: x = p/q \in \Raz \et 0 < p < q \in \Nat \et (p,q) = 1
\et y = 1/q}.
\]
Notice that $X$ is locally closed (and \emph{a fortiori} constructible),
$\cl X = X \cup \Pa{[0,1] \times \set 0}$,
and $\dim X = 0$, while $\dim \cl X = 1$.
\end{example}

\subsection{Pseudo-finite sets}\label{sec:pf}

Define $\delta(X) := \inf\set{d(x, x'): x, x' \in X \ \&\ x \neq x'}$.

\begin{remark}
If $X$ is discrete, then it is nowhere dense in $\K^n$.
\end{remark}
\begin{proof}
Clear.
\end{proof}

\begin{lemma}\label{lem:fin-subset}
If $X$ and $X'$ are pseudo-finite, then $X \times X'$ is also pseudo-finite.
Moreover, if $X$ is pseudo-finite, then every definable subset of $X$ is also pseudo-finite.
\end{lemma}
\begin{proof}
The first part is clear from the definition.

Let $Y \subseteq X$ be definable.
It suffices to prove that $Y$ is closed in~$\K^n$, to conclude that $Y$ is
pseudo-finite.
Let $x \in X \setminus Y$.
Since $X$ is discrete, $\set x$ is open in~$X$, and therefore $Y$ is closed in~$X$; since $X$ is closed in~$\K^n$, $Y$ is also closed in~$\K^n$.
\end{proof}

\begin{lemma}
The following are equivalent:
\begin{enumerate}
\item $X$ is pseudo-finite;
\item $X$ is bounded and has no accumulation points in~$\K^n$;
\item $X$ is bounded and $\delta(X) > 0$.
\end{enumerate}
\end{lemma}
\begin{proof}
($1 \Leftrightarrow 2$) follows from the definition of pseudo-finite.

($3 \Rightarrow 2$) is clear.

($1 \Rightarrow 3$).
Assume that $X$ is pseudo-finite.
We want to prove that $\delta(X) > 0$.
Let $Y := X \times X$, and $\Delta(Y)$ be its diagonal.
Consider the map $d: Y \setminus \Delta(Y) \to \K$,
mapping $(x, x')$ to $d(x, x')$.
Note that $Y \setminus \Delta(Y)$ is pseudo-finite, and that $\delta(X) =
\inf_{Y   \setminus \Delta(Y)} d(y)$.
Thus, $d$~attains a minimum on $Y \setminus \Delta(Y)$, and therefore
$\delta(X) > 0$.
\end{proof}

\begin{lemma}
$X \subseteq \K^n$ is discrete and closed iff, for every $r > 0$,
$X \cap B(0,r)$ is pseudo-finite.
\end{lemma}
\begin{proof}
($\Rightarrow$) is clear, because if $X$ is discrete and closed, then $X \cap
B(0,r)$ is discrete, closed and bounded.
($\Leftarrow$) follows from the fact that $X$ has no accumulation points
in~$\K^n$.
\end{proof}

\begin{remark}
 If $\K$ defines an unbounded discrete subset, then it defines an unbounded
 discrete \emph{closed} subset.
\end{remark}
The above remark answers a question in~\cite[\S5]{miller05}.
\begin{proof}
Let $D \subset \K$ be discrete and not closed.
\Wlog, we can assume that $D$ is unbounded (if $a$ is an accumulation point
for~$D$, then $\infty$ is an accumulation point for $1/(D-a)$).
For every $r > 0$, let
\[
D(r) := \set{x \in D: D \cap B(x, r) = \set x},
\]
the set of points in $D$ at distance at least $r$ from the other points of~$D$.
Each $D(r)$ is discrete and closed.
If $D(r)$ is unbounded for some~$r$, we are done.
Otherwise, each $D(r)$ pseudo-finite; let $z(r) := \max\Pa{D(r)}$,a nd $Z :=
\set{z(r) : r > 0}$.
Since $D$ is unbounded, $Z$~is also unbounded.
\begin{claim}
$Z$~is closed and discrete.
\end{claim}
Otherwise, $Z$~would have an accumulation point~$a$.
For every $r > 0$, let $Z(r) := \set{z(r'): r' \leq r}$.
Notice that $Z(r)$ is bounded and $\delta\Pa{Z(r)} \geq r$;
thus, $Z(r)$ is pseudo-finite.
Moreover, since $z(r)$ is an increasing function of~$r$, there exists $r_0 >
0$ such that $z(r) > a + 1$ for every $r > r_0$.
Hence, $a$~cannot be an accumulation point of~$Z$, absurd.
\end{proof}

\begin{lemma}\label{lem:discrete-countable}
If $X$ is discrete, then it is the union of a definable increasing family of
pseudo-finite sets.
In particular, $X$~is an~$\Fs$.
\end{lemma}
\begin{proof}
After a change of coordinates, we can assume that $X$ is bounded.
For every $r > 0$, define $X(r) := \bigl\{x \in  X: X \cap B(x,r) = \{x\}\bigr\}$.
Since $X$ is discrete, $X = \bigcup_r X(r)$.
Therefore, it suffices to prove that, for each~$r$, $X(r)$ is pseudo-finite.
Fix $r > 0$.
It is clear that $\delta\Pa{X(r)} \geq r > 0$, and therefore $X(r)$ is pseudo-finite.
\end{proof}

\begin{lemma}
Let $X$ be pseudo-finite.
If $f: X \to \K^m$ is definable (not necessarily continuous),
then $f(X)$ is also pseudo-finite.
\end{lemma}
\begin{proof}
Given $f: X \to \K^m$, we want to prove that $f(X)$ is \dcompact and discrete.
By a change of coordinates in~$\K^m$, we can always assume that $f(X)$ is
bounded.
Assume, for contradiction, that $f(X)$ is not pseudo-finite.
Let $y \in \K^m$ be an accumulation point for~$f(X)$.
For every $r > 0$, let $U(r) := f^{-1}\Pa{B(y,r) \setminus \set y}$; note that each $U(r)$ is non-empty.
By Lemma~\ref{lem:fin-subset},
each $U(r)$ is closed in~$X$.
Since $X$ is \dcompact, $\bigcap_r U(r) \neq \emptyset$, which is absurd.
\end{proof}

\begin{corollary}\label{cor:product}
Let $X \subseteq \K^n$ and $X' \subseteq \K^{n'}$ be definable.
Then, $X \times X'$ is pseudo-finite iff $X$ and $X'$ are pseudo-finite.
\end{corollary}

\begin{corollary}
$X$ is pseudo-finite iff every projection of $X$ on the coordinate axes is pseudo-finite.
\end{corollary}

\begin{definizione}
A \intro{pseudo-finite family} of sets is a definable family 
$\Pa{X_a: a \in  A}$, such 
that $A$ is pseudo-finite.
\end{definizione}

\begin{lemma}\label{lem:finite-union}
Let $P$ be an additive definable property.
Let $\Pa{X_y: y \in A}$ be a pseudo-finite family,
such that, for every $y \in A$, $P(X_y)$.
Then, $P(\bigcup_{y \in A} X_y)$.
\end{lemma}
\begin{proof}
For simplicity, we assume that $A \subseteq \K$.
Let
\[
B := \set {y \in A: P(\bigcup_{\substack {z \in A,\\ z \leq y}} A_z)}.
\]
Since $P$ is definable, $B$ is a definable subset of~$A$.
Hence, $B$~is pseudo-finite, and therefore it has a maximum~$b$.
It is now easy to see that $b$ is also the maximum of~$A$.
\end{proof}

Since ``being closed'' is an additive definable property, 
we see that the union of a
pseudo-finite family of closed sets is closed,
and the intersection of a pseudo-finite family of open sets is open.
Similarly, the union of a pseudo-finite family of pseudo-finite sets is
pseudo-finite, and the union of a pseudo-finite family of nowhere dense sets
is nowhere dense.

\begin{conjecture}[Pigeon Hole Principle]
Let $X \subseteq \K^n$ be pseudo-finite and $f: X \to X$ be definable.
If $f$ is injective, then it is surjective.
\end{conjecture}


\subsection{Bad and locally closed sets}

\begin{lemma}\label{lem:bad-1}
Let $d \leq n$, $A \subseteq \K^{n}$ be definable, $\pi := \Pi^{n}_d$, and
\[
Z := Z(A) := \set{a \in A: \exists U \text{ neighbourhood of } A:
\pi(A \cap U) \text{ is nowhere dense}}.
\]
Then, $Z$ is definable and open subset in~$A$, and $\pi(Z)$ is meager.
\end{lemma}
\begin{proof}
The fact that $Z$ is definable and open in $A$ is trivial.
Let $Z' := Z(\cl A)$.
\begin{claim}
$Z = Z' \cap A$.
\end{claim}
Let $a \in Z$.
Then, $\pi(A \cap U)$ is nowhere dense, for some $U$ neighbourhood of~$a$.
Moreover, $\pi(\cl A \cap U) \subseteq \cl{\pi(A \cap U)}$, because $U$ is
open.
Hence, $\pi(\cl A \cap U)$ is nowhere dense, and therefore $a \in Z'$.

Therefore, it suffices to prove the lemma in the case when $A$ is closed.
Since $Z$ is closed in~$A$, $Z$ is locally closed.
Moreover, for every $a \in Z$ there exists $U$ neighbourhood of~$a$,
such that $\pi(U \cap Z)$ is nowhere dense;
thus, by~\cite[Cor.~3.7]{FS}, $\pi(Z)$~is meager.
\end{proof}

\subsubsection{Bad set}

\begin{definizione}
Let $A \subseteq \K^{n + m}$.
The set of ``bad points'' for $A$ is 
\[
\B_n(A) := \set{x \in \K^n: \cll(A)_x \setminus \cll(A_x) \neq \emptyset}.
\]
\end{definizione}
Notice that $\B_n(A) = \set{x \in \K^n: \cll(A)_x \neq \cll(A_x)}$.

In the following, it will often be necessary to prove that $\B_n(A)$, the set
of bad points of~$A$, is ``small'' (in some suitable sense).

\begin{remark}\label{rem:bad-3}
Assume that $A \subseteq C \subseteq \cl A \subseteq \K^{n+m}$.
Then, $\B_n(A) \supseteq \B_n(C)$.
\end{remark}
\begin{proof}
\[
\cll(A)_x \setminus \cll(A_x) = \cll(C)_x \setminus \cll(A_x)
\supseteq \cll(C)_x \setminus \cll(C_x). 
\qedhere
\]
\end{proof}

\begin{lemma}\label{lem:bad-2}
If $A$ is an $\Fs$, then $\B_n(A)$ is 
the projection of a $\Gd$ set.
If $A$ is open, then $\B_n(A)$ is a meager~$\Fs$.
\end{lemma}
\begin{proof}




Let
\[
F(A) := \set{ \pair{x, y, r, y'} \in \K^n \times \K^m \times \K \times \K^m: 
r > 0 \et \pair{x, y'} \in A \et \abs{y - y'} < r}.
\]
\begin{claim}
\[
\B_n(A) = \pi\Pa{\cl A \cap \pi'(\set{r > 0} \setminus \pi''(F(A)))},
\]
where $\pi'$ and $\pi''$ are suitable projections.
\end{claim}
If $A$ is an $\Fs$, then $F(A)$ is also an~$\Fs$, and therefore $\B_n(A)$ is
the projection of a $\Gd$ set.

If $A$ is open, then $F(A)$ is also open, and therefore $\B_n(A)$ is an~$\Fs$.

For every $r > 0$, define 
$C(r) := \set{\pair{x,y} \in \cl A: \abs{x,y} \leq 1/r \et  d(y, A_x) \geq r}$.
Notice that $\B_n(A) = \bigcup_{r > 0} \pi(C(r))$.
Assume that $A$ is open.
Then, each $C(r)$ is \dcompact.
Hence, to prove that $\B_n(A)$ is meager, it suffices to prove that, for every
$r > 0$, $\pi(C(r))$ has empty interior.
\Wlog, $\K$~is Baire.
Assume, for contradiction, that $\pi(C(r_0))$ contains a non-empty open
box~$W$, for some $r_0 > 0$. To simplify the notation, assume that $m = 1$.
Define $f: W \to \K$, $f(x) := \min\Pa{C(r_0)_x}$.
By~\cite{FS}, there exists a non-empty open box $W' \subseteq W$, such that
$f \rest {W'}$ is continuous; \wloG, $W = W'$.
Fix $x_0 \in W$, call $y_0 := f(x_0)$, and let $V_{x_0}$ be an open box
around~$x_0$ contained in~$W$, and such that, for every $x \in V_{x_0}$,
$d(f(x), y_0) < r_0/4$. 
Since $\pair{x_0, y_0} \in \cl U$, there exists $\pair{x, y'} \in U$,
such that $x \in V_{x_0}$, and $d(y_0, y') < r_0/4$.
Let $y := f(x)$.
Since $\pair{x, y} \in C(r)$, $d(y, U_x) \geq r$; in particular, 
$d(y, y') \geq r$.
However, this contradicts $d(y, y_0) < r_0/4$ and $d(y_0, y') < r_0/4$.
\end{proof}

\begin{lemma}
$\B_n(A \cup B) \subseteq \B_n(A) \cup \B_n(B)$.
\end{lemma}

\subsubsection{Locally closed sets}
\begin{definizione}
Let $X \subseteq \K^n$.
Define $\lc(X) := \set{x \in X : X \text{ is locally closed at } x}$, (that is,
$x \in \lc(X)$ iff there exists an open ball $U$ of center~$x$, such that $X
\cap B = \cl X \cap B$), and $\nlc X := X \setminus \lc(X)$.

Define $\inlc X 0 := X$, and, for each $k \in \Nat$,
$\inlc X {k + 1} := \nlc{\inlc X {k}}$.
\end{definizione}

Notice that $\lc(X)$ is locally closed, and therefore constructible.
Notice also that, if $X$ is definable, then also $\nlc X$ and each $\inlc X {k}$
are definable.
Therefore, if $X$ is an~$\Fs$, then $\nlc X$ is also an~$\Fs$.

\begin{proposition}
$\nlc A  = A \cap \partial(\partial A)$.
$A$ is the union of $m$ locally closed sets if and only if $\inlc A {m + 1}$
is empty.
\end{proposition}
\begin{proof}
See~\cite{allouche96}, where $\partial A$ is denoted by $\check A$,
and $\nlc A$ by either $\mathcal B(A)$ or~$H(A)$.
\end{proof}

\begin{remark}
If $U\subseteq \K^n$ is open, then $\lc(A) \cap U = \lc(A \cap U)$, and
$\nlc{A \cap U} = \nlc A \cap U$.
\end{remark}

\begin{remark}
Let $U \subseteq \K^n$ be open, and $A \subseteq U$ be closed in~$U$.
Let $E := \cl A$.
Then, $E \cap \cl U = E$.
Therefore, $E \cap \cl U = \cl{E \cap U}$.
\end{remark}
\begin{proof}
The $\subseteq$ inclusion is obvious.
The opposite inclusion follows immediately from $E = \cl{E \cap U}$.
\end{proof}

\begin{thm}\label{thm:lc-fiber}
Let $A \subseteq \K^n$ be locally definable closed, and $d \leq n$.
Let $U \subseteq \K^n$ be open, such that
$A = \cl A \cap U$.
Then, for every $x \in \K^d$,
\[
A_x = \cll(A_x) \cap U_x,
\]
and in particular $A_x$ is locally closed.
Moreover, $\B_d(A) \subseteq \B_d(U)$, and therefore $\B_d(A)$ is meager.
\end{thm}
\begin{proof}
$A_x \subseteq \cll(A_x) \cap U_x$ is obvious.
For the opposite inclusion, 
\[
\cll(A_x) \cap U_x \subseteq (\cl A)_x \cap U_x =
(\cl A \cap U)_x = A_x.
\]
Assume, for contradiction, that $x \in \B_d(A) \setminus \B_d(U)$.
Let $E := \cl A$; notice that $A = E \cap U$.
Since $x \notin \B_d(U)$, $\cll(U_x) = (\cl U)_x$.
By the above Remark, applied to~$E_x \cap U_x$ and to $E \cap U$,
\[
\cll(A_x) = \cll(E_x  \cap U_x) = E_x \cap \cll(U_x) =
E_x \cap (\cl U)_x = (E \cap \cl U)_x = E_x = (\cl A)_x,
\]
contradicting $x \notin \B_d(A)$.

By Lemma~\ref{lem:bad-2}, $\B_d(U)$ is meager, and we are done.
\end{proof}

\begin{corollary}
Let $A \subseteq \K^n$ be definable and constructible, and $d \leq n$.
Then, $\B_d(A)$ is meager.
\end{corollary}


\section{Proof of Thm.~\ref{thm:a-minimal}}\label{sec:proof-a}
\mbox{}\indent($\ref{en:a-df} \Rightarrow \ref{en:a-dcf}$) is obvious.

($\ref{en:a-dcf} \Rightarrow \ref{en:a-uf}$).
Let $\Pa{X(r)}_{r \in \K}$ be a definable increasing family of pseudo-finite
subsets of~$\K$, and $X := \bigcup_r X(r)$.
Assume, for contradiction, that $X$ is not pseudo-finite.
\Wlog, we can assume that $X \subseteq (0,1)$, and that $0$ is an accumulation
point of~$X$.
For every $r \in \K$, let $z(r) := \min \Pa{X(r)}$,
$Z := \set{z(r): r \in \K}$,
$Y := \set{1/z(r): r \in  \K}$.
\begin{claim}
The only accumulation point of $Z$ in $\K$ is $0$.
\end{claim}
In fact, suppose, for contradiction, that $c > 0$ is an accumulation point of~$Z$.
Since $0$ is an accumulation point for $X$; there exists $r_0 \in \K$ such
that $X(r_0) \cap (0, c/2) \neq \emptyset$.
Thus, $z(r_0) < c/2$, and, since $z(x)$ is a decreasing function, $z(r) < c/2$
for every $r \geq r_0$.
Let $Z(r_0) := \set{z(r) : r < r_0}$, and $Z' := Z \setminus Z(r_0)$.
Since $Z(r_0) \subseteq X(r_0)$, $Z(r_0)$ is pseudo-finite.
Moreover, since $Z' \subseteq (0, c/2)$, $c$~is not
an accumulation point of~$Z'$, and thus it is not an accumulation point
of~$Z$, absurd. 

By the claim, $Y$~is discrete and closed,
and therefore, by hypothesis, it is pseudo-finite.
Hence, $Z$ is also pseudo-finite; therefore, $0$ cannot be an accumulation
point for~$Z$, absurd. 

($\ref{en:a-uf} \Leftrightarrow \ref{en:a-ufn}$) is clear form Corollary~\ref{cor:product}.

($\ref{en:a-uf} \Rightarrow \ref{en:a-df}$).
Follows from Lemma~\ref{lem:discrete-countable}.

Hence, we have the equivalence
($\ref{en:a-dcf} \Leftrightarrow \ref{en:a-df} \Leftrightarrow \ref{en:a-uf}
\Leftrightarrow \ref{en:a-ufn}$).

We now prove that ($\ref{en:a-dcf} \Rightarrow \ref{en:a-mf}$).
Let $X \subseteq \K$ be meager; thus, $X$ is the union of a definable
increasing family $\Pa{X(r)}_{r \in K}$ of nowhere dense subsets of~$\K$.
We want to prove that $X$ is pseudo-finite.
\Wlog, we can assume that each $X(r)$ is \dcompact.
By condition~\ref{en:a-uf}, it suffices to prove that each $X(r)$ is
pseudo-finite.
Thus, we fix $r \in \K$, and prove that $Y := X(r)$ is pseudo-finite,
knowing that it is \dcompact and has empty interior.
Assume, for contradiction, that $Y$ has an accumulation point in~$\K$.
\Wlog, we can assume that $Y \subset (0,1)$, and that $0$ is an accumulation
point of~$Y$.
Since $\K$ is definably complete and $Y$ is closed, $(0,1) \setminus Y$ is a
union of disjoint open intervals; let $D$ be the set of centres of those
intervals, that is:
\[
D := \set{z \in (0,1): \exists r > 0,\ z - r \in Y,\ z + r \in Y,\ (z-r, z + r)
  \cap Y = \emptyset}.
\]
Note that $D$ is discrete.
By condition~\ref{en:a-df}, $D$ is pseudo-finite; let $a := \min(D)$, and
$r > 0$ such that $a - r, a + r \in Y$ and $(a - r, a + r) \cap Y =
\emptyset$.
Thus, $(0, a-r) \subseteq Y$.
Since $Y$ has empty interior, $a - r = 0$.
However, this contradicts the fact that $0$ is an accumulation point for~$Y$.

($\ref{en:a-mf} \Rightarrow 5 \Rightarrow \ref{en:a-dcf}$) is clear.

($7 \Rightarrow \ref{en:a-dcf}$).
Let $X \subseteq \K$ be discrete and close.
Assume, for contradiction, that $X$ is not bounded;
let $Y := \set{1/x: 0 \neq x \in X} \cup \set 0$.
Since $Y$ is nowhere dense, $Y$ is also discrete, contradicting the fact that
$X$ is unbounded.
Hence, we have the equivalence 
($\ref{en:a-dcf} \Leftrightarrow \ref{en:a-df}
\Leftrightarrow 5 \Leftrightarrow \ref{en:a-mf} \Leftrightarrow 7 
\Leftrightarrow \ref{en:a-uf} \Leftrightarrow \ref{en:a-ufn}$).

Now, we prove that condition~\ref{en:a-mf} implies that every $\Fs$ subset of
$\K^n$ is constructible.
This in turns imply that the open core of $\K$ is locally o-minimal, since
every set definable in the open core will be constructible, and constructible
sets with empty interior are meager, and thus pseudo-finite.

It is clear that condition~\ref{en:a-mf} implies that $\K$ is Baire.
Let $X \subseteq \K^n$ be an~$\Fs$.

If $n = 1$, then $X = \inter X \cup (X \setminus \inter X)$, and therefore $X$
is the union of an open set and a meager set; thus, $X$ is constructible.

If $n > 1$, we proceed, as in the proof of \cite[3.4]{DMS}, by induction on
$n$ and $(d,k) = \fdim X$ (the full dimension of~$F$).
Note that if $d = 0$, then $X$ is pseudo-finite, and hence constructible.

If $d = n$, then $X = \inter X \cup (X \setminus \inter X)$.
$\inter X$ is open, and hence constructible, while 
$X \setminus \inter X$ has dimension less than $n$, and therefore it is also
constructible; thus, $X$ is constructible.

If $0 < d < n$, \wloG we can assume that $\pi(X)$ has non-empty interior,
where $\pi := \Pi^n_d$.
Let $A := \set{a \in \K^d: \dim(X_a) > 0 }$.
By Lemma~\ref{lem:dim-fiber}, $A$~is and $\Fs$ of dimension~$< d$;
therefore, by inductive hypothesis, $A$~is constructible.
Let $Y := \pi^{-1}(A) \cap X$.
Since $\fdim(Y) < \fdim(X)$, by induction, $Y$~is constructible.
Thus, it suffices to prove that $X \setminus Y$ is constructible, and hence we
can reduce to the case when $Y = \emptyset$.

Thus, we can reduce to the case when $d < n$, $B := \pi(X)$ has non-empty
interior, and, for every $y \in B$, $X_y$~is pseudo-finite. 

Here the proof will diverge from~\cite{DMS}, because we do not have the
hypothesis~UF, and therefore we cannot proceed by induction on a uniform
bound $N$ on the cardinality of the fibers~$X_y$.

We claim that $\fdim(\nlc X) < \fdim(X)$.
If the claim is true, then, by inductive hypothesis,
$\nlc X$~is constructible, and therefore $X$ is constructible.

If, for contradiction, $\fdim(\nlc X) = \fdim(X)$,
then $\pi(\nlc X)$ contains a non-empty open ball~$B'$;
by shrinking~$X$, we can assume than $B = B'$.

By hypothesis, there exists a definable increasing family
$\Pa{X(t)}_{t \in  K}$ of \dcompact subsets of~$\K^n$,
such that $X = \bigcup_t X(t)$.
Let $Y (t) := \pi \Pa{X(t)}$; note that each $Y(t)$ is \dcompact,
and $B = \bigcup_t Y(t)$.
Since $\K$ is Baire, by~\cite[Prop.~2.7]{FS}, $\K^d$ is also Baire, and
therefore there exists $t_0 \in K$ such that $Y(t_0)$ has non-empty interior.
Let $B' \subseteq Y(t_0)$ be a non-empty open box; by shrinking~$X$, we can
assume that $B = B'$, and, by re-defining the family $\Pa{Y(t)}_{t \in \K}$,
that $B = \pi \Pa{Y(t)}$ for every~$t$.

For every $y \in B$, define $X_{[y]} := \set y \times X_y \subseteq X$.
The fact that $X_{[y]}$ is pseudo-finite easily implies the following:
\begin{remark}
For every $y \in B$ there exists $t \in \K$ such that $X_{[y]} \subseteq X(t)$.
\end{remark}

\begin{lemma}\label{lem:amin-constructible-dense}
There exists $B' \subseteq B$ non-empty open box, and $t_0 \in \K$,
such that $X \cap (B' \times \K^{n-d}) \subseteq X(t_0)$.
\end{lemma}
\begin{proof}
For every $t \in \K$, define $Z(t) := X \setminus X(t)$,
and $Y(t) := B \setminus \pi\Pa{Z(t)}$.
Note that each $Z(t)$ is an $\Fs$, and therefore $Y(t)$ is a~$\Gd$.
By induction on~$n$, $Y(t)$ is constructible.
Moreover, by the remark, $B = \bigcup_t Y(t)$.
Since $B$ is not meager, there exists $t_0 \in \K$ such that $\cl{Y(t_0)}$ has
non-empty interior.
However, since $Y(t_0)$ is constructible, this implies that $Y(t_0)$ itself
has non-empty interior, and therefore contains an open box~$B'$.
\end{proof}

Hence, $X \cap (B' \times \K^{n-d}) = X(t_0) \cap (B' \times K^{n-d})$, and
therefore $X \cap (B' \times  \K^{n-d}) \subseteq \lc(X)$.
Since the same reasoning can be made for every open box $\tilde B \subseteq
B$ instead of~$B$, we conclude that $\pi(\nlc X)$ has empty interior,
absurd.

To conclude, it remains to prove $(\ref{en:a-dcf} \Rightarrow 2)$.
Let $X \subseteq \K^n$ be discrete and definable.
Since we have seen that $\K$ is Baire, this implies that every projection
$\mu(X)$ of $X$ on a coordinate axis has empty interior; however, $\mu(X)$ is
an $\Fs$, and thus $\mu(X)$ is pseudo-finite for every coordinate axis.
Thus, $X$~is pseudo-finite.

The fact that every nowhere dense definable subset of $\K^n$ is meager follows
from the fact that, if $X$ is nowhere dense, then it is contained in a nowhere
dense $\Fs$ set~$Y$; we have seen that $Y$ is constructible; hence, $Y$ is
nowhere dense.
\hfill \qedsymbol

\subsection[Further results]{Further results on a-minimal structures}

From the above proof, we can extract the following results.

\begin{proviso}
For the remainder of this subsection, $\K$ is \aminimal.
\end{proviso}

\begin{lemma}\label{lem:amin-dim-lc}
Let 
$C \subseteq \K^n$ be definable and constructible
(or, equivalently, an~$\Fs$) set of dimension~$d$.
Then, $\dim(\nlc C) < d$.
Therefore, $C$~is the union of $d + 1$ locally closed definable sets.
\end{lemma}

As a corollary, we have that, for an \aminimal structure,
``being constructible'' is a definable property.
In fact (for $\K$ \aminimal)
$X \subseteq \K^n$ is constructible iff $\inlc X {n + 1}$ is empty.
Thus, Lemma~\ref{lem:finite-union} implies the following:

\begin{corollary}\label{cor:finite-constructible}
Let 
$A$~be pseudo-finite, 
and $(X_a)_{a \in A}$ be a definable family of constructible subsets
of~$\K^n$.
Then, $\bigcup_{a \in A} X_a$ is also constructible.
\end{corollary}

Moreover, the dimension is well-behaved for \emph{constructible} sets
definable in an \aminimal structure.
We have already seen that $\dim (C \cup C') = \max(\dim C, \dim C')$ for
constructible definable sets (Lemma~\ref{lem:dim-union}).
\begin{lemma}\label{lem:amin-dim-1}
Let $C$ be a definable constructible 
subset of~$\K^n$, where $\K$ is \aminimal.
$\dim C \leq 0$ iff $C$ is pseudo-finite.
\end{lemma}
Moreover, in Lemma~\ref{lem:loc-min-dim},
we will show that, for $C$ definable, 
constructible, and non-empty subset of an \aminimal structure,
we have $\dim(\fr C) < \dim C$ if $C$ is non-empty.

\begin{lemma}\label{lem:amin-dim-2}
Let $C(t)$ be a definable increasing family of subsets of~$\K^n$,
such that each $C(t)$ is constructible (or, equivalently, an~$\Fs$),
and $\dim(C(t)) \leq d$.
Let $C := \bigcup_t C(t)$.
Then, $\dim C \leq d$.
\end{lemma}
\begin{proof}
Assume, for contradiction, that $\pi(C)$ contains a non-empty open set~$U$,
where $\pi := \Pi^n_d$.
For every $t$, let $D(t) := \pi(C(t))$.
By hypothesis, each $D(t)$ has empty interior and is an~$\Fs$, and therefore
it is meager.
Thus, by \aminimality, each $D(t)$ is nowhere dense.
However, $\bigcup_t D(t) = U$, and thus $U$ is meager, which is absurd.
\end{proof}
For a similar result, see also \ref{thm:imin}(\ref{en:i-increasing}).

We can prove a different version of Corollary~\ref{cor:finite-constructible},
albeit with a longer proof.
\begin{lemma}\label{lem:amin-U-Fs}
Let 
$\Pa{A(t)}_{t \in \K}$ be a definable  increasing
family of constructible subsets of~$\K^n$.
Then, $X := \bigcup_{t \in \K} A_t$ is also constructible.
\end{lemma}
\begin{proof}
We will proceed by induction on $(d,k) := \fdim X$.
If $d = 0$, then each $A(t)$ is pseudo-finite (by Lemma~\ref{lem:amin-dim-1});
therefore, $X$ is pseudo-finite, and hence constructible.
If $d = n$, then $X := \inter X \sqcup (X \setminus \inter X)$.
Define $B(t) := A(t) \setminus \inter X$;
notice that $X \setminus \inter X = \bigcup_t B(t)$, and $\Pa{B(t)}_{t \in
  \K}$ is a definable increasing family of constructible sets.
Thus, by inductive hypothesis, $X \setminus \inter X$ is constructible,
and therefore $X$ is constructible.

It remains to treat the case $0 < d < n$.
\Wlog, $\pi(X)$ has non-empty interior, where $\pi := \Pi^n_d$.
By inductive hypothesis, $\pi(X)$ is constructible.
For each~$t \in \K$, define $C(t) := \set{y \in \K^d: \dim\Pa{X(t)_y} > 0}$.
Each $C(t)$ is definable in the open core of~$\K$, and thus constructible.

\begin{claim}\label{cl:amin-finite-fiber}
$C(t)$ has empty interior, for every~$t$.
\end{claim}
If, for some $t \in \K$,
$C(t)$ has non-empty interior, then, since $\dim(A(t)) \leq \dim X = d$, 
$\pi\Pa{A(t)}$ has empty interior, and thus it is nowhere dense.
Hence, $\pi(X)$ is meager, a contradiction.

Let $Y := \bigcup_t C(t)$.
By inductive hypothesis, $C$ is constructible; moreover,
by Claim~\ref{cl:amin-finite-fiber},
it is meager, and therefore $\dim Y < d$.
Let $X_1 := X \setminus (Y \times \K^{n-d})$ and
$X_2 := X \cap (Y \times \K^{n-d})$.
We have seen that $\fdim(X_2) < \fdim X$, and thus, by inductive hypothesis,
$X_2$ is constructible.
Thus, \wloG we can assume that $X = X_1$, that is,
$A(t)_y$ is pseudo-finite for every~$t \in \K$ and $y \in \K^d$.
By Lemma~\ref{lem:amin-dim-1}, this means that $X_y$ is pseudo-finite for
every $y \in \K^d$.

Notice that, for every $t \in \K$,
$X \setminus A(t) = \bigcup_s \Pa{A(s) \setminus A(t)}$,
and therefore, by induction on~$n$,
$\pi\Pa{X \setminus A(t)}$ is constructible.
Thus, reasoning as in Lemma~\ref{lem:amin-constructible-dense},
we cam prove the following Claim.
\begin{claim}\label{cl:amin-lc-dense}
For every non-empty open box $B$ there exists
$t \in \K$ and $B' \subseteq B$ non-empty open box,
such that $X \cap (B' \times \K^{n-d}) = A(t) \cap (B' \times \K^{n-d})$.
\end{claim}
Fix $B'$ as in the above claim, and let
$\tilde A(t) := A(t) \cap (B' \times \K^{n-d})$.
By Lemma~\ref{lem:amin-dim-lc}
$\dim\Pa{\tilde A(t) \setminus \lc(\tilde A(t))} <
\dim\Pa{\tilde A(t)} \leq d$.
Moreover, $\nlc X \cap (B' \times \K^{n-d}) =
\nlc{\tilde A(t)}$.
Therefore, by Claim~\ref{cl:amin-lc-dense},
$\pi(\nlc X)$ has empty interior.
Thus, by inductive hypothesis, $\nlc X$ is constructible,
and hence $X$ is constructible.
\end{proof}

\begin{lemma}\label{lem:amin-ao}
Let 
$C \subseteq \K^n$ be definable.
\Tfae:
\begin{enumerate}
\item $C$ is \ao;
\item $C \setminus \inter C$ is nowhere dense;
\item $\fr C$ is nowhere dense;
\item $\bd C$ is nowhere dense;
\end{enumerate}

\end{lemma}
\begin{proof}
$(4 \Rightarrow 3 \Rightarrow 1)$ and  $(4\Rightarrow 2 \Rightarrow 1)$ are
clear, and are true even without the \aminimality hypothesis.
For $(1 \Rightarrow 2)$, assume $C$ \ao.
Then, $C = U \Delta F$, for some definable meager set~$F$,
and some definable open set~$U$.
Since $\K$ is \aminimal, $F$~is nowhere dense; let $D := \cl F$.
\begin{claim}
$U \setminus D \subseteq \inter C$.
\end{claim}
In fact,
\[
\inter C = \interior (U \Delta F) \supseteq \interior (U \setminus F)
= U \setminus D.
\]

Thus,
\[
C \setminus \inter C \subseteq C \setminus (U \setminus D) =
(U \Delta F) \setminus (U \setminus D) \subseteq
\Pa{ U \setminus(U \setminus D)} \cup F \subseteq D \cup F = D,
\]
and $D$ is nowhere dense.

The above proof applied to $\K^n \setminus C$ shows that
$(1 \Rightarrow 3)$.

Finally, for $1 \Rightarrow 4$, observe that
$\bd C = \fr C \cup (C \setminus \inter C)$.
Hence, if $C$ is \ao, then, by (2) and (3),
$\fr C$ and $C \setminus \inter C$ are nowhere dense,
and thus $\bd C$ is nowhere dense.
\end{proof}

\begin{lemma}\label{lem:amin-U-meager}
Let $(A(t))_{t \in \K}$ be a definable increasing family of subsets of~$\K^n$,
and $C := \bigcup_t A(t)$.
If each $A(t)$ is meager,  then $C$ is meager (and thus nowhere dense).
If each $A(t)$ is \ao, then $C$ is \ao.
\end{lemma}
\begin{proof}
Since $\K$ is \aminimal, if each $A(t)$ is meager,
then each $A(t)$ is nowhere dense, and therefore $C$ is
meager.

If each $A(t)$ is \ao, 
let $U := \inter C$, and $A'(t) := A(t) \setminus V$.
Notice thar $A'(t) \subseteq A'(t) \setminus interior(A'(t))$;
thus, Lemma~\ref{lem:amin-ao}, each $A'(t)$ is nowhere dense.
Therefore, $C \setminus U = \bigcup_t A'(t)$ is meager,
and hence $C$ is \ao.
\end{proof}

\section{Locally o-minimal structures}\label{sec:lmin}

\begin{lemma}\label{lem:lmin}
The following are equivalent:
\begin{enumerate}
\item $\K$ is locally o-minimal;
\item for every $X \subset \K$ definable, and for every $x \in \K$, there
exists $y > x$ such that, either $(x ,y) \subseteq X$, or $(x,y) \subseteq \K
\setminus X$;
\item for every $X$ definable subset of~$\K$, either $X$ is pseudo-finite, or it has non-empty interior.
\end{enumerate}
\end{lemma}
See~\cite[2.11]{DMS} for other equivalent formulations of local o-minimality
(for definably complete structures, not necessarily expanding a field).
\begin{proof}
($1 \Rightarrow 2$).
Apply the definition of locally o-minimal to the characteristic function of~$X$.
($2 \Rightarrow 1$).
Let $f: \K \to \K$ be definable, and $x \in \K$.
Let $X := f^{-1}(0)$.
By hypothesis, there exists $y> x$ such that $(x,y)$ is a subset either of~$X$,
or of $\K \setminus X$.
In the first case, we are done; in the second, let $Y := f^{-1}\Pa{(-\infty,
  0)}$.
By decreasing $y$ if necessary, either $(x,y) \subseteq Y$,
or $(x,y) \subseteq \K \setminus (X \cap Y)$.

($2 \Rightarrow 3$).
Let $X \subset \K$ be definable with empty interior.
We must prove that $X$ is pseudo-finite.
\Wlog, $X$~is bounded.
Let $x \in \K$.
By condition~2, since $X$ has empty interior,
there exists $a < x < b$ such that both $(a,x)$ and $(x,b)$ are disjoint from~$X$.
Thus, $x$~cannot be an accumulation point for $X$, and therefore $X$ is
pseudo-finite.

($3 \Rightarrow 2$).
Let $X \subseteq \K$ be definable, $x \in \K$,
$Y := (x, +\infty) \cap X$, and $Z := \inter Y$.
Since $Z$ is open, $\bd(Z) := \cl Z \setminus \inter Z$ has empty interior,
and therefore it is pseudo-finite.
Thus, there exists $y > x$ such that $(x,y)$ is contained either in $Z$ or in
$\K \setminus Z$.
In the first case, we have the conclusion.
In the second case, we have that $W := X \cap (x,y)$ has empty interior.
Thus, $W$ is pseudo-finite, and therefore,
after possibly decreasing $y > x$, we have
$(x,y) \subseteq \K \setminus X$.
\end{proof}

\begin{proviso}
In the remainder of this section, $\K$ is a definably complete and locally
o-minimal expansion of a field.
\end{proviso}
Since locally o-minimal structures are \aminimal, Thm.~\ref{thm:a-minimal}
applies to them.
Moreover, as shown in Section~\ref{sec:proof-a}, every definable constructible
subset of $\K^n$ has a well-behaved dimension.
We will show later that every definable set is constructible.


\subsection{Examples}
Every o-minimal structure is locally o-minimal.
Since ``local o-minimality'' is a first-order property, an ultra-product of
locally o-minimal structures also locally o-minimal.
The above observation leads us to the following example.

\begin{example}
Let $M$ be a fixed o-minimal structure, in a language~$\Lang$.  Let $P$ be a
new unary predicate.  For every $n \in \Nat$, let $P_n := \set{1, 2, \dotsc,
  n} \subset M$.  Let $M_n := (M, P_n)$ be the $L(P)$-expansion of~$M$, where
$P$ is interpreted by~$P_n$.  Let $N := (M^*, P^*)$ be a non-principal
ultra-product of the~$M_n$.  Then, $N$ is locally o-minimal, but not o-minimal
(because $P^*$ is pseudo-finite, but not finite).
\end{example}

If $M$ is an expansion of~$\Real$, then the above structure can be
considered a restriction of an ultra-product of $(\Real, \Nat)$
(the expansion of $\Real$ with a predicate for the natural numbers). 
However, we can take for $M$ the counter-example of Hrushovski and Peterzil:
$M^*$ is an elementary extension of~$M$, and therefore also satisfies a
formula that cannot be true in any expansion of the reals, and the same holds for~$N$.

Since every locally o-minimal structure is also Baire (we are assuming that
everything is definably complete), then the above is also a non-trivial
example of definably complete and Baire structure.

Along the same lines, we can also consider the following example:
let $M$ be as before, and $M'$ be an elementary extension of~$M$, such that
$M$ is dense in $M'$ and different from~$M'$.
Define $P_n$ as before, let
$M_n := (M', M, P_n)$, and $N := (M'^*, M^*, P^*)$ be a non-principal
ultra-product of the~$M_n$.
Each $M_n$ has o-minimal open core,
hence it is \aminimal; thus, $N$ is \aminimal (and hence Baire), 
but does not have o-minimal open core.
Again, if $M$ is the counter-example of Hrushovski and Peterzil, we see that
$N$ is not the restriction of an elementary extension of an expansion
of~$\Real$. 

Regarding 
structures with o-minimal open core, \cite[\S8]{DMS} ask the following
question: 
Suppose that $\K$ is Archimedean and has o-minimal open core;
does $\Th(\K)$ have a model over $\Real$ that is unique up to isomorphism?
While we do not have an answer regarding the existence, it is easy to see that
uniqueness might fail.
In fact, consider the case when $\K$ is given by the dense pair  $(\Rbar,
\Ralg)$.
Let $S$ be an uncountable real-closed subfield of~$\Real$, which is not all
of~$\Real$.
Then, $(\Rbar, S)$ is elementarily equivalent to~$\K$, but it is not
isomorphic to~$\K$.
A more interesting question is: with the above hypothesis, does $\K$ have an
elementary extension expanding~$\Rbar$?
Again, uniqueness cannot be expected.
For instance, let $U$ be a real-closed subfield of~$\Real$, such that
$U \cap S = \Ralg$,  $\Ralg \neq U \neq \Real$, and $U$ and $S$ are free 
over~$\Ralg$;
then, both $(\Rbar, S)$ and $(\Rbar, \Ralg)$ are elementary
extensions of $(U, \Ralg)$ (see \cite[Corollary~2.7]{vdd-dense}).

\subsection{The monotonicity theorem}
Let $A \subseteq (a,b)$ be pseudo-finite.
For every $x \in A$, we denote by $s(x)$ the smallest element of $A$ strictly
greater than~$x$, if $x$ is not the maximum of~$A$, and $b$ otherwise.
Similarly, we denote by $p(x)$ the greatest element of $A$ strictly smaller
than~$x$, or $a$ if $x$ is the minimum of~$A$.
We denote by $s(a)$ the minimum of~$A$.

\begin{thm}[Monotonicity theorem]
Let $f : (a,b) \to \K$ be a definable function .  Then, there exists a pseudo-finite set
$A \subseteq (a,b)$, such that on each sub-interval $\Pa{x, s(x)}$, with $x
\in A \cup \set a$, the function is either constant, or strictly monotone and continuous.
\end{thm}

The proof proceeds as in \cite[Thm.~3.1.2]{vdd}:
we derive it from the three lemmata below.
Let $I \subseteq \K$ be an open interval, and $f: I \to \K$ be definable.

\begin{lemma}
There is a subinterval of $I$ on which $f$ is constant or injective.
\end{lemma}

\begin{lemma}
If $f$ is injective, then $f$ is strictly monotone on a subinterval of~$I$.
\end{lemma}

\begin{lemma}
If $f$ is strictly monotone, then $f$ is continuous on a subinterval of~$I$.
\end{lemma}

The three lemmata imply the monotonicity theorem as follows.
Let
\[\begin{aligned}
X := \bigl\{&
 x \in (a,b): \text{ on some subinterval of } (a,b) \text{ containing } x, \\
&f \text{ is either constant, or strictly monotone and continuous }\bigr\}.
\end{aligned}\]
By the three lemmata, $A := (a,b) \setminus X$ has empty interior, and
therefore it is pseudo-finite.
It is easy to see that $A$ satisfies the conclusion of the theorem.

The proof of the first lemma is as in~\cite[\P3.1.5, Lemma 1]{vdd},
substituting ``finite'' with ``pseudo-finite'', and ``infinite'' with ``non
pseudo-finite'', and using the results in Section~\ref{sec:pf}.

The proof of the second lemma is as in~\cite[\P3.1.5, Lemma 2]{vdd}.
The ``difficult case'', that is proving that the set of points in $\K$ satisfying the formula
\[
\Phi_{++}(x) := x \text{ is a strict local maximum for } f
\]
has empty interior, is immediate from Lemma~\ref{lem:min}.

The proof of the third lemma in~\cite[\P3.1.5, Lemma 3]{vdd} works also here,
with the same modifications as for the first lemma.

We also have the following consequence:
\begin{corollary}
Let $f: (a,b)  \to \K$ be definable, and $c \in (a,b)$.
The limits $\lim_{x \to c^-} f(x)$ and $\lim_{x \to c^+} f(x)$ exist in~$\Kinf$
Also the limits $\lim_{x \to a^+} f(x)$ and $\lim_{x \to b^-} f(x)$ exist
in~$\Kinf$. 
\end{corollary}

\subsection{Constructibility and partition into cells }
In this sub-section we will prove the following theorem.
\begin{thm}\label{thm:lomin-constructible}
If $\K$ is locally o-minimal, then every definable subset of $\K^n$ is
constructible.
Therefore, $\K$ coincides with its open core.
\end{thm}
Let $X \subseteq \K^n$ be definable.

If $n = 1$, then $\fr(X)$ has empty interior (by Remark~\ref{rem:boundary})
and therefore is pseudo-finite.
The same is true for $\K \setminus X$, and therefore, (also by
Remark~\ref{rem:boundary}) $\bd(X)$ is pseudo-finite.
Thus, by Lemma~\ref{lem:discrete-boundary}, $X$ is locally closed, and,
\emph{a fortiori}, constructible.

If $n = 2$, we consider first the case when, for every $a \in \K$, $X_a$ is
pseudo-finite.
\Wlog, we can assume that $X \subseteq (0,1) \times (0,1)$.
Following~\cite[3.1.7]{vdd}, we call $(a,b) \in \K^2$ \emph{normal} if there
is an open box $I \times J$ around $(a,b)$ such that
\begin{itemize}
\item either $(I \times J ) \cap X = \emptyset$ (and hence $(a,b) \notin X$,
\item or $(a,b) \in X$ and $(I \times J) \cap X = \Gamma(f)$ for some
continuous function $f: I \to J$.%
\footnote{Since we assumed that $X \subseteq (0,1) \times (0,1)$, we do not need
to worry about the behaviour at infinity.}
\end{itemize}

Note that the set $\Normal := \set{(a,b) \in \K^2: (a,b) \text{ is normal}}$ is
definable.
Note also that $Z := \K^2 \setminus \Normal$ is contained in $[0,1] \times [0,1]$.

We call $\bad := \Pi^2_1(Z) \subseteq \K$ the set of ``bad'' points, and $\good
:= \K \setminus \bad$ the set of ``good'' points.

\begin{claim}
$Z$ is \dcompact.
\end{claim}
In fact, it is bounded, and, by definition, its complement $Y$ is open.

Hence, $\bad$ is \dcompact.
\begin{claim}
$\bad$ is pseudo-finite.
\end{claim}
\begin{proof}
If not, let $I \subseteq \bad$ be a non-empty open interval.
For every $a \in I$; let $\beta(a) := \min(Z_a)$.
Conclude as in the proof of~\cite[3.1.7]{vdd}.
\end{proof}


Note that the set of normal points $\Normal$ is locally closed by definition.
Moreover, by Lemma~\ref{lem:finite-union}, $Z$ is pseudo-finite, and hence
closed.
Thus, $X = \Normal \cup Z$ is constructible.

We treat now the general case, when there exists some $a \in \K$ such that
$X_a$ is not pseudo-finite.

Let $A := \set{a \in \K: X_a \text{ is not pseudo-finite}}$.

\begin{lemma}
If $A$ is not meager, then $X$ has non-empty interior.
\end{lemma}
\begin{proof}
If $A$ is not meager, then it contains an open interval~$I$.
So, \wloG we can assume $A = I$.
For every $a \in A$, there exist $b \in \K$ accumulation point of $X_a$.
Let $f(a) \in \K$ be the smallest such~$b$.
Note that there exists $b' > f(a)$ such that $\Pa{f(a), b'} \subseteq
X_a$; let $g(a)$ the greatest such~$b'$.
Let $I' \subseteq I$ open interval, such that $f$ and $g$ are continuous
on~$I'$.
Then, the set $\set{(a,b) : a \in I, f(a) < a < g(a)}$
is an open set contained in~$X$.
\end{proof}

Note that $X = \inter X \cup (X \setminus \inter X)$.
Since $X \setminus \inter X$ has empty interior, we can assume that $X$ has
empty interior.
Thus, $A$ is pseudo-finite.
By the case $n = 1$, we have that, for every $a \in \K$,
$X_{[a]}$~is constructible.
Thus, by Corollary~\ref{cor:finite-constructible},
$X \cap \pi^{-1}(A) = \bigcup_{a \in A} X_{[a]}$ is constructible.
Since we have seen that $X \setminus \pi^{-1}(A)$ is also constructible, we
are done.

\emph{The rest of the proof proceeds by enumerating all the possible values of~$n$.}

For the general case, let us prove, by induction on $n$, the following 2
statements:
\begin{enumerate}
\item[$I_n$] Every definable $X \subseteq \K^n$ is constructible;
\item[$II_n$] If $A \subseteq \K^n$ is open and definable, and $f: A \to \K$ is
a definable function, then there exists a non-empty open box $B \subseteq A$
such that $f \rest B$ is continuous.
\end{enumerate}
Note that $II_n$ implies that the set of discontinuity points of $f$ is nowhere dense.

We have already proved $I_1$ and $II_1$.
So, assume that we have already proved $I_m$ and $II_m$, for every $m \leq n$,
and let us prove them for $n + 1$.

We will prove $I_{n+1}$ by induction on $(d, k) := \fdim(X)$, the full
dimension of~$X$. Assume we have already proved the statement for any $X'$ of
full dimension less than $(d,k)$.

If $d = n$, note that $X = \inter X \sqcup (X \setminus \inter X)$.
$\inter X$ is open, and therefore constructible, and $(X \setminus \inter X)$
has empty interior, and thus dimension less than~$n$; therefore, by the
inductive hypothesis, $X \setminus \inter X$ is constructible, and we are
done.

If $d = 0$, then $X$ is pseudo-finite, and we are done.

If $0 < d < n$, \wloG we can assume that $Y := \pi(X)$ has non-empty interior,
where $\pi := \Pi^n_d$.
Let $X' := X \cap \pi^{-1}(Y \setminus \inter Y)$.
Note that $\fdim(X') < \fdim(X)$, and therefore $X'$ is constructible.
Thus, we can assume that $Y$ is open.

Moreover, after a change of coordinates, we can assume that $X$ is bounded.

We call $(a,b) \in \K^n \times \K$ normal if there is an
open box $B \times J$ around $(a,b)$ such that
\begin{itemize}
\item either $(B \times J ) \cap X = \emptyset$ (and hence $(a,b) \notin X$),
\item or $(a,b) \in X$ and $(B \times J) \cap X = \Gamma(f)$ for some
continuous function $f: B \to J$.
\end{itemize}

Note that the set
$\Normal := \set{(a,b) \in \K^n \times \K: (a,b) \text{ is normal}}$ is definable.
Note also that $Z := \K^n \times \K \setminus \Normal$ is contained in $[0,1] \times [0,1]$.

We call $\bad := \pi(Z) \subseteq \K^n$ the set of ``bad'' points,
and $\good := \K^n \setminus \bad$ the set of ``good'' points.

Again, $\Normal$ is open by definition, and thus $\bad$ is \dcompact.
Thus, it suffices to prove that $\fdim(Z) < \fdim(X)$ to obtain that $X$ is
constructible.
If, for contradiction, $\fdim(Z) = \fdim(X)$, then $\pi(Z)$ has non-empty
interior; let $B \subseteq \pi(Z)$ be a non-empty open box.
For every $a \in B$, let $\beta(a) := \min Z_a$.
Proceeding as in the case $n = 2$, and using the inductive hypothesis $II_n$,
we get a contradiction.

Let us prove now $II_{n+1}$.
We use the same technique in~\cite[3.2.17]{vdd}.
Let $f : A \to \K$ be a definable function, with $A \subseteq \K^n \times \K$
open and definable.
Define the set $A^*$ of well-behaved points for $f$ as in~\cite[3.2.17]{vdd}.
By the same proof as in~\cite{vdd}, $A^*$ is dense in~$A$.
Since $A$ is definable, by $I_{n+1}$ $A$ is consructive, and therefore it
contains a non-empty open box~$B$.
Moreover, by~\cite[Lemma~3.2.16]{vdd}, $f$ is continuous on~$B$.
\hfill \qedsymbol
\bigskip

We shall now prove a stronger version of Thm.~\ref{thm:lomin-constructible},
where we give some more details on the structure of definable sets.

\begin{definizione}
Let $X \subseteq \K^n$ be definable (and hence constructible) and of
dimension~$d$, and let $\pi: \K^n \to L$ be a projection onto a coordinate
space of dimension~$d$.
For notational convenience, assume that $\pi = \Pi^n_d$, and that
$X$ is bounded.
A point $a \in \K^d$ is \intro{$(X,\pi)$-good} if there exists a neighbourhood
$A$ of~$a$, such that, 
for every $b \in \K^{n-d}$ there exists a neighbourhood $B$ of~$b$, with 
either $A \times B$ is disjoint from~$X$,
or $(A \times B) \cap X = \Gamma(f)$ for some (unique and definable)
continuous function $f: A \to B$.
A point $x \in \K^d$ is \intro{$(X,\pi)$-bad} if it is not good.
\end{definizione}
If $X$ is unbounded, let $\phi : \K \to (0,1)$ be a definable order-preserving
homeomorphism, and $\phi^n: \K^n \to (0,1)^n$ be the corresponding map.
Then, we say that $a$ is \intro{$(X,\pi)$-good} if 
$\phi^d(a)$ is $(\phi^n(X), \pi)$-good: 
the definition dose not depend on the particular choice of~$\phi$,
and it is equivalent to the one in~\cite{vdd}.

Note that if $L = \K^n$, then no point of $\fr X$ is $(X,\pi)$-good (where
$\pi$ is the identity map).
Note also that if $y \in L$ is $(X,\pi)$-good, then $X_y$ is pseudo-finite
(because it is discrete).

\begin{definizione}
A \intro{multi-cell} of dimension $d$ in~$\K^n$,
with respect to a coordinate plane~$L$,
is a definable subset $X \subseteq \K^n$ of dimension equal to~$\dim L$,
such that, calling $\pi$ the orthogonal projection onto~$L$:
\begin{enumerate}
\item $U := \pi(X)$ is open in~$L$;
\item every point of $U$ is $(X,\pi)$-good.
\end{enumerate}
\end{definizione}
Note that the multi-cells of dimension $0$ are the pseudo-finite sets, and the
multi-cells of dimension $n$ in $\K^n$ are the definable open subsets of~$\K^n$.

Note also that if $X$ is a multi-cell, then, locally around every point
of~$X$, $X$~is a cell.
However, $X$ might not be a cell around points not of~$X$.
For instance, let $D \subseteq (0,1)$ be an infinite pseudo-finite set, and let
\[
X := \set{(x,y) \in \K^2: x > 0, y > 0, \exists m \in D, y = m x}.
\]
Then, $X$ is a multi-cell, but, for every neighbourhood $V$ of $0$,
$X \cap V$ has infinitely many definably connected components, and thus cannot
be a cell, nor a finite union of cells.

\begin{thm}\label{thm:lmin-cell}
Let $X \subseteq \K^n$ be definable.
Then, there exists a finite partition of $X$ into multi-cells.
The number of multi-cells is bounded by a function of~$n$.
\end{thm}
\begin{proof}
We will prove the statement by induction on $n$ and $(d,k) := \fdim(X)$.
Assume that we have already proved the statement for every $n' < n$ and for
every $X'$ of full dimension less than $(d,k)$.

If $d = n$, then $X = \inter X \sqcup (X \setminus \inter X)$.
$\inter X$ is open, and thus a multi-cell.
$(X \setminus \inter X)$ has dimension less than~$n$, and thus it is a
finite disjoint union of multi-cells.
Therefore, $X$ is a finite disjoint union of multi-cells.

If $d = 0$, then $X$ is pseudo-finite, and hence a multi-cell.

If $0 < d < n$, let $L$ be a coordinate space of dimension~$d$, such that
$\pi(X)$ has non-empty interior in~$L$, where $\pi$ is the projection on~$L$.
As in the proof of Thm.~\ref{thm:lomin-constructible}, the set $\bad$ of
$(X,\pi)$-bad points is nowhere dense and closed, and~$\good$, the set of
$(\pi,X)$-good points, is open.
Thus, let $Y := X \cap \pi^{-1}(\bad)$ and $Z := X \cap \pi^{-1}(\good)$.
By definition, $Z$ is a multi-cell, and, since $\bad$ is nowhere dense, $Y$
has full dimension less than $(k,d)$, and therefore it is a finite disjoint union of
multi-cells.
Thus, $X$ is a finite union of multi-cells, and we are done.

From the above argument, we can see that the number $N$ of multi-cells
partitioning $X$ is bounded by a function of $n$, $d$, and~$k$.
However, $d$ and $k$ are bounded by functions of $n$, and thus $N$ is bounded
by a function of~$n$.%
\footnote{Something like~$n^2$.}
\end{proof}

We would like to obtain a structure theorem for open definable sets.
Such a theorem should also give a further refinement of the above theorem.

\begin{remark}
Let $f : \K^n \to \K^m$ be definable.
Since $\Gamma(f)$ is constructible, then, by Lemma~\ref{lem:function-fs},
$\Dis(f)$ is meager, and thus nowhere dense and of dimension strictly less
than~$n$.
\end{remark}


\subsection{Additional results on locally o-minimal structures}
\begin{proviso}
In this subsection, $\K$ is a locally o-minimal structure.
\end{proviso}
Therefore, $\K$ is \aminimal,
and every definable subset of $\K^n$ is constructible.
Hence, we can apply lemmata~\ref{lem:dim-union}, \ref{lem:amin-dim-1},
and~\ref{lem:amin-dim-2}, and obtain the following lemma
(\cf Example~\ref{exa:dim-constructible}).
\begin{lemma}\label{lem:loc-min-dim}
Let $C$ and $C'$ be definable subsets of $\K^n$.
Then, $\dim(C \cup C') = \max\Pa{\dim C, \dim C'}$,
and $\dim(\fr C) < \dim C$ if $C$ is non-empty;
besides, $\dim C = 0$ iff $C$ is pseudo-finite.
\end{lemma}
\begin{proof}
We only need to show that $\dim( \fr C) < \dim C$.

First, we claim that $\fdim (\cl C) = \fdim C$.
In fact, $\Pi^n_m(\cl C) \subseteq \cl{\Pi^n_m C}$, and therefore, by local
o-minimality, if the former has non-empty interior, also the latter has
non-empty interior.
Hence, $\fdim(\fr C) \leq \fdim C$.

Assume, for contradiction, that $d := \dim C = \dim(\fr C)$;
\wloG, $\pi(\fr C)$ contains a non-emtpy open box~$U$, where $\pi := \Pi^n_d$.
Let $\bad$ be the set of $\pair{C, \pi}$-bad points:
$\bad$ is nowhere dense, and therefore there exists a non empty open box
$V \subseteq U \setminus \bad$.
However, $\pi \Pa{\fr(C \cap \pi^{-1}(V))} \subseteq \fr V$,
contradicting the definition of~$U$. 
\end{proof}

\begin{lemma}
Let $C(t)$ be a definable increasing family of subsets of~$\K^n$,
such that $\dim(C(t)) \leq d$ for every~$t$.
Let $C := \bigcup_t C(t)$.
Then, $\dim C \leq d$.
\end{lemma}

\begin{lemma}
Every locally o-minimal structure has definable Skolem functions.
\end{lemma}
\begin{proof}
Same as in~\cite[Prop.~6.1.2]{vdd}.
\end{proof}

\begin{lemma}[Curve selection]
If $a \in \fr X$, where $X \subseteq \K^n$ is definable, then there exists a
definable continuous injective map $\gamma: (0, 1) \to X$, such that
$\lim_{t  \to 0} \gamma(t) = x$.
\end{lemma}
\begin{proof}
Same as in~\cite[Corollary~6.1.5]{vdd}.
\end{proof}

Given $A \subseteq \K$, denote by $\dcl(A)$ the definable closure of $A$ in~$\K$.

\begin{lemma}
Assume that $\dcl(\emptyset)$ is dense in~$\K$.
Then, $\K$~has the Exchange Property (EP).
\end{lemma}
\begin{proof}
Let $A \subseteq \K$, and $a, b \in \K$.
We have to prove that, if $b \in \dcl(A \cup \set a) \setminus \dcl(A)$,
then $a \in \dcl(A \cap \set b)$.
Let $f: \K \to \K$ be $A$-definable, such that $f(a) = b$.
By the monotonicity theorem, there exists $D \subset \K$ pseudo-finite, such
that $f$ is continuous and either constant or strictly monotone on each
subinterval of $\K \setminus D$.
Note that $D$ can be taken $A$-definable.

If $a \in D$, let $y', y'' \in \dcl(\emptyset)$, such that $y' < a < y''$,a
nd $D \cap (y', y'') = \set a$.
Then, $a \in \dcl(A)$, a contradiction.

If instead $a \notin D$, let $y', y'' \in \dcl(\emptyset)$, such that $y' < a
< y''$ and $f$ is continuous and either constant or strictly monotone on $I :=
(y',
y'')$.
If $f$ is constant on~$I$, then $b \in dcl(A)$, absurd.
If instead $f$ is continuous and strictly monotone on~$I$, then $f$ has a
compositional inverse $g$ on~$I$, and $g$ is $A$-definable;
thus, $a = g(b)$, and therefore $a \in \dcl(A \cup \set b)$.
\end{proof}

\begin{example}
A locally o-minimal structure does not necessarily satisfy the NIP.
Equivalently, there exists a locally o-minimal structure that does have the
Independence Property.
In fact, fix an o-minimal structure~$M$ (expanding a field) in a language~$L$.
Let $P$ a new binary predicate symbol.
For every $0< n \in \Nat$, let $f_n : {1 , \dotsc, 2^n} \to \mathcal P(n)$ be
an enumeration of the subsets of ${1, 2, ..., n}$, such that $f_n$ extends
$f_{m}$ for every $m \leq n$.
Let $P_n := \bigcup_{k = 1}^{2^n} \set k \times f_n(k) \subseteq M^2$,  and
$N_n := (M, P_n)$ be the $L(P)-$expansion of~$M$, where $P$ is interpreted
by~$P_n$.
Let $N^* = (M^*, P^*)$ be a non-principal ultra-product of then~$N_m$.
Since $M$ is o-minimal, $M^*$ and each $N_m$ are o-minimal, and $N^*$ is
locally o-minimal.
However, it is clear that $N^*$ is not o-minimal, and does have the Independence Property.
Moreover, if $L$ is countable, then $N^*$ is $\omega$-saturated, and therefore,
by the proof~\cite[1.17]{DMS}, $N^*$~does not have the Exchange Property.
\end{example} 

The definitions and proof in~\cite{miller96} work almost verbatim for locally
o-minimal structures.
Hence, we have the following theorem.
\begin{thm}
If $\K$ is a locally o-minimal structure, then it is either
\intro{power-bounded}, or it is \intro{exponential}
(that is defines an exponential function).
If $\K$ is power bounded, then for every ultimately non-zero definable
function $f: \K \to \K$ there exists $c \in \K^*$ and $r$ in the field of
exponents of~$\K$, such that $f \sim cx^r$.
\end{thm}


\section[Definably complete structures]{More results on definably complete structures}
\label{sec:definably-complete}
\subsection{Functions}
Let $\K$ be a definably complete.
In this sub-section, we will prove that certain subsets of $\K^n$ are meager.
Thus, if $\K$ is Baire, those subsets cannot be all of~$\K^n$.
However, by inspecting the proofs, one will see that in some cases (marked by
(*)), we actually prove that those sets are pseudo-enumerable, and thus, even
without assuming that $\K$ is Baire, they cannot be all of~$\K^n$ (see subsection~\ref{sec:enumerable}).

Let $f: A \to \K^n$ be a definable function.
For every $s > 0$, define
\begin{multline*}
\Dis(f, s) := \bigl\{x \in \K: \forall t > 0,\\
f\Pa{A \cap B(x,t)} \text{ is not contained in any open ball of radius } s
\bigr\}.%
\footnotemark
\end{multline*}
\footnotetext{Note that there is a misprint in the definition of $\Dis(f, s)$
  in~\cite[1.8]{DMS}.}
Then, each $\Dis (f, s)$ is closed in~$A$, and $\Dis(f)$, the set of discontinuity
points of~$f$, is equal to $\bigcup_s \Dis(f, s)$, and, therefore, it is an
$\Fs$ subset of~$A$.

\begin{lemma}[*]
Let $f: \K \to \K$ be definable and monotone.
Then, $\Dis(f)$ is meager.
\end{lemma}
\begin{proof}
If $\Dis(f)$ were not meager, then there would exists $\varepsilon > 0$ such
that $\Dis(f, \varepsilon)$ has non-empty interior; 
let $(a,b) \subseteq \Dis(f, \varepsilon)$, with $a < b$.
\Wlog, $f$ is increasing.
Define $m := \sup_{t < b} f(t) = \lim_{t \to b^-} f(t) \in \K$.
Let $c \in (a,b)$ such that $f(c) > m - \varepsilon/2$.
Since $f$ is monotone, and $(a,b) \subseteq \Dis(f, \varepsilon)$, there
exists $d \in (c,b)$ such that $f(d) \geq f(c) + \varepsilon > m +
\varepsilon/2$, contradicting the definition of~$m$.
\end{proof}


The open question is whether, with the same hypothesis as in the above lemma,
there exists a point $x \in \K$ such that $f$ is differentiable at~$x$.

\begin{lemma}\label{lem:function-fs}
Let $f: \K^n \to \K^m$ be definable.
If $\Gamma(f)$ is an $\Fs$ set, then $\Dis(f)$ is meager.
\end{lemma}
\begin{proof}
If, for contradiction, $\Dis(f)$ is not meager, then, since it is an $\Fs$, it
contains a non-empty open box~$B$.
Therefore, \wloG we can assume that $\Dis(f) = \K^n$.
Let $\Gamma(f) = \bigcup_t X(t)$, where $\Pa{X(t)}_{t \in \K}$ is a definable
increasing family of \dcompact sets.
Let $Y(t) := \Pi^{n+m}_n \Pa{X(t)}$.
Note that each $Y(t)$ is \dcompact, and $\K^n = \bigcup_t Y(t)$.
Since $\K^n$ is Baire, there exists $t_0$ such that $Y(t)$ contains a
non-empty open box~$B'$.
Let $B''  \subseteq B'$ be a closed box with non-empty interior, and $g := f
\rest B''$.
Note that $\Gamma(g) = X(t_0) \cap (B'' \times \K^m)$; therefore, $\Gamma(g)$
is compact, and so $g$ is continuous, contradicting the fact that $B''
\subseteq \Dis(f)$.
\end{proof}

\begin{lemma}[*]\label{lem:min}
Let $f: \K^n \to \K$ be definable.
Define $N_f := \set{x \in \K^n: x \text{ is a strict local minimum for } f}$.
Then, $N_f$ is meager in~$\K^n$.
\end{lemma}
\begin{proof}
For every $r > 0$, let
\[
N(r) := \set{x \in \K^n: \abs{x} \leq r \ \&\ x \text{
    is a strict minimum for } f \text{ in the ball } B(x, 1/r)}.
\]
Note that $N_f = \bigcup_r N(r)$.
Moreover, each $N(r)$ is closed and discrete, and hence nowhere dense.
Thus, $N_f$ is meager.
\end{proof}

\begin{lemma}[*]\label{lem:min-image}
Let $f: \K^n \to \K$ be definable.
Define $M_f := \set{x \in \K^n: x \text{ is a local minimum for } f}$.
Then, $f(M_f)$ is meager in~$\K$.
\end{lemma}
\begin{proof}
For every $r > 0$, let
\[
M(r) := \set{x \in \K^n: \abs{x} \leq r \ \&\ x \text{
    is a minimum for } f \text{ in the ball } B(x, 1/r)}.
\]
Note that
$f(M_f) = \bigcup_r f\Pa{M(r)}$. 
Fix $r > 0$.
It suffices to prove that $Y := f\Pa{M(r)}$ is nowhere
dense.
We will prove the stronger statement that $Y$ is pseudo-finite.
Assume, for contradiction, that $Y$ has an accumulation point $y \in \K$.
For every $\delta > 0$, let 
$U(\delta) := f^{-1}\Pa{B(y,\delta) \cap Y  \setminus \set y}$,
that is: $x \in U(\delta)$ iff $f(x) \in B(y,\delta)$, $f(x) \neq y$,
$\abs{f(x)} \leq r$, and $x$ is a minimum of $f$ on $B(x, 1/r)$.

Let $C(\delta)$ be the closure of $U(\delta)$ in~$\K^n$.
Since each $C(\delta)$ is a non-empty \dcompact subset of $\cl{B(0,r)}$, the
intersection of the $C(\delta)$ is non-empty; let $x \in \bigcap_\delta
C(\delta)$.
Let $x_1, x_2 \in B(x,1/(2r)) \cap U(\delta)$, such that $f(x_1) < f(x_2)$ (they
exist by definitions).
However, $x_1 \in B(x_2,r)$, and $x_2 \in M(r)$;
therefore, $f(x_1) \geq f(x_2)$, absurd.
\end{proof}


\begin{definizione}\label{def:first class}
A definable function $f: X \to Y$ is \intro{of first class} if there exists
a function $F: \K \times X \to Y$, such that, for every $t \in \K$ 
and $x \in X$,
\begin{enumerate}
\item  $f_t(x) := F(t, x) : X \to Y$ is a continuous function
of~$x$,
\item $\lim_{t \to + \infty} f_t(x) = f(x)$;
\end{enumerate}
that is, $f$ is a point-wise limit of a definable family of continuous
functions $(f_t)_{t \in \K}$.
\end{definizione}

\begin{lemma}\label{lem:first-class}
Let $X$ be Baire, and  $f: X \to \K^m$ be of first class.
Then, $\Dis(f)$ is meager (in~$X$).
\end{lemma}
\begin{proof}
Minor variation of~\cite[Thm.~7.3]{oxtoby}.
It suffices to show that, for every $\varepsilon > 0$,
$F_\varepsilon := \set{x \in X: \omega(x) \geq 5 \varepsilon}$ is nowhere
dense, where
\[
\omega(x) := \lim_{\delta \to 0^+} 
\sup \set{ \abs{f(x') -   f(x'')}: x', x'' \in B(x; \delta)}.
\]
Fix an open definable subset $X' \subseteq X$.
For every $i \in \K$, let
\[
E_i := \bigcap_{s, t \geq i}
\set{x \in X': \abs{f_s(x) - f_t(x)} \leq \varepsilon}.
\]
Then, $(E_i)_{i \in \K}$ is an increasing family of closed subsets of~$X'$,
and $\bigcup_i E_i = X'$.
Since $X$ is Baire and $X'$ is open, $X'$ is not meager in itself,
and therefore $E_{i_0}$ must have non-empty interior (in~$X'$), 
for some $i_0 \in \K$.
Let $V$ be a non-empty open subset in~$E_{i_0}$.
We have $\abs{f_t(x) - f_s(x)} \leq \varepsilon$ for all $x \in V$
and $s, t  \geq i_0$.
Putting $t = i_0$ and letting $s \to \infty$, it follows that
$\abs{f_{i_0}(x) - f(x)} \leq \varepsilon$ for all $x \in V$.
For all $x_0 \in B$ there exists a neighbourhood $U(x_0) \subseteq V$,
such that $\abs{f_{i_0}(x) - f_{i_0}(x_0)} \leq \varepsilon$ for all 
$x \in U(x_0)$.
Hence $\abs{f(x) - f_{i_0}(x_0)} \leq 2 \varepsilon$ for all $x \in U(x_0$.
Therefore $\omega(x_0) \leq 4 \varepsilon$, 
and so no point of $V$ belongs to~$F$.
Thus, every non-empty open set $X'$ contains a non-empty open subset~$V$
disjoint from~$F$.
This shows that $F$ is nowhere dense.
\end{proof}

\subsection{Pseudo-enumerable sets}
As usual, $\K$ will be a definably complete structure, expanding an ordered
field.

\begin{definizione}
$N \subset \K$ is a \pN subset of $\K$ if it is a definable closed,
discrete and unbounded subset of $K^{\geq 0}$.
\end{definizione}

There are 2 cases: either every closed, definable and discrete subset of $\K$
is pseudo-finite (and hence $\K$ is \aminimal, and thus Baire), or not.
Equivalently: either $\K$ is \aminimal, or a \pN set exists.
Suppose we are in the second case.

\begin{definizione}
Let $X \subseteq \K^n$ be definable.
We say that $X$ is \intro{pseudo-enumerable} if 
there exists $N$ a \pN subset of~$\K$, and a
definable bijection $f: N \bij X$.
We say that $X$ is \intro{at most pseudo-enumerable}
if $X$ is either pseudo-finite or pseudo-enumerable.

We say that $\Pa{X(i)}_{i \in I}$ is a \intro{strongly uniform family} 
of at most
pseudo-enumerable sets, if $\Pa{X(i)}_{i \in I}$
is a definable family of subsets of~$\K^n$ (for some~$n$),
and there exist $Y \subset \K \times Y$ definable, and
$F: Y \to \K^n$ definable,
such that, for every $i \in I$, $F(Y_{[i]})  = X(i)$,
and $Y_i$ is a discrete definable closed subset of~$\K^{\geq 0}$.
\end{definizione}

Note that we did \emph{not} claim that 2 pseudo-enumerable sets are in a
definable bijection (we do not know whether this is true or not, but we
conjecture it \emph{is} true).

\begin{remark}
Let $X$ be at most pseudo-enumerable, and $Y$ be a definable subset of~$X$.
Then, $Y$ is at most pseudo-enumerable.
In particular, a definable subset of a \pN set is at most pseudo-enumerable.
\end{remark}
\begin{proof}
Assume that $Y$ is not pseudo-finite.
Therefore, $X$ is not pseudo-finite, and hence it is pseudo-enumerable;

If $X$ is a \pN set, then $Y$ is not bounded (otherwise, it would be
pseudo-finite), and hence it is a \pN set, and in particular pseudo-enumerable.

Otherwise, let $g: N \to X$ bijection, where $N$ is a \pN subset of~$\K$.
Let $h :=  g \rest Y: Y \to N$,
and $N' := h(N) \subseteq N$.
Notice that $h: Y \to N'$ is a bijection;
therefore, $N'$ is not pseudo-finite, and $Y$ is pseudo-enumerable.
\end{proof}

\begin{remark}
If $X$ is at most pseudo-enumerable, and $f: X \to \K^m$  is definable,
then $f(X)$ is at most pseudo-enumerable.
In particular, the image of a \pN subset of~$\K$ is either pseudo-finite,
or pseudo-enumerable.
\end{remark}
\begin{proof}
Assume that $Y := f(X)$ is not pseudo-finite.
Thus, $X$ is not pseudo-finite, and hence it is pseudo-enumerable;
let $g: N \to X$ bijection, where $N$ is a \pN subset of~$\K$.
For every $y \in Y$, define 
\[
h(y) := \min\set{i \in N: f(g(i)) = y},
\]
and $N' := h(N) \subseteq N$.
Notice that $h: Y \to N'$ is a bijection;
therefore, $N'$ is not pseudo-finite.
Thus, $N'$ is a \pN set, and $h^{-1}: N' \to Y$ is a bijection,
proving that $Y$ is pseudo-enumerable.
\end{proof}

\begin{lemma}\label{lem:enum-1}
If $X$ is at most pseudo-enumerable, then there exists $Y \subset \K^{\geq 0}$
definable, closed and discrete, and $f: Y \bij X$ definable bijection.
\end{lemma}
\begin{proof}
If $X$ is pseudo-enumerable, the above is true by definition.
If $X \subset \K^n$ is pseudo-finite, we proceed by induction on~$n$.

If $n = 1$, $Y = X$ satisfies the conclusion.
If $n = 2$, \wloG we can assume that $X \subset [0,1]^2$.
Let $X_1 := \Pi^2_1(X) \subset [0,1]$, and
$d := \frac 1 2 \min\set{x' - x: x < x' \in   X_1}$
(or $d := 1/2$ if $X_1$ has only one element).
For every $\pair{x, y} \in X$, define $g(x, y) := x + d y $.
Then, $g$ is injective, and $Y := g(X)$ is a pseudo-finite subset of~$\K$.
Assume that we have already proved the conclusion for $n \geq 2$,
we prove it for $n + 1$.
Let $X_1 := \pi(X)$, where $\pi := \Pi^{n+1}_n$.
By inductive hypothesis, there exists $Y_1 \subset \K$ pseudo-finite, and
$g_1: X_1 \bij Y_1$ definable bijection.
For every $\pair{\x, y} \in X$, let $f(\x,y) := \pair{g_1(\x), y}$,
$f: X \to \K^2$, and $X' := f(X)$.
Then, $f$ is injective, and, by the case $n = 2$, we can conclude.
\end{proof}

\begin{remark}
The union of 2 at most pseudo-enumerable sets is pseudo-enumerable.
\end{remark}
\begin{proof}
Let $X_1$ and $X_2$ be at most pseudo-enumerable.
We want to prove that $X_1 \cup X_2$ is at most pseudo-enumerable.
\Wlog, $X_1$ and $X_2$ are disjoint.
If both $X_1$ and $X_2$ are pseudo-finite, then we are done.
Otherwise, \wloG $X_2$ is pseudo-enumerable.
By Lemma~\ref{lem:enum-1}, there exist
$N_i$ closed, definable, and discrete subsets of~$\K^{\geq 0}$, and
$f_i : N_i \bij X_i$ definable bijections, $i = 1, 2$.

For every $i \in N_k$, let $s_k(i) := \min\set{i' \in N_k: i' > i}$ be the
successor of $i$ in $N_k$ (or $s_1(m) = m + 1$ if $m = \max(S_1)$), 
$k = 1, 2$.
Define $f: X_1 \cup X_2 \to \K$ as:
\[
f(x) := \begin{cases}
f_2(x) & \text{if } x \in X_2;\\
f_1(x) & \text{if } x \in X_1 \et f_1(x) \notin N_2;\\
\frac{1}{2} \Bigl( f_1(x) + \min \Pa{s_1(f_1(x)), s_2(f_1(x))} \Bigr)
  & \text{if } x \in X_1 \et f_1(x) \in N_2.
\end{cases}
\]
Then, $f$ is injective, and $f(X_1 \cup X_2)$ is a \pN subset of~$\K$.
\end{proof}

\begin{lemma}\label{enum:product}
\begin{enumerate}
\item 
The product of 2 pseudo-enumerable sets is pseudo-enumerable.
\item
The product of 2 at most pseudo-enumerable sets is at most pseudo-enumerable.
\end{enumerate}

\end{lemma}
\begin{proof}
It suffices to prove that if $X$ and $Y$ are at most pseudo-enumerable,
then $Z := X \times Y$ is at most pseudo-enumerable.
If $X$ and $Y$ are both pseudo-finite, then $Z$ is also pseudo-finite, and we
are done.
By Lemma~\ref{lem:enum-1} and the above Remark,
\wloG $X = Y$ is a \pN subset of~$\K$, and $X > 1$.
By applying the above remark again to $X \times \set 0$ and $Y \times \set 1$,
\wloG $X$ and $Y$ are disjoint subsets of a \pN subset $N$ of $\K$,
and $N > 1$.

For every $i \in N$, let $s(i) := \min \set{i' \in N: i' > i}$ 
be the successor of $i$ in~$N$, and
\[
d(i) := \min\set{(s(i') - i'): i \geq i' \in N} > 0.
\]
Define $f: X \times Y \to \K$,
\[
f(x,y) :=
\begin{cases}
x + \frac{d(x)}{2x}y & \text{if } x > y;\\
y + \frac{d(y)}{2y}x & \text{if } x < y.
\end{cases}
\]
Notice that the definition is symmetric in $X$ and $Y$
(that is, if we exchange the r\^oles of $X$ and $Y$, $f$ remains the same).

\begin{claim}
$f$ is injective.
\end{claim}
Assume not.
Let $\pair{x,y} \neq \pair{x', y'} \in X \times X$, such that
$f(x,y) = f(x', y')$.
Let $a := \max(x,y)$, $b := \min(x,y)$,
$a' := \max(x', y')$, and $b' := \min(x',y')$.
Then, $f(x,y) = a + \frac{d(a)}{2 a} b$, and similarly for $f(x',y')$.
Hence,
\[
a + \frac{d(a)}{2 a} b = a' + \frac{d(a')}{2 a'} b'.
\]
\Wlog, $a' \geq a$.
If $a' = a$, then $b = b'$; since $X$ and $Y$ are disjoint, this implies that
$\pair{x,y} = \pair{x', y'}$, a contradiction.
Otherwise, $a' > a$, and
\[
d(a)/2 < a' - a = \frac{d(a)}{2 a} b - \frac{d(a')}{2 a'} b'
\leq d(a)/2,
\]
absurd, proving the claim.

\begin{claim}
$f(Z)$ is \pN.
\end{claim}
Since $X$ and $Y$ are not pseudo-finite, it suffices to show that
$f(Z)$ is closed and discrete.
Both closedness and discretedness follow if we show that, for every $r > 0$,
$f(Z) \cap [0,r]$ is pseudo-finite.
However, $f(x,y) > \max(x, y)$, and therefore
$f^{-1}([0,r]) \subseteq \Pa{X \cap [0,r]} \times \Pa{Y \cap [0,r]}$,
and the latter is pseudo-finite.
\end{proof}

\begin{lemma}
The union of an at most pseudo-enumerable strongly uniform family
of at most pseudo-enumerable sets is at most pseudo-enumerable.
\end{lemma}
\begin{proof}
Let $\Pa{X(i)}_{i \in N}$ be a pseudo-enumerable uniform family of at most
pseudo-enumerable subsets of $\K^n$, and $Z := \bigcup_{i \in N} X(i)$.
If $\K$ is \aminimal, then $Z$ is pseudo-finite.
Otherwise, \wloG $N$ is a \pN subset of~$\K$.
Let $Y \subset \K \times N$ and $F: Y \to \K^n$ be definable,
such that, for every $i \in N$, $F(Y_{[i]})  = X(i)$,
and $Y_i$ is a discrete definable closed subset of~$\K^{\geq 0}$.
For every $i, j \in N$, let
$Y(i,j) := Y_i \cap B(0; j)$.
Notice that $Y(i,j)$ is pseudo-finite for every $i, j$,
and
\[
Z = \bigcup_{\pair{i,j}\in N^2} F(Y(i,j) \times \set{i}).
\]
Since $N^2$ is pseudo-enumerable, by changing the family,
\wloG we can assume that each $X(i)$ is pseudo-finite.
Moreover, \wloG each $X(i)$ is of the form $Y(i) \times \set i$,
where $Y(i)$ is a pseudo-finite subset of $\K^{\geq 0}$.
As usual, for every $i \in N$, let $s(i)$ be the successor of $i$ in $N$.
Define $d(i) := s(i) - i$, and
Moreover, let $u(i) := \max(Y(i))$.
Define $g : Z \to \K$,
\[
g(y,i) = i + \frac{d(i)}{2 u(i)} y.
\]
Then, $g$ is injective, and $g(Z)$ is a closed discrete subset of $\K^{\geq 0}$.\end{proof}

\begin{lemma}
Let $D \subseteq \K^n$ be pseudo-finite.
Then, there exists $C \subseteq \K$ pseudo-finite,
and a definable bijection $f: D \bij C$.
Moreover, $C$ and $f$ can be defined uniformly from $C$:
that is, if $\Pa{D_i}{i\in I}$ is a definable family of pseudo-finite subsets
of~$\K^n$,
then there exists $C \subset \K^n \times I$ and $f: D \to C$ definable and
injective,  such that, for every $i \in I$,
$f(C_{[i]}) = D_{[i]}$.
Hence, every definable family of pseudo-finite sets is uniform.
\end{lemma}
\begin{proof}
Induction on $n$.
If $n=1$, take $D = C$.
If $n = 2$, let $X := \pi(D)$ and $Y := \theta(D)$,
where $\pi$ and $\theta$ are the projections onto the first and second
coordinates, respectively.
\Wlog, $0 = \min Y$; let $r:= \max Y$.
For every $x \in X$, let $s(x)$ be the successor of $x$ in $X$
(or $x + 1$ if $x = \max X$) and $d(x) := s(x) - x > 0$.

Define $f: D \to \K$,
\[
f(x,y) := x + \frac{d(x)}{2r} * y,
\]
and $C := f(D)$.
It is cleat that $f$ is injective, and, since $D$ is pseudo-finite,
$C$ is also pseudo-finite.

The case $n > 2$ follows easily by induction on~$n$.
\end{proof}

\textbf{Missing:} 
If $\K$ is saturated, every def. family of pseudo-enumerable sets is uniform.

\begin{lemma}
Every definable discrete subset $X$ of $\K^n$ is at most pseudo-enumerable.
\end{lemma}
\begin{proof}
If $\K^n$ is \aminimal, then $X$ is pseudo-finite.
Otherwise, let $N$ be a \pN subset of~$\K$.
By Lemma~\ref{lem:discrete-countable}, $X$ is a union
of a uniform increasing family of pseudo-finite sets.
As usual, $X = \bigcup_{i \in N} Y(i)$, where each $Y(i)$ is pseudo-finite.
Hence, $X$ is the union of a pseudo-enumerable uniform family of pseudo-finite
sets, and thus $X$ is at most pseudo-enumerable.
\end{proof}

\begin{corollary}
Let $X \subseteq \K^n$ be definable.
Then, $X$ is pseudo-enumerable iff there exists $D \subset \K^{n+1}$ closed
and discrete, such that $X = \Pi^{n+1}_n(X)$.
\end{corollary}
\begin{proof}
Assume $X \subseteq \K^n$ is pseudo-enumerable.
Let $N \subset \K$ be a \pN set, and $f: N \to X$ be a definable bijection.
Let $D \subset \K^{n+1}$ be the graph of~$f$.
Then, $D$ is discrete and $X = \pi(D)$.
\end{proof}

\begin{corollary}
Let $U \subseteq \K$ be definable and open.
Then, $U$ is the union of an at most  pseudo-enumerable family of open intervals.
\end{corollary}
\begin{proof}
Let $D$ be the set of centers of the connected components of~$U$.
Then, $D$ is discrete and definable, and hence at most pseudo-enumerable,
and $U$ is a union of open intervals indexed by~$D$.
\end{proof}

\subsection{Uniform families}
\begin{proviso}
In this subsection, we assume that $\K$ is definably complete and Baire.
\end{proviso}
We do not know if, in general, definable increasing union of~$\Fs$ (resp.,
meager) sets are~$\Fs$ (resp., meager).
We will give some sufficient conditions for this to be the case.

Remember that if we say that $(A_i)_{i\in I}$ is a definable family of subsets
of~$\K^n$, 
we mean that $I$ is a definable subset of~$\K^m$, for some~$m$,
and $\Afam := \bigcup_{i \in I} A_i \times \set i$ is a definable subset
of~$\K^{n+m}$.
Moreover, remember that $A \subset \K^n$ is an $\Fs$ iff it is the projection
of some definable closed subset of~$\K^{n+1}$ \cite{FS}.

\begin{definizione}
Let $(A_i)_{i\in I}$ be a definable family of subsets of~$\K^n$.
We say that $\Afam$ is a \intro{uniform family} of $\Fs$ sets if
there exists a definable set $C \subseteq \K^{n + 1} \times \K^{m'}$, for
some~$m'$, such that:
\begin{itemize}
\item each fiber $(C_t)_{t \in \K^{m'}}$ is closed;
\item for each $i \in I$ there exists $t \in \K^{m'}$ such that
$A_i = \Pi^{n+1}_n(C_t)$.
\end{itemize}
\end{definizione}
Notice that in the above definition, up to taking a smaller definable
family~$C'$, we can always impose the additional condition that,
for each $t \in \K^{m'}$, either $C_t$ is empty,
or $\Pi^{n+1}_n(C_t) = A_i$ for some $i \in I$.

\begin{definizione}
Let $(A_i)_{i\in I}$ be a definable family of subsets of~$\K^n$.
We say that $\Pa{A_i}_{i \in I}$ is a \intro{strongly uniform family}
of $\Fs$ sets if
and there exist $C \subset \K^{n+1} \times I$ definable, and
$F: C \to \K^n$ definable,
such that, for every $i \in I$, $F(C_{[i]})  = A_i$,
and $C_i$ is a  closed subset of~$\K^{n+1}.$.
\end{definizione}


\begin{remark}
Let $\Afam := (A_i)_{i\in I}$ be a definable family of subsets of~$\K^n$.
Then, $\Afam$ is a uniform family of $\Fs$ sets iff
there exists $k, m' \in \Nat$ and $D \subseteq \K^{n+k} \times \K^{m'}$,
such that:
\begin{itemize}
\item each fiber $(D_t)_{t \in \K^{m'}}$ is closed;
\item for each $i \in I$ there exists $t \in \K^{m'}$ such that
$A_i = \Pi^{n + k}_n(D_t)$.
\end{itemize}
\end{remark}
\begin{proof}
The ``only if'' direction is obvious.
For the ``if'' direction, let $(D_t)_{t \in \K^{m'}}$ as in the Remark.
For each $t \in \K^{m'}$, define
\[
C_t := \cl{\bigcup_{0 < r \in \K}
\Pi^{n+k}_n \Pa{D_t \cap  \clB(0;r) } \times \set r}.
\qedhere \]
\end{proof}

Given $X \subseteq \K^n$ definable, and $\K' \succeq \K$,
denote by $X(\K')$ the interpretation in $\K'$ of (the formula defining) $X$;
similarly, we denote $\Afam(\K') := \Pa{A_i(\K')}_{i \in I(\K')}$.
\begin{lemma}
Let $\Afam := (A_i)_{i \in I}$ be a definable family of subsets of~$\K^n$.
\Tfae:
\begin{enumerate}
\item $\Afam$ is a uniform family of $\Fs$ sets;
\item for every $\K' \succeq \K$, and for each $i \in I(\K')$,
$A_i(\K')$ is an $\Fs$ set;
\item there exists $\K' \succeq \K$ such that $\K'$ is $\omega$-saturated,
and for each $i \in I(\K')$,
$A_i(\K')$ is an $\Fs$ set.
\end{enumerate}
\end{lemma}
\begin{proof}
Standard compactness argument.
For instance, let us show that not~(1) implies not~(3).
Assume that $\Afam$ is not a uniform family of $\Fs$ sets.
Let $\Psi$ be the set of formulae
$\psi(x, y)$ (without parameters), where $x$ and $y$ are of length
$n + 1$ and $m'$ respectively ($m'$~varies),
such that, for every $c \in \K^{m'}$, $\psi(\K^{n+1}, c)$ is closed.
By hypothesis, for every $\phi \in \Psi$,
there exists $i \in I$ such that, for every~$c$,
$A_i \neq \Pi^{n+1}_n\Pa{ \phi(\K^{n+1}, c)}$.
The above condition determine a partial type~$q(i)$,
with a finite number of parameters (that is, the parameters used to
define~$\Afam$).
Let $\K' \succ \K$ be $\omega$-saturated, and let $i_0 \in I(\K')$
satisfying~$q$; we claim that $A_{i_0}(\K')$ is not an~$\Fs$.
If, for contradiction, $A_{i_0}(\K')$ were an~$\Fs$,
there would exists a definable closed set $C \subset \K^{n+1}$
such that $A_{i_0}(\K') = \Pi^{n+1}_n(C)$.
Let $c \in \K'^{m'}$ and $\phi_0(x,c)$ be the formula defining~$C$.
Define $J := \set{y \in \K^{m'}: \phi_0(\K^{n+1}, y) \text{ is closed}}$,
and $\phi(x, y) \equiv \Pa{\phi_0(x,y) \et y \in J}$.
Then, $\phi \in \Psi$, but $A_{i_0}(\K') = \Pi^{n+1}_n \Pa{\phi(\K^{n+1}, c)}$,
a contradiction.
\end{proof}

\begin{lemma}
Let $\Afam := (A_i)_{i \in I}$ be a family of subsets of~$\K^n$, definable
with parameters~$\av$.
Then, \tfae:
\begin{enumerate}
\item $\Afam$ is a strongly uniform family of $\Fs$ sets;
\item There exits $\K' \succeq \K$ such that $\K'$ is $\omega$-saturated and,
for every $i \in I(\K')$, there exists $C_i \subseteq \K^{n + 1}$ which is closed
and definable with parameters $\av i$, and such that $\Pi^{n+1}_n(C_i) = A_i$;
\item
For every $\K' \succeq \K$ and
for every $i \in I(\K')$, 
there exists $C_i \subseteq \K^{n + 1}$ which is closed
and definable with parameters $\av i$, and such that $\Pi^{n+1}_n(C_i) = A_i$.
\end{enumerate}
\end{lemma}

\begin{example}
Let $\K$ be \aminimal.
Then, every definable family $(A_i)_{i \in I}$ of $\Fs$ subsets of
$\K^n$ is strongly uniform.
\end{example}
\begin{proof}
Since each $\Fs$ set $X$ is a union of at most $n+1$ locally closed sets
$X_0, \dotsc, X_n$, and each $X_i$ is definable in a uniform way from~$X$,
\wloG we can assume that each $A_i$ is locally closed.
For each $x \in A_i$, let
\[
r_i(x) :=
\sup\set{s \in \K^{>0}: A_i \cap \clB(x;s) \text{ is compact}},
\]
and $U_i := \bigcup_{x \in A_i} B(x; r_i(x)/2)$.
Notice that each $U_i$ is open, and $A_i$ is closed in~$U_i$.
For each $i \in I$ and $r \in \K^{> 0}$, define
\[
C(i,r) := \cl{A_i} \cap \set{x \in U_i: d(x, U_i^\complement) \geq r}.
\]
Notice that each $C(i,r)$ is closed, and $A_i := \bigcup_{r> 0} C(i,r)$.
Finally, let $D(i) := \cl{\bigcup_{r> 0} C(i,r) \times \set r}$.
Each $D(i)$ is closed, and $A_i = \Pi^{n+1}_n D(i)$.
\end{proof}



\begin{lemma}\label{lem:uniform-U-Fs}
Let $\Afam := (A_i)_{i \in I}$ be a strongly uniform family of $\Fs$ sets,
and assume that either $I$ is at most pseudo-enumerable,
or that $\Afam$ is increasing.
Then, $D := \bigcup_{i \in I} A_i$ is~$\Fs$.
\end{lemma}
\begin{proof}
We distinguish 2 cases: either $\K$ is \aminimal, or not.
If $\K$ is \aminimal, the conclusion follows from
Lemma~\ref{lem:amin-U-Fs} and Cor.~\ref{cor:finite-constructible}.

Otherwise, let $J$ be a \pN subset of~$\K$.
For every $i \in I$ and $r \in \K^{>0}$,
let $E(i,r) :=  \set{x \in C_i: \abs x \leq r}$.
Notice that each $E(i,r)$ is compact, and that
$A_i = \bigcup_r \theta(E(i,r))$, where $\theta := \Pi^{n + 1 + m'}_n$.
Then, $D = \bigcup_{\pair{i,j} \in I \times J} \theta\Pa{E(i, j)}$.
Since $I \times J$ is pseudo-enumerable, and each $\theta\Pa{E(i, j)}$ is
compact, $D$~is an $\Fs$ set.
\end{proof}

Remember that an $\Fs$ is meager iff it has empty interior.
\begin{definizione}
Let $\Afam := (A_i)_{i \in I}$ be a definable family of subsets of~$\K^n$,
and $\pi := \Pi^{n+1}_n$.
\begin{itemize}
\item $\Afam$~is a \intro{uniform family} of meager sets, if there exists 
a definable family $(C_t)_{t \in \K^m}$ of closed subsets of $\K^{n+1}$,
such that
\begin{enumerate}
\item for every $t \in \K^m$, $\pi(C_t)$ has empty interior;
\item and for every $i \in I$ there exists $t \in \K^m$
such that $A_i \subseteq \pi(C_t)$.
\end{enumerate}
\item $\Afam$~is a \intro{strongly uniform family} of meager sets, if there exists a definable family $(C_i)_{i \in I}$ of closed subsets of $\K^{n+1}$,
such that,
for every $i \in I$, $\pi(C_i)$ has empty interior,
and $A_i \subseteq \pi(C_i)$.
\item $\Afam$~is a \intro{uniform family} of \ao sets, if there exists 
a definable family $(C_t)_{t \in \K^m}$ of closed subsets of $\K^{n+1}$,
and a definable family $(U_s)_{s \in  \K^l}$ of open subsets of $\K^n$,
such that
\begin{enumerate}
\item for every $t \in \K^m$, $\pi(C_t)$ has empty interior;
\item and for every $i \in I$ there exists $t \in \K^m$ and $s \in \K^l$,
such that $A_i \sdiff U_s \subseteq \pi(C_t)$.
\end{enumerate}
\item 
$\Afam$~is a \intro{strongly uniform family} of \ao sets,
if there exists a definable family $(U_i)_{i \in I}$ of open subsets
of~$\K^n$, and a definable family $(C_i)_{i \in I}$ of closed subsets of
$\K^{n+1}$, such that,
for every $i \in I$, $\pi(C_i)$ has empty interior,
and $A_i \Delta U_i\subseteq \pi(C_i)$.
\end{itemize}
\end{definizione}

\begin{question}
Are uniform family of $\Fs$ sets (resp.\ meager, resp.\ \ao sets)
strongly uniform?
\end{question}
If $\K$ has definable choice, the answer is positive.

Reasoning as for the $\Fs$ case, we can prove the following results:

\begin{lemma}
Let $\Afam := (A_i)_{i \in I}$ be a definable family of subsets of~$\K^n$.
\Tfae:
\begin{enumerate}
\item $\Afam$ is a uniform family of meager (resp.,~\ao) sets;
\item for every $\K' \succeq \K$, and for each $i \in I(\K')$,
$A_i(\K')$ is a meager (resp.,~\ao) set;
\item there exists $\K' \succeq \K$ such that $\K'$ is $\omega$-saturated,
and for each $i \in I(\K')$,
$A_i(\K')$ is a meager (resp.,~\ao) set.
\end{enumerate}
\end{lemma}
\begin{proof}
Let us prove $(3 \Rightarrow 1)$ for meager sets.
Assume not.
By adding a finite number of constants to the language if necessary, we can
assume that either $\K$ is \aminimal, or there exists $N$ \pN subset of $\K$
definable \emph{without parameters}.
Let $\Psi$ be the set of formulae (without parameters) $\phi(\x,\y)$,
such that, for every $\y$, $\phi(\K, \y)$ is a closed subset of~$\K^{n + 1}$,
and $\pi(\phi(\K, y))$ has empty interior, where $\pi := \Pi^{n+1}_n$.
By hypothesis, for every $\phi \in \Phi$, there exists $i \in I$,
such that, for every $\cv$, $A_i \nsubseteq \pi(\phi(\K^n, \cv))$
Let $q(i)$ be the partial type 
$\set{\forall \cv\ A_i \nsubseteq \pi(\phi(\K^n, \cv)) : phi \in \Phi}$.
$q$~has a finite number of parameters (the parameters used in the definition
of~$\Afam$), and is consistent;
therefore, $q$ has a realization $i_0$ in~$\K'$.
Since $A_{i_0}$ is an~$\Fs$, 
there exists $D \subseteq \K'^{n+1}$ closed and
definable with parameters~$\cv$, such that $A(i_0) \subseteq \pi(D)$, and $\interior\Pa{\pi(D)} = \emptyset$.
Hence, $D = \phi(\K, \cv)$ for some $\phi \in \Phi$, contradiction.
\end{proof}

\begin{example}
Let $\K$ be \aminimal.
Then, every definable family of meager (resp.,~\ao) sets is strongly uniform.
\end{example}

\begin{lemma}
Let $\Afam :=  (A_i)_{i \in I}$
be a strongly uniform family of meager sets (resp.,~\ao),
and assume that either $I$ is at most pseudo-enumerable,
or that $\Afam$ is increasing.
Then, $D := \bigcup_{i \in I} A_i$ is meager (resp.,~\ao).
\end{lemma}
\begin{proof}
If $\K$ is \aminimal, use Lemma~\ref{lem:amin-U-meager}.
Otherwise, it suffices to treat the case when $I$ is pseudo-enumerable,
and \wloG we can reduce to the case when $I$ is a \pN subset of~$\K$.

If $\Afam$ is a strongly uniform family of meager set,
Let $(C_i)_{i \in I}$ be a definable family of closed subsets of 
$\K^{n+1}$, such that, for every$i \in I$, $A_i \subseteq \pi(C_i)$, 
and $\pi(C_i)$ has empty interior,
where $\pi := \Pi^{n+1}_n$.
Then,
\[
D \subseteq \bigcup_{i \in I} \pi(C_i) = \bigcup_{(i,j) \in I \times J}
\Pi^{n+1}_n (C_i \cap \cl{B(0; j)}),
\]
and the latter is the union of a pseudo-enumerable family of nowhere dense
sets. 

If $\Afam$ is a strongly uniform family of \ao sets, let $(U_i)_{i \in I}$ be a
definable family of open sets, such that $(A_i \Delta U_i)_{i \in I}$ is a
strongly uniform family of meager sets.
Let $V := \bigcup_{i \in I} U_i$.
Notice that
\[
D \Delta V = (\bigcup_i A_i) \Delta (\bigcup_j U_j)
\subseteq \bigcup_i (A_i \Delta U_i),
\]
and, by the previous point, the latter is a meager set.
\end{proof}

\begin{corollary}
Assume that either $\K$ be $\omega$-saturated and has definable choice.
or $\K$ is \aminimal.
Let $(A_i)_{i \in I}$ be a definable family.
Assume that either $I$ is at most pseudo-enumerable, or that
or $(A_i)_{i \in I}$ is increasing.
Assume moreover that each $A_i$ is an $\Fs$ (resp.,~meager, resp.~\ao).
Then, $\bigcup_{i \in I} A_i$ is an $\Fs$ (resp.,~meager, resp.~\ao).
\end{corollary}
\begin{proof}
The hypothesis on $\K$ implies that 
definable families of $\Fs$ (resp.,~meager, resp.~\ao) sets are strongly
uniform.
\end{proof}

\subsection{Is every definably complete structure Baire?}\label{sec:enumerable}

We are not able to answer the above question, even if we conjecture that the
answer is positive.

However, we shall give a partial results, where we prove that a strong failure
of the Baire property is not possible in a definably complete structure.

Note that if $\K$ is pseudo-enumerable, then $\K$ is not Baire.

\begin{thm}
$\K$ is not pseudo-enumerable.
\end{thm}

The proof can be understood better if one considers the case when $\K =
\Real$, and considers it as a proof of the fact that $\Real$ is not countable.

We need some preliminary results and definitions.

For every $I = (a,b) \subseteq (0,1)$ open interval, and $0 < \lambda < 1/2$,
define $\abs I := b -a $, and
$C_\lambda(I) := [a + \lambda(b-a), b -\lambda(b-a) ]$.
Note that $C_\lambda(I)$ is a closed subset of~$I$, and that
$\abs{C_\lambda(I)} = (1 - 2\lambda) \abs I > 0$.

Fix such an open interval~$I$, and
let $J \subseteq I$ be a closed interval, and $0 < d < 1$, such that $\abs J
\geq d \abs I$.
Let $x \in I$, and consider the two intervals $I'_1 :=(a, x)$ and $I'_2 := (x,
b)$.
\begin{remark}\label{rem:intersection-intervals}
There exists $k \in \set{1,2}$ such that
\[
\frac{\abs{C_\lambda(I'_k) \cap J}}{\abs{I'_k}} \geq (d - 2 \lambda).
\]
\end{remark}
\begin{proof}
Let call the elements of $J$ ``good points''.
The fraction of good points in $I$ is at least~$d$.
Hence, for at least one~$k$, the fraction of good points in $I'_k$ is at
least~$d$.
Only a fraction of $2 \lambda$ points are not in $C_\lambda(I'_k)$,
and therefore at lest a fraction of $d -2 \lambda$ points are
good and in $C_\lambda(I'_k)$.
\end{proof}

Now, suppose, for contradiction, that $\K$ is pseudo-enumerable.
Hence, $X := [0, 1]$ is also pseudo-enumerable;
let $N$ be a \pN subset of~$\K$, and $f: N \to X$
be a definable surjective function.
\Wlog, $1$~is the minimum of~$N$.
We denote by $x_n := f(n)$, for every $n \in N$.


For every $n \in N$, let $s(n)$ be the successor of $n$ in~$N$.

Note that $N$ must be unbounded (otherwise, it would be pseudo-finite).
Let $\phi: \K^{\geq 0} \to (0,1]$ be an order-reversing homeomorphism, and
$D := \phi(N)$; denote by $d_n := \phi(n)$.
Let $\lambda_1 := (1 - d_1) / 4$, and, for every $n \in N$, let $\lambda_{s(n)} := (d_n - d_{s(n)}) / 4 > 0$.

For every $n \in N$, let $N_{\leq n} := \set{n' \in N: n' \leq n}$.
Let $X(n) := f(N_{\leq n})$.
Note that $X_{s(n)} = X(n) \cup \set{x_{s(n)}}$.
Note also that $N_n$ is pseudo-finite, and thus $X$ is also pseudo-finite.
Let $U(n) := X \setminus X(n)$: note that $U(n)$ is open, and, since $X(n)$ is
pseudo-finite, $U(n)$ is a pseudo-finite union of disjoint open intervals:
\[
U(n) = \bigsqcup_{i \in A(n)} I_{i,n},
\]
where each $I_{i,n}$ is a non-empty open interval, and $A(n)$ is
pseudo-finite.
Define
\[
F(n) := \bigcup_{i \in A(n)} C_{\lambda_n}(I_{i,n}).
\]
$F(n)$ is a pseudo-finite union of closed intervals, and therefore it is
definable and closed.

We claim that $F := \bigcap_n F(n) \neq \emptyset$.
If it is so, since $F(n) \subset U(n)$, we have that $X \neq \bigcup_n X(n)$,
and we have a contradiction.

For every $n \in N$, let $G(n) := \bigcap_{m \leq n} F(m)$.
Note that $G(n)$ is a definable decreasing family of \dcompact sets, and that
$F = \bigcap_n G(n)$.
Thus, to prove that $F$ is non-empty, it suffices to prove that each $G(n)$ is
non-empty.
We shall prove this by induction on~$n$; however, we will need to prove a
stronger statement.

\begin{lemma}
For each $n \in N$,
\begin{enumerate}
\item $G(n)$ is a pseudo-finite union of disjoint closed intervals, with
non-empty interior, $(J_{n,j})_{j \in B(n)}$;
\item the intervals $J_{n,j}$ are the definably connected components of~$G(n)$;
\item each $J_{n,j}$ is contained in a unique open interval $I_{n,i}$, where
$i = i_{n,j}$ depends in a definable way from $(n,j)$;
\item there exists at least one $j \in B(n)$ such that
\[
\frac {\abs{J_{n,j}}}{\abs{I_{n,i_{n,j}}}} \geq d_n.
\]
\end{enumerate}
\end{lemma}
Note that the lemma implies that each $G(n)$ is non-empty.
\begin{proof}
If $n = 1$, then, by definition, $A(1) = B(1)$, and every interval $J_{1,i}$
satisfies $J_{1, i} = C_{\lambda_1}(I_{1, i_{1,j}})$, and
$\abs{J_{1,j}} = (1 - 2\lambda_1) \abs I_{1, i_{1,j}}$.

Suppose that the statement is true for~$n$; we want to prove it for $s(n)$.
Let $x := x_{s(n)}$, $\lambda := \lambda_{s(n)}$, 
and let $J := J_{n,j}$  and $(a,b) := I := I_{n, i_{n,j}}$ be the
intervals satisfying the condition for~$n$.
If $x \notin I$, then, since $d_{s(n)} < d_n$, the same $J$ and $I$ satisfy
the condition for $n + 1$.

If instead $x \in I$, let $I'_1 := (a,x)$ and $I'_2 := (x,b)$.
Note that $I$ is substituted by the pair $(I'_1, I'_2)$ in the expression of
$U_{s(n)}$ as disjoint union of open intervals.
Let $L'_k := C_{\lambda}(I'_k)$, $k = 1, 2$.
By Remark~\ref{rem:intersection-intervals}, at least one of the $L'_k$
satisfies
\[
\frac{\abs{L'_k \cap J}} {\abs{I'_k}} \geq d_n - 2 \lambda > d_{n+1}.
\]
Since $J'_k := L'_k \cap J$, $k = 1, 2$, are definably connected components of $G(n+1)$,
we are done.
\end{proof}

\begin{lemma}\label{lem:Q}
Assume that $\K$ is not Baire (but it is still definably complete and expands
a field). 
Then, there exists a dense pseudo-enumerable subset $Q \subset \K$.
\end{lemma}
Note that, by the previous theorem, $Q$~has empty interior.
\begin{proof}
$\K$~is not \aminimal, and therefore there exists
$D$ \pN subset of~$\K$.

Let $\Pa{X(n)}_{n \in D}$ be a definable increasing family of \dcompact
nowhere-dense subsets of~$[0,1]$, such that $[0,1] = \bigcup_{n \in D} X(n)$.
For every $n \in D$, let $Y(n) \subseteq X(n)$ be the set of isolated points
of~$X(n)$, and $Z(n) \subset [0,1]$ be the set of centres of the definable
connected components of $[0,1] \setminus X(n)$.
Note that both $Y(n)$ and $Z(n)$ are discrete; thus, $Y(n)$ and $Z(n)$ are
(at most) pseudo-enumerable.
Let $Q' :=  \bigcup_{n \in D} \Pa{Z(n) \cup Y(n)}$.
Since $Q'$ is the union of a pseudo-enumerable family of pseudo-enumerable
sets, $Q'$ is pseudo-enumerable.
Moreover, $Q'$ is dense in~$[0,1]$.
Using~$Q'$, it is trivial to obtain $Q \subseteq \K$ pseudo-enumerable and
dense in~$\K$.
\end{proof}

\begin{corollary}
If $\K$ is not Baire, then, for every $n \in \Nat$, $\K^n$ has a
pseudo-enumerable basis of open sets.
\end{corollary}
\begin{proof}
Let $I$ be a \pN subset of~$\K$, such that $0 \notin I$, 
and $(x_i)_{i \in I}$ be a dense pseudo-enumerable subset of~$\K^n$.
Then, $\set{B(x_i; 1/i): i \in I}$ is a pseudo-enumerable basis of open sets
for~$\K^n$.
\end{proof}

\subsection{Fixed point theorems}

\begin{lemma}[Banach fixed point theorem]
Let $X \subseteq \K^n$ be a definable non-empty closed set, and
$f: X \to X$ be a definable map.
Assume that:
\begin{enumerate}
\item either $X$~is \dcompact, and, for every $x \neq y \in X$, 
$\abs{f(x) - f(y)} < \abs{x - y}$;
\item or there exists $C \in \K$ such that $0 \leq C < 1$ and, 
for every $x, y \in Y$, $\abs{f(x) - f(y)} \leq C \abs{x - y}$.
Then, $f$~has a unique fixed point.
\end{enumerate}
\end{lemma}
\begin{proof}
(1) Let $g: X \to \K$, $g(x) := \abs{x - f(x)}$, and let
$d := \min(g)$.
If $f$ has no fixed point, then $d > 0$;
let $x_1 \in X$ such that $g(x_1) = d$, and $x_2 := f(x_1)$.
Then, $\abs{f(x_2) - x_2} < \abs {x_2 - x_1} = c$, absurd.

(2) Choose $x_1 \in X$; let $x_2 := f(x_1)$, 
$r_0 := \abs{x_1 - x_1}$, $r := r_0/(1-c)$, and $Y := X \cap \clB(x_0, r)$.
Then, $Y$ is \dcompact and $f(Y) \subseteq Y$.
Hence, by (1), $f$ has a fixed point in~$Y$.
\end{proof}

\begin{conjecture}
Brower fixed point theorem.
Kakutani fixed point theorem.
\end{conjecture}

\subsection{Baire structures}
Let $\K$ be definably complete and \textbf{Baire}.

If $X$ and $Y$ are subsets of $\K^n$,
then $X - Y := \set{x - y: x \in X, y \in Y}$.

\begin{lemma}[Pettis' Theorem]
Let $A \subseteq \K^n$ be definable and  \ao.
If $A$ is non-meager, then $A - A$ contains a non-empty open neighbourhood
of~$0$.
If $\K$ is \iminimal with DSF
(see \S\ref{sec:i0min}) and $A - A$ is non-meager, then $A$ is non-meager.
\end{lemma}
\begin{proof}
Minor variation of~\cite[Thm.~4.8]{oxtoby}.
Let $A = U \sdiff P$, where $U$ is open and definable, and $G$ is meager.
$A$~is non-meager iff $U$ is non-empty.
If $A$ is non-meager, let $B \subseteq U$ be a non-empty open ball, of radius $\delta > 0$.
For any $x \in \K^n$, we have
\begin{multline*}
(x + A) \cap A = \Pa{(x + U) \sdiff (x + P)} \cap (U \sdiff P) = \\
\Pa{(x + U) \cap U} \sdiff \Pa{(x + U) \cap P} \sdiff \Pa{(x + P) \cap U}
\sdiff \Pa{(x + P) \cap P} 
\supseteq [(x + B) \cap B] - [P \cup (x + P)].
\end{multline*}
If $\abs x < \delta$, the right member represents a non-empty open set,
minus a meager set; it is therefore non-empty.
Thus, for every $x \in B(0;\delta)$, $(x + A) \cap A$ is non-empty,
and therefore $x \in A - A$.

Conversely, if $A - A$ is non-meager, then, by Thm.~\ref{thm:imin}, 
$\dim (A - A) = 1$; therefore, by Lemma~\ref{lem:dim-function},  $\dim A = 0$, 
and thus $A$ is meager.
\end{proof}

Define $F: \K^4 \to \K$ as $F(x_1, x_2, y_1, y_2) := (x_1 - x_2) / (y_1 - y_2)$
if $y_1 \neq y_2$, and $0$ otherwise.

\begin{corollary}
Let $A \subseteq \K$ be definable and \ao.
If $A$~is non-meager, then $F(A^4) = \K$.
If $\K$ is \iminimal with \DSF (see \S\ref{sec:i0min}) and $F(A^4)$ is
non-meager, then $A$ is non-meager.
\end{corollary}
\begin{proof}
If $A$ is non-meager, then $A - A$ contains an open neighbourhood of~$0$, 
and therefore $F(A^4) = \K$.
The converse is proved as in the previous lemma.
\end{proof}


\section{I-minimal and constructible structures}\label{sec:i0min}
As usual, $\K$ is a definably complete structure, expanding a field.

\subsection{I-minimal structures}

\begin{definizione}
$\K$ is \intro{\iminimal{}} if, for every unary definable set~$X$, 
if $X$ has empty interior, then $X$ is nowhere dense.
\end{definizione}
See also~\cite{miller05} for the case when $\K$ is an expansion of~$\Real$.

\begin{lemma}
If $\K$ is \iminimal, then it is Baire
\end{lemma}
\begin{proof}
By Lemma~\ref{lem:Q}, if $\K$ were not Baire, it would contain a definable
dense and co-dense subset~$Q$.
However, $Q$ would have empty interior, but it would not be nowhere dense.
\end{proof}

\begin{examples}
If $\K$ is locally o-minimal, then it is \iminimal (trivial).

If $\K$ is d-minimal, then it is \iminimal.
In fact, by definition of d-minimality if $X \subseteq \K$ is definable and
with empty interior, then $X$ is a finite union of discrete sets~$X_1, \dotsc,
X_n$.
Every discrete subset of a definably complete structure is nowhere dense.
Thus, $X$~is a finite union of nowhere dense sets,
and thus it is nowhere dense.
\end{examples}

\begin{thm}\label{thm:imin}
The following are equivalent:
\begin{enumerate}
\item\label{en:i-imin-1}
$\K$ is \iminimal;
\item\label{en:i-imin-n} 
for every $n\in \Nat$, if $X$ is a definable subset of
$\K^n$ with empty interior, then $X$ is nowhere dense;
\item\label{en:i-bd}
for every definable set~$X$, $\bd(X)$ has empty interior;
\item\label{en:i-dim-1}
for every definable $X \subseteq \K$, $\dim X = \dim \cl X$;
\item\label{en:i-dim-n}
for every definable~$X$, $\dim X = \dim \cl X$;
\item\label{en:i-f} 
if $U \subseteq \K^n$ is definable and open, and $f : U \to \K$ is definable, then $\Dis(f)$ is nowhere dense;
\item\label{en:i-B} 
for every $n, m \in \Nat$, if $A \subseteq \K^{n+m}$ is definable,
then $\B_n(A)$~is nowhere dense, where
\[
\B_n(A) := \set{x \in \K^n: \cll(A)_x \setminus \cll(A_x) \neq \emptyset}.
\]
\item\label{en:i-fr}
for every $n, m \in \Nat$, if $A \subseteq \K^{n+m}$ is definable, then the
set
\[
\set{x \in \K^n: (\fr A)_x \neq \fr(A_x)}
\]
is nowhere dense;
\item\label{en:i-meager-1} 
$\K$ is Baire, and, for every definable $X \subseteq \K$, either $X$ has interior, or it is meager;
\item\label{en:i-meager-n}
$\K$ is Baire, and every definable set either has interior, or it is meager;
\item\label{en:i-dim-union-1}
for all definable~$A, B \subseteq \K$,
$\dim (A \cup B) = \max \set{\dim A,  \dim B}$;
\item\label{en:i-dim-union-n}
for all definable~$A, B \subseteq \K^n$,
$\dim (A \cup B) = \max \set{\dim A,  \dim B}$;
\item\label{en:i-dim-fiber}
for all definable~$A \subseteq \K^n$, if $\dim A = d$,
then $\set{x \in \K^d: \dim A_x > 0}$ is nowhere dense;
\item\label{en:i-dim-fiber-2}
let $d, k, m, n \in \Nat$, with $k \leq n$ and $d \leq m$;
let $A \subseteq \K^{n + m}$ be definable, and $\dim A \leq d + k$;
define $C:= \set{x \in \K^n: \dim A_x \geq d}$; then, $\dim(C) \leq k$;
\item\label{en:i-countable-1}
any pseudo-enumerable union of subsets of $\K$ with empty interior has empty
interior;
\item\label{en:i-countable-n}
for every $d \leq n \in \Nat$, any pseudo-enumerable union of subsets of
$\K^n$ of dimension less or equal to $d$ has dimension less or equal to~$d$;
\end{enumerate}
Moreover, if $\K$ is \iminimal, then:
\begin{enumerate}[I]
\item\label{en:i-meager}
every meager set is nowhere dense;
\item\label{en:i-increasing}
for every $d \leq n \in \Nat$, any increasing definable union of subsets of
$\K^n$ of dimension less or equal to $d$ has dimension less or equal to~$d$;
\item\label{en:i-monotonicity}
if $U \subseteq \K$ is open and definable, and $f: U \to \K$ is definable,
then there exists $D \subseteq \K$ definable, closed and with empty
interior, such that, for every definably connected component $I$ of
$U \setminus D$,  $f \rest I$ is continuous, and either constant or strictly
monotone; 
\item\label{en:i-continuous}
if $U \subseteq \K^n$ is open and definable, and $f: U \to \K$ is definable,
then there exists $D \subseteq \K$ definable, closed and with empty
interior, such that the restriction $f \rest U \setminus D$ is continuous.
\end{enumerate}
\end{thm}
The proof is postponed to Section~\ref{sec:imin-proof};
cf.~\cite[Main Lemma]{miller05}.

\begin{example}
Let $(M', M)$ be o-minimal structures (expanding a field), such that $M$ is an
elementary substructure of $M'$ and it is dense in~$M'$.
Thus, the structure $N := (M, M')$ has o-minimal open core.
Therefore, if $X \subseteq N$ is meager, then $X$ is nowhere dense.
However, $N$ is not \iminimal, because $M$ is a definable dense subset of $N$
such that $\cl M = N$ (thus, clause~\ref{en:i-meager} in Thm.~\ref{thm:imin}
does not imply \iminimality).
\end{example}

\begin{lemma}
$\K$ is locally o-minimal iff it is \aminimal and \iminimal.
\end{lemma}
\begin{proof}
The ``only if'' direction is clear.
Let us prove the ``if'' direction.
Let $X \subseteq \K$ be definable and with empty interior.
By \iminimality, $X$ is nowhere dense.
By \aminimality, $X$ is pseudo-finite.
Thus, $\K$ is locally o-minimal.
\end{proof}

\begin{corollary}\label{cor:i-o-min}
Assume that $\K$ is \iminimal.
Then, $\K$ is o-minimal iff every definable discrete set is finite.
\end{corollary}

\begin{lemma}
The following are equivalent:
\begin{enumerate}
\item\label{en:lm-lm} $\K$ is locally o-minimal; 
\item\label{en:lm-dim-fr-1} for every $X \subseteq \K$ definable and
non-empty $\dim(\fr X) < \dim X$;
\item\label{en:lm-dim-fr-n} for every $X \subseteq \K^n$ definable and
non-empty $\dim(\fr X) < \dim X$.
\end{enumerate}
\end{lemma}
\begin{proof}
($\ref{en:lm-dim-fr-n} \Rightarrow \ref{en:lm-dim-fr-1}$) is obvious.

($\ref{en:lm-dim-fr-1} \Rightarrow \ref{en:lm-lm}$).
First, let us prove that $\K$ is \aminimal.
Let $X \subseteq [0,1)$ definable, discrete and closed in $[0,1)$.
We have to prove that $X$ is pseudo-finite.
If not, $1$ is an accumulation point for~$X$.
Thus, $\fr X = \set 1$, and therefore $\dim(\fr X) = \dim X = 0$, absurd.

Then, we prove that $X$ is locally o-minimal.
It suffices to prove that if $X \subseteq \K$ has dimension~$0$, then it is
pseudo-finite.
Since $\dim X = 0$, $\fr X$ is empty, and therefore $X$ is closed; thus, $X$
is nowhere dense, and hence, by \aminimality, $X$ is pseudo-finite.

($\ref{en:lm-lm} \Rightarrow \ref{en:lm-dim-fr-1}$) is obvious, because, if
$A \subseteq \K$ is definable, then, if $\dim A =0$, then $A$ is
pseudo-finite, and hence closed, and hence $\fr A = \emptyset$; if instead
$\dim A = 1$, then $\fr A$ has empty interior, and thus $\dim(\fr A) \leq 0$.

($\ref{en:lm-lm} \Rightarrow \ref{en:lm-dim-fr-n}$).
Let $A \subseteq \K^n$ be definable, with $\dim A = d \geq 0$.
We have to prove that $\dim(\fr A) < d$.
If $n = 1$, we have already proved it above.
We proceed by induction on~$n$: we assume we have proved the conclusion for
every $n' < n$, and we want to prove it for~$n$.
If $d = n$, then $\fr A$ has empty interior, and thus $\dim(\fr A) < n$, and we
are done.
Thus, we can assume $0 \leq d < n$.
Assume, for contradiction, that $\dim(\fr A) \geq d$.
\Wlog, $U := \Pi^n_d(\fr A)$ has non-empty interior.
Since $\K$ is \iminimal, the set
\[
\set{u \in \K^d: (\fr A)_u \neq \fr(A_u)}
\]
is nowhere dense; thus, there exists $V \subseteq U$ open, definable and
non-empty, such that, for every $v \in V$, $(\fr A)_u = \fr(A_u)$.
By \iminimality, and since $\dim A = d$,
$\dim(A_u) = 0$ outside a nowhere dense set; thus, after shrinking~$V$, we can
assume that $\dim(A_u) = 0$ for every $u \in V$.
Thus, by the inductive hypothesis, $\fr(A_u)$ is empty for every $u \in V$,
contradicting the fact that $V \subseteq \Pi^n_d(\fr A)$.
\end{proof}
 Since there do exists \iminimal structures that are not locally o-minimal
(\eg, d-minimal not o-minimal expansions of the real field), we have that for
some \iminimal structure there is some definable non-empty set $X$ such that
$\dim(\fr X) = \dim X$ (notice that if $\K$ is \iminimal, then $\dim(\fr X)
\leq \dim X$, because $\dim X = \dim(\cl X) = \max\set{\dim(X), \dim(\fr X)}$).

\begin{proviso}
For the remainder of this subsection, we will assume that $\K$ is \iminimal.
\end{proviso}

\begin{lemma}\label{lem:imin-C1}
Let $f : U \to \K$ be definable, where $U \subseteq \K^n$ is open and definable.
Then, for every $p \in \Nat$, there exists $D \subset U$ closed, definable and
nowhere dense, such that $f$ is $C^p$ on $U \setminus D$.
\end{lemma}
\begin{proof}
If $p = 0$, we have already proved it in
Thm.~\ref{thm:imin}(\ref{en:i-continuous}).
If one proves the case $p = 1$, then an easy induction
on~$p$ proves the conclusion for every~$p$.

The case $n = 1$ can be done as for o-minimal structures
\cite[Proposition~7.2.5]{vdd}.

The case $n > 1$ can be done as in~\cite[Thm.~3.3]{miller05} (note that Miller
uses instead Lebesgue differentiability theorem for the case $n = 1$).
\end{proof}

\begin{lemma}
Let $d \leq n$, $A \subseteq \K^{n}$ be definable, $\pi := \Pi^{n}_d$, and
\[
Z := Z(A) := \set{a \in A: \exists U \text{ neighbourhood of } A:
\pi(A \cap U) \text{ is nowhere dense}}.
\]
Then, $Z$ is a definable open subset of~$A$, and $\pi(Z(A))$ is nowhere dense.
\end{lemma}
\begin{proof}
Follows immediately from Lemma~\ref{lem:bad-1}.
\end{proof}

\begin{definizione}
We define \intro{$\Pi$-good} sets as in~\cite[\S7]{miller05}.
That is, we say that a definable set $A \subseteq \K^{n+m}$ is $\pi$-good
(where $\pi := \Pi^{n + m}_m$) if:
\begin{itemize}
\item $\dim A = m$;
\item $\pi A$ is open;
\item $\pi(A \cap U)$ has interior for every $a \in A$ and open neighbourhood
$U$ of~$a$;
\item for all $x \in \pi A$, $\dim(A_x) = 0$ and $\cll(A_x) = \cll(A)_x$.
\end{itemize}
More generally, $A$ is $\mu$-good (where $\mu$ is a projection from
$\K^{n + m}$ to an $m$-dimensional coordinate space) if there is a permutation
of coordinates $\sigma$ such that $\mu = \pi \circ \sigma$, and $\sigma A$ is
$\pi$-good.
Finally, $A$ is $\Pi$-good, if it is $\mu$-good for some $\mu$ as above.
\end{definizione}

\begin{lemma}[Partition Lemma]
Let $\Afam$ be a finite collection of definable subsets of~$\K^n$.
Then, there exists a $\Pi$-good partition of $\K^n$ compatible with~$\Afam$.
\end{lemma}
\begin{proof}
The proof proceeds as in \cite[\S7, Partition Lemma]{miller05}, using  the
previous lemma.
\end{proof}

\begin{lemma}\label{lem:i-lc}
Let $A \subseteq \K^{n + m}$ be definable.
Then,
\begin{enumerate}
\item $\set{x \in \K^m: \lc(A_x) \neq (\lc A)_x}$ is nowhere dense;
\item for each $k \in \Nat$, $\set{x \in \K^m: (A^{(k)})_x \neq (A_x)^{(k)}}$
is nowhere dense;
\item if $\set{x \in \K^m: \lc(A_x) \neq \emptyset}$ is somewhere dense, then
$\lc(A) \neq \emptyset$.
\end{enumerate}
\end{lemma}
\begin{proof}
The same as~\cite[Lemma~8.1]{miller05}.
\end{proof}


\subsection{Constructible structures}\label{sec:ipmin}

\begin{definizione}
$\K$ is and \intro{\ipminimal{}} structure if, for every $\K' \elem \K$, 
every $0$-dimensional set definable in $\K'$ is constructible.
$T$ is an \ipminimal theory if every model of $T$ is \ipminimal.
\end{definizione}

\begin{thm}
The following are equivalent:
\begin{enumerate}
\item $\K$ is \ipminimal;
\item for every $\emptyset$-definable $A \subseteq \K^{n + m}$ there exits $N
\in \Nat$ such that for all $x \in \K^m$, if $\dim A_x = 0$, then $(A_x)^{(N)}
= \emptyset$;
\item 
every $\emptyset$-definable set is a finite union of
$\emptyset$-definable locally closed sets;
\item every definable set is constructible;
\item
every definable subset of $\K^n$ is a finite union of sets of the form
\[
\set{x \in \K^m: f(b,x) = 0 \et g(b,x) = 0},
\]
where $f$ and $g$ are $\emptyset$-definable and continuous, and $b \in \K^m$.
\end{enumerate}
Moreover, if $\K$ is \ipminimal, then it is \iminimal.
\end{thm}
\begin{proof}
The equivalence of the first 4 points is proved in the same way
as~\cite[Thm.~3.2]{miller05}, using Lemma~\ref{lem:i-lc}.
$(5 \Rightarrow 4)$ is obvious.
$(4 \Rightarrow 5)$ is proved in the same way as \cite[Lemma~2.10]{vdd-dense}.
The ``moreover'' clause follows from the fact that a constructible set 
with empty interior is nowhere dense.
\end{proof}

\begin{remark}
\Iminimality is not equivalent to \ipminimality.
In fact, it is not difficult to build an ultra-product of \ipminimal
structures which is not \ipminimal.
\end{remark}

\cite{pillay87}, extending the work in~\cite{robinson74}, studies topological
structures $M$ satisfying a weaker version of \ipminimality; that is, Pillay's
condition~(A) asks that every definable subset of $M$ is constructible.
In this context, he defines the \intro{dimension rank} of closed definable
subsets of $M$, which we will denote by $\rkP$,  in the following way:
\begin{enumerate}
\item If $X$ is non-empty, then $\rkP(X) \geq 0$.
\item $\rkP(X) \geq \lambda$ iff $\rkP(X) \geq \alpha$ for all 
$\alpha < \lambda$, where $\lambda$ is limit.
\item $\rkP(X) \geq \alpha + 1$ iff $X$ contains subset $Y$ which is closed,
nowhere dense (in $X$), and with $\rkP(Y) \geq \alpha$.
\end{enumerate}
Notice that $\rkP$ might depend on the ambient space $M$.

Let $M$ be a Hausdorff topological structure, such that every $M'$ elementarily
equivalent to $M$ satisfies condition (A) (\eg, $M$ is \ipminimal),
and $X \subseteq M$ be definable and closed.

\begin{lemma}[Pillay]
\begin{enumerate}
\item $\rkP(X) = 0$ iff $X$ is discrete and non-empty;
\item $Y \subseteq X \Rightarrow \rkP(Y) \leq \rkP(X)$;
\item If $X = X_1 \cup \dots X_n$, where the $X_i$ are closed and definable,
then $\rkP(X) = \max_{1 \leq i \leq n} \rkP(X_i)$.
\end{enumerate}
\end{lemma}

\begin{lemma}[Pillay]
\Tfae:
\begin{enumerate}
\item $\rkP(X) = \infty$;
\item 
there is a decreasing sequence $(X_i)_{i < \omega}$ of definable closed
subsets of $X$, such that $X_{i + 1}$ is definable, closed, and nowhere dense
in $X_i$, for all $i < \omega$;
\item
$X$ contains a definable definable closed nowhere dense subset~$Y$,
such that $\rkP(Y) = \infty$;
\item
$\rkP(X) > 2^{\card M}$.
\end{enumerate}
\end{lemma}

\begin{lemma}
Let $\K$ be \ipminimal, and $C \subseteq \K$ be non-empty, definable, closed,
and with empty interior.
Then, $C$~has at least one isolated point; in particular,
no non-empty closed perfect subset of $\K$ with empty interior is definable.
\end{lemma}
\begin{proof}
Assume, for contradiction, that $C$ is perfect.
Let $A := \K \setminus C$; $A$ is an open set;
let $C^L$ be the set of left end-points of the connected components of~$A$
and $C^R$ be the set of right end-points.
Notice that $C^L$ and $C^R$ are definable subsets of~$C$.
Since $\K$ is \ipminimal, there exists an open interval $I$ such that
$I \cap C^L$ and $I \cap C^R$ are closed in~$I$ 
and $I \cap C^L \neq \emptyset$.
Let $a \in C^L \cap I$.
Since $C$ is perfect and has empty interior, 
$a$~is an accumulation point for~$C^R$; hence, $a \in C^L \cap C^R$,
implying that $a$ is isolated in~$C$, absurd.
\end{proof}
In the above Lemma, the hypothesis ``$\K$ \ipminimal'' can be relaxed to
``for every non-empty definable subset $A$ of $\K$ there exists an open
interval~$I$, such that $I \cap A$ is non-empty and closed in~$I$''.

\begin{lemma}
Let $\K$ be \ipminimal, $C \subset \K$ be definable with empty interior, and
$D$ be the set of isolated points of~$C$.
Then, $D$~is discrete, definable, and dense in~$C$.
Moreover, $C' := C \setminus D$ is nowhere dense in~$C$.
\end{lemma}
\begin{proof}
That $D$ is discrete and definable is clear.
\begin{claim}
It suffices to prove the conclusion for $\cl C$.
\end{claim}
In fact, the isolated points of $\cl C$ are isolated points of~$C$.
Thus, \wloG $C$ is closed.
Let $A := \K \setminus C$, $C^L$ be the set of left end-points of the
connected components of~$A$ and $C^R$ be the set of right end-points.
$C^L$ is dense in~$C$.
If, for contradiction, $D$ is not dense in~$C$, let $I$ be a closed interval,
such that $C \cap D$ has no isolated points and is non empty: but this
contradicts the previous lemma.

The fact that $C'$ is nowhere dense in $C$ follows immediately from the first part.
\end{proof}

From the above lemmata, it is easy to deduce the following.
\begin{lemma}
Assume that $\K$ is \ipminimal, and $X \subseteq \K$ is definable and closed.
\begin{enumerate}
\item If $X$ is a finite union of discrete sets, then $\rkCB(X) = \rkP(X)$.
\item $\K$ is locally o-minimal iff $\rkP(\K) = 1$.  
\item If $\K$ is not locally o-minimal, then $\rkP(\K) \geq \omega$.
\item If $\K$ is \dminimal but not locally o-minimal, 
then $\rkP(\K) = \omega$.
\item If $\K$ is $\omega$-saturated,  \ipminimal, but not \dminimal, 
then $\rkP(\K) = \infty$.
\end{enumerate}
\end{lemma}
\begin{proof}
The last point follows from the fact that, since $\K$ is not \dminimal, then,
by saturation, we can find $X \subset \K$ definable, closed, with empty
interior, and such that $\rkCB(X) \geq \omega$.
Hence, by the previous lemma, $X \supset X^{(1)} \supset X^{(2)} \supset
\dots$ is an infinite descending chain of definable sets, such that 
$X^{(i + 1)}$ is closed and nowhere-dense in $X^{(i)}$.
\end{proof}


\subsection{Proof of Thm.~\ref{thm:imin}} \label{sec:imin-proof}
\mbox{}\indent
($\ref{en:i-imin-1} \Leftrightarrow \ref{en:i-dim-1}$) and ($\ref{en:i-dim-n} \Rightarrow \ref{en:i-dim-1}$) are clear.

For every $0 < n \in \Nat$, and $K = 2, $, let $(K_n)$ be the instantiation at
$n$ of the $K$th statement.
For instance, $(\ref{en:i-imin-n}_n)$ is equal to $(\ref{en:i-imin-1})$.
We will prove that $(\ref{en:i-imin-n}_n)  \Rightarrow (\ref{en:i-B}_n)
\Rightarrow (\ref{en:i-f}_n) \Rightarrow (\ref{en:i-imin-n}_n)$, 
that $(\ref{en:i-imin-n}_n) \Rightarrow (\ref{en:i-imin-n}_{n+1})$,
that $(\ref{en:i-imin-n})_n \Leftrightarrow (\ref{en:i-bd})_n$,
and that $(\ref{en:i-imin-n}_n)$ implies that every meager $X \subset \K^n$ is nowhere dense.

By induction on~$n$, the above would imply
($\ref{en:i-imin-1} \Leftrightarrow \ref{en:i-imin-n} 
\Leftrightarrow \ref{en:i-bd}
\Leftrightarrow \ref{en:i-dim-1} \Leftrightarrow \ref{en:i-f} \Leftrightarrow \ref{en:i-B}$).

($(\ref{en:i-imin-n})_n \Leftrightarrow (\ref{en:i-bd})_n$) is clear.

($(\ref{en:i-f}_n) \Rightarrow (\ref{en:i-bd}_n)$).
Let $X \subseteq \K^n$ be definable and with empty interior.
Let $f:= 1_X$ be the characteristic function of~$X$;
then, $\Dis(f) = \bd(X)$; thus, $\bd(X)$ is nowhere dense.

($(\ref{en:i-B}_n) \Rightarrow (\ref{en:i-f}_n)$).
Let $f: U \to \K$ be definable, with $U \subseteq \K^n$ open.
We want to prove that $\Dis(f)$ is nowhere dense;
\wloG, $f$ is bounded.
Let $A := \Gamma(f)$; then, $\Dis(f) = \B_n(A)$.
Thus, $\Dis(A)$ is nowhere dense.

Assume now that we have $(\ref{en:i-imin-n}_n)$.

Note that if $Y \subseteq \K^n$ is meager and definable, then, since $\K$ is
Baire, $Y$~has empty interior, and thus
$Y$ is nowhere dense; therefore, we have proved ($\ref{en:i-meager}_n$).

We prove now ($\ref{en:i-B}_n$).
Let $A \subseteq \K^{n + m}$ be definable, and $B := \B_n(A)$.

If we prove that $B$ is meager, then, since $\K$ is Baire, $B$~has empty
interior, and thus, by inductive hypothesis, $B$~is nowhere dense.
\Wlog, $A$~is bounded (because, after using a definable homeomorphism from
$\K$ to $(0,1)$, $B$~can only become larger).

Let $\pi := \Pi^{n + m}_n$, and $U := \pi(A)$.
If $U$ has empty interior, then, by ($\ref{en:i-imin-n}_n$),
$U$~is nowhere dense; thus, $B$ is also nowhere dense, because $B \subseteq \cl{\pi(A)}$.
Therefore, we can assume that $U$ has non-empty interior, and hence, \wloG,
that $U$ is open.

Thus, for every $r > 0$, let
\[
C(r) := \set{(x,y) \in \cl A: d(y, A_x) \geq r},
\]
and $B(r) := \pi\Pa{C(r)}$.
Since $B = \bigcup_r B(r)$, and by ($\ref{en:i-imin-n}_n$),
it suffices to prove that each $B(r)$ has empty interior.

Fix $r > 0$, and assume, for contradiction, that $U' \subseteq B(r)$ is a
non-empty open definable set.
\Wlog, $U' = U$.
Let $C' := \cl{C(r)}$: notice that $C'$ is \dcompact.
Define
$g: U \to \K^m$, $x \mapsto \lexmin (C'_x)$.
Since $g$ is lower semi-continuous, there exists $U' \subseteq U$ open and
non-empty, such that $g$ is continuous on~$U'$; \wloG, $U = U'$.
Define also $f : U \to \K^m$; $x \mapsto \lexinf\Pa{C(r)_x}$.
Note that $f(x) \geq g(x)$ for every $x \in U$.

\begin{claim}\label{cl:disc-nd}
The set $D := \set{x \in U: f(x) > g(x)}$ is nowhere dense.
\end{claim}
We will do only the case $m = 1$.
It suffices to prove that $D$ is meager.
For every $s > 0$, let $D(s) := \set{x \in U: f(x) \geq g(x) + s}$.
If we prove that each $D(s)$ is nowhere dense, we have the claim.
By ($\ref{en:i-imin-n}_n$), it suffices to prove that $D(s)$ has empty interior.
Assume, for contradiction, that $V$ is a non-empty subset of~$U$, and let $x
\in V$.
Since $g$ is continuous, we can assume that $d\Pa{g(x'), g(x)} < s/2$ for
every $x' \in V$.
By definition of $f$ and~$g$,
there exists $x' \in V$ such that $d\Pa{f(x'), g(x)} < s/2$. 
Hence, $d\Pa{f(x'), g(x')} < s$, absurd.

Thus, after shrinking~$U$, we can assume that $f = g$.
Fix $x \in U$, and let $y := f(x)$.
Since $g$ is continuous on $U$, after shrinking $U$ we can assume that
$d\Pa{f(x'), y} < r/3$ for every $x' \in U$. 
Since $\Gamma(f) \subseteq \cl A$, there exists $(x', y') \in A$, such that
$(x', y')$ is near $(x, y)$; that is, $x' \in U$ and $d(y', y) < r/3$.
Moreover, by definition of~$f$, there exists $y'' \in C_{x'}$, such that
$d \Pa{f(x'), y''} < r/3$.
However, this imply that $d\Pa{y', f(x')} < r$, contradicting the definition
of $C(r)$.


We prove now ($\ref{en:i-imin-n}_{n+1}$):
let $A \subseteq \K^{n+1}$ is definable and has
empty interior; we want to show that $A$ is nowhere dense.
If not, let
\[
E' := \set{x \in \K^n: \cll(A_x) \text{ has non-empty interior}}.
\]
By assumption, $E'$ has non-empty interior.
\Wlog, $A$ is bounded.
Let $E := \set{x \in \K^n: \cll(A)_x \text{ has non-empty interior}}$.
Since, by ($\ref{en:i-f}_n$), $E \sdiff E'$ is nowhere dense,
$E$~has non-empty interior.
Since $\K$ is Baire, there exists $0 < r \in \K$ such that
\[
E(r) := \set{ x \in \K^n: A_x \text{ contains an interval of length } r}
\]
has non-empty interior; let $U \subseteq E(r)$ be a non-empty open set.

After shrinking $A$ if necessary, we can assume that, for every $x \in U$,
$A$ is bounded and $A_x$ is an open interval of length~$r$.
For every $x \in U$, let $h(x)$ be the centre of~$A_x$.

By ($\ref{en:i-f}_n$),
the set of points where $h$ is continuous has non-empty interior.
Thus, after shrinking~$U$, we can assume that $h$ is
continuous.
But then the set $\set{(x,y) \in U \times \K: h(x) - r < y < h(x) + r}$ is
open and contained in~$A$, absurd.

($\ref{en:i-B} \Leftrightarrow \ref{en:i-fr}$) is clear.

($\ref{en:i-imin-n} \Rightarrow \ref{en:i-meager-n} \Rightarrow
\ref{en:i-meager-1})$ are also clear.

($\ref{en:i-meager-1} \Rightarrow \ref{en:i-imin-1}$).
Let $X \subseteq \K$ be definable and have empty interior.
By hypothesis, $X$~is meager.
Moreover, $\fr X = \cl X \setminus X$ has also empty interior, and thus it is
meager.
Therefore, $\cl X$ is meager.
Since $\K$ is Baire, $\cl X$ has empty interior.

$(\ref{en:i-B} \Rightarrow \ref{en:i-dim-n}$).
Let $X \subseteq \K^n$ be definable, and let $d := \dim \cl X$.
We want to prove that $\dim X = d$.
\Wlog, $\pi(\cl X)$ contains an open subset of~$\K^d$, where $\pi := \Pi^n_d$.
If, for contradiction, $\dim X < d$, then, by ($\ref{en:i-imin-n}$),
$\pi(X)$~is nowhere dense.
Notice that $\pi(\cl X) \setminus \pi(X) \subseteq \B_d(X)$.
Since $\B_d(X)$ is nowhere dense, we get a contradiction.

($\ref{en:i-imin-n} \Rightarrow \ref{en:i-dim-fiber}$).
Let $X := \set{x \in \K^d: \dim A_x > 0}$, where $d := \dim A$, and
assume, for contradiction, that $\dim X > 0$.
If $d = 0$, then $X = \emptyset \subset \K^0 = \set{0}$,
and we have a contradiction.
Thus, \wloG, $A$~is closed (because $\dim \cl A = \dim A)$, 
and $Y := \Pi^n_{d + 1}(A)$ satisfies
\[
\forall x \in X\ \dim(Y_x) > 0.
\]
By Kuratowski-Ulam's theorem, this implies that $Y$ is not meager,
and thus has non-empty interior, contradicting $\dim A = d$.

($\ref{en:i-imin-n} \Rightarrow \ref{en:i-dim-union-n}$).
Let $A_1, A_2 \subseteq \K$ be definable, such that $\dim A_i < d$, $i = 1, 2$.
We have to prove that $\dim(A_1 \cup A_2) < d$.
Assume, for contradiction, that $B := \Pi^n_d(A_1 \cup A_2)$ has non-empty
interior. 
Let $B_i := \Pi^n_d(A_i)$;
notice that $\dim B_i < d$, and $B_1 \cup B_2 = B$.
By (\ref{en:i-imin-n}), the $B_i$ are nowhere dense in~$\K^d$;
thus, $B$~is nowhere dense, absurd.

($\ref{en:i-dim-union-n} \Rightarrow \ref{en:i-dim-union-1}$) is obvious.

($\ref{en:i-dim-union-1} \Rightarrow \ref{en:i-dim-1}$).
Let $A \subseteq \K$ be definable.
We have to prove that $\dim (\cl A) = \dim(A)$.
However, $\cl A = A \cup \fr A$.
Since $\fr A$ has empty interior, $\dim \fr A = 0$, and we are done.

($\ref{en:i-dim-fiber} \Rightarrow \ref{en:i-dim-union-1}$).
Let $A, B$ be definable subsets of $\K$ with empty interior.
We have to prove that $A \cup B$ has also empty interior.
Define $X := \Pa{(0,1) \times A} \cup \Pa{\fr B \times (0,1)} \subset \K^2$.
Notice that $\dim X = 1$.
By~$\ref{en:i-dim-fiber}$, the set
$Y := \set{y \in \K: \dim(X_y) > 0}$ is nowhere dense.
However, $Y = A \cup B$; thus, $\dim(A \cup B) = 0$.

($\ref{en:i-imin-n} \Rightarrow \ref{en:i-increasing}$).
Let $\Pa{A_x}_{x \in \K}$ be an increasing definable family of subsets of~$\K^n$, each of them of dimension less or equal to~$d$.
Let $A := \bigcup_x A_x$.
Assume, for contradiction, that $\dim A > d$; \wloG, $U := \Pi^n_{d+1}(A)$ has
non-empty interior.
However, $U = \bigcup_x \Pi^n_{d + 1}(A_x)$.
Since $\dim A_x \leq d$, each $\Pi^n_{d + 1}(A_x)$ is nowhere dense, and thus
$U$ is meager, contradicting the fact that $\K$ is Baire.

($\ref{en:i-dim-fiber-2} \Rightarrow \ref{en:i-dim-fiber}$) is obvious.

($\ref{en:i-dim-fiber} \Rightarrow \ref{en:i-dim-fiber-2}$).
Assume, for contradiction, than $\dim C > k$; \wloG, $U := \Pi^n_{k + 1}(C)$ has
non-empty interior.
Moreover, since $\dim(A) = \dim(\cl A)$, by property (\ref{en:i-increasing}),
\wloG $A$ is \dcompact.
By~\ref{en:i-dim-union-n}, \wloG the set
\[
C' := \set{x \in \K^n: \dim \Pa{\Pi^m_d(A_x)} \geq d }
\]
has dimension greater than~$k$, 
and $D' := \Pi^n_{k + 1}(C')$ has non-empty interior.
Let $B := \Pi^{n + m}_{d + k + 1}(A)$; by assumption, $B$~is nowhere dense.
Hence, by Kuratowski-Ulam's theorem, the set
\[
D := \set{u \in \K^{k + 1}: \dim(B_u) \geq d}
\]
has empty interior.
However, for every $u \in \K^{k + 1}$, $B_u = \Pi_d(A_u)$, and thus $D'
\subseteq D$, absurd.

($\ref{en:i-imin-1} \Rightarrow \ref{en:i-monotonicity}$)
and ($\ref{en:i-imin-1} \Rightarrow \ref{en:i-continuous}$)
have  the same proof as~\cite[Thm.~3.3]{miller05}.
%
\mbox{}\hfill \qedsymbol
\bigskip


\section{Definable choice}\label{sec:DSF}
As usual, $\K$~is some definably complete expansion of an ordered field.

Some version of the following lemma has been proved by C.~Miller.
\begin{lemma}[Definable Choice]\label{lem:skolem}
\begin{enumerate}
\item Let $P \subseteq \K$ be a set of parameters, and
$A \subset \K^n$ be $P$-definable, non-empty and constructible.
Then, there exists a $P$-definable point $a \in A$ .
\item Let $X \subseteq \K^{n + m}$ be definable and such that $X_b$ is
constructible for every $b \in \K^m$.
Then, $X$ has a definable $n$-choice function, that is a definable function 
$f : \Pi^{n + m}_m(X) \to X$, such that $f(a) \in X_a$ for every $a \in \Pi^{n + m}_m(X)$.
\item Let $X \subseteq \K^{n + m}$ be definable and constructible.
Then, $X$ has a definable $n$-choice function.
\item Suppose that every unary definable set contains a locally closed point.
Then, $\K$~has definable Skolem functions (\intro{\DSF}).
\end{enumerate}
\end{lemma}
\begin{proof}
($1$).
Since $A$ is constructible, $A' := \lc(A)$ is non-empty; thus, since $A'$ is
also $P$-definable, it suffices to prove the conclusion for~$A'$;
therefore, \wloG $A$ is locally closed.

First, we will prove the case when $A$ is closed in~$\K^n$.
For every $r > 0$, let $A(r) := \set{a \in A: \abs a \leq }$, let $r_0 :=
\inf\set{r \in \K: A(r) \neq \emptyset}$, and let $A' := A(2r_0)$.
Notice that $A'$ is \dcompact, non-empty and $P$-definable.
Let $a := \lexmin(A')$: notice that $a$ is also $P$-definable, and in~$A$.

Next, we will treat the case when $A$ is discrete.
For every $r > 0$, let $A(r) := \set{a \in A: B(a,r) \cap A = \set a}$, 
let $r_0 := \inf\set{r \in \K: A(r) \neq \emptyset}$, and let $A' := A(2r_0)$.
Notice that $A'$ is closed, non-empty, and $P$-definable, and apply the
previous case.

We will prove the general case by induction on~$n$.
If $n = 1$, for every $a \in A$ let 
$r(a) := \sup \set{r > 0: A \cap B(a,r) = \cl A \cap B(a,r)}$.
Since $A$ is locally closed, $r(a) > 0$ for every $a \in A$.
Let $U := \bigcup_{a \in A} B\Pa{a, r(a)/2}$.
Notice that $U$ is open, $P$-definable, and $A = \cl A \cap U$.
By the ``open set'' case, there exists $u \in U$ $P$-definable.
Let $I$ be the definably connected component of $U$ containing $u$:
$I$ exists because $n = 1$, and it is a $P$-definable open interval;
let $A' := A \cap I$.
Therefore, $A'$ is closed in $I$, $P$-definable and non-empty;
since $I$ is definably homeomorphic to $\K$, via a $P$-definable map, we can
apply the ``closed set'' case to find a $P$-definable point in~$A'$.

If $n > 1$, we proceed by induction, and we assume we have already proved the
conclusion for $n - 1$.
let $A' := \Pi^{n}_{n - 1}(A)$.
By the inductive hypothesis, there exists $a' \in A'$ $P$-definable.
By the case $n = 1$, there exists $a'' \in A_{a'}$ also $P$-definable.
Let $A := (a', a'')$.

$(2)$.
The construction in ($1$) gives a definable way to choose $x_b \in X_b$ for
every $b \in \Pi^{n + m}_m(X)$.

$(3)$ is immediate from ($2$).

$(4)$.
Let $A \subseteq \K^{n + m}$ be definable.
We have to prove that there exists $f: B \to A$ definable, such that $f(b) \in
A_b$ for every $b \in B$, where $B := \Pi^{n + m}_m(A)$.
We proceed by induction on~$n$.

If $n = 0$, $f$ is the identity.
If $n = 1$, for every $b \in B$ let $C_b := \lc(A_b)$.
By hypothesis, $C_b \neq \emptyset$ for every $b \in B$, and it is
constructible. Thus, by ($2$), $C$ has a definable 1-choice function, and the
same function will work for~$A$.

Assume that $n > 1$ and we have already proved the conclusion for every
$n' < n$.
Let $C := \Pi^{n + m}_{m + 1}(A)$.
By inductive hypothesis, there exists a definable $(n-1)$-choice function
$g: C \to A$.
By the case $n = 1$, there exists a definable 1-choice function
$h: \Pi^{m + 1}_m(C) \to C$.
Let $f := g \circ h : \Pi^{n + m}_m(A) \to A$: $f$ is an n-choice function
for~$A$.
\end{proof}

Therefore, locally o-minimal, d-minimal and \ipminimal structures
have \DSF.
On the other hand, \aminimal structures might not have \DSF:
for instance, $(\Real, \Ralg)$ does not have \DSF~\cite[5.4]{DMS}.

For the same reason, \ipminimal structure have elimination of imaginaries.


\begin{lemma}
Let $X \subseteq \K^n$ be definable.
Assume that $X$ is both and $\Fs$ and a $\Gd$, and that $\cl X$ is Baire.
Then, $\lc(X) \neq \emptyset$.
\end{lemma}
Cf.~\cite[\S34.VI]{kuratowski}.
\begin{proof}
Let $Y := \cl X$.
We have to prove that the interior of $X$ inside $Y$ is non-empty.
Otherwise, $X$ is both dense and co-dense in~$Y$.
However, since $X$ is an~$\Fs$, this implies that $X$ is meager in~$Y$.
For the same reason, $Y \setminus X$ is meager in~$Y$,
contradicting the fact that $Y$ is Baire.
\end{proof}


\subsection{Sard's Lemma and dimension inequalities}

For every $\Cone$ function $f : \K^m \to \K^n$, define
$\Lambda_f(k) := \set{x \in \K^n: \rkM \Pa{Df(x))} \leq k}$, and
$\Sigma_f(k) := f\Pa{\Lambda_f(k)}$.
The set of singular values of $f$ is 
$\Sigma_f := \bigcup_{k = 0}^{n - 1} \Sigma_f(k)$.

\begin{lemma}[Sard's Lemma]
Let $f : \K^m \to \K^n$ be definable and~$\Cone$.
If $\K$ is \iminimal and has \DSF (and, in particular, if $\K$ is \ipminimal),
then $\dim\Pa{\Sigma_f(d)} \leq d$.
\end{lemma}
\begin{proof}
If $d \geq n$, the conclusion is trivial.
Let $d < n$, and assume, for contradiction, that
$\Pi^n_d\Pa{\Sigma_f(d)}$ contains a non-empty open box~$B$.
By \DSF, there exists $g: B \to \Lambda_f(d)$, such that
$\Pi^n_d \circ f \circ g = \identity_B$.
Since $\K$ is \iminimal, we can apply Lemma~\ref{lem:imin-C1}, and,
by shrinking $B$ if necessary, we can assume that $g$ is~$\Cone$.
Hence, by differentiation, we have that, for every $x \in B$,
$\Pi^n_d(f(g(x))) \cdot (Df)(g(x)) \cdot Dg(x) = \idmatrix_d$.
Therefore, $\rkM\Pa{(Df)(g(x))} \geq d$, a contradiction.
\end{proof}

\begin{lemma}\label{lem:imin-dimension-fiber}
Assume that $\K$ is \iminimal with \DSF.
Let $l, m, n, d$ be natural numbers.
Let $X \subseteq \K^n$ be definable, such that $\dim X \geq m + d$, 
$Y \subseteq \K^l$ be definable of dimension~$m$,
and $f: X \to Y$ definable, 
Then there exists $a \in Y$ such that  $\dim(f^{-1}(a)) \geq d$.
\end{lemma}
\begin{proof}
By induction on $(l,m,n,d)$.
\begin{itemize}
\item 
If $m + d < n$, then \wloG $\pi(X)$ contains a non-empty open box~$B$,
where $\pi := \Pi^n_{m + d}$.  By \DSF, there exists a definable function 
$g: B \to X$ such that $g \circ \pi = \id_B$.
Let $\tilde f := f \circ g: B \to Y$.
By induction on~$n$, there exists $a \in Y$ such that
$\dim(\tilde f^{-1}(a)) \geq d$.
Therefore, \wloG $m + d = n$ and $X = \K^n$.
Moreover, by Theorem~\ref{thm:imin}(\ref{en:i-continuous}), \wloG
$f$ is continuous.
\item
If $m = 0$, since $f$ is continuous, $f(\K^n)$ is a connected subset of the
$0$-dimensional set~$Y$, and therefore $f(\K^n) = \set{a}$ for some
$a \in Y$.
\item
If $d = 0$, there si nothing to prove.
Thus, \wloG $m > 0$ and $d > 0$.
\item
If $l > m$, then, by the Partition Lemma, \wloG $Y$ is $\mu$-good,
where $\mu := \Pi^{l}_m$.
Let $g := \mu \circ f: \K^n \to \K^m$.
By induction on~$l$, there exists $b \in \K^m$ such that
$\dim(\tilde g^{-1}(b)) \geq d$.
Moreover, since $Y$ is $\mu$-good, $\dim(Y_b) = 0$.
Let $\tilde X := \rest {g^{-1}(b)}$ and
$\tilde := f \rest {g^{-1}(b)} : \tilde X \to Y_{[b]}$.
By the case $m = 0$, there exists $a \in Y_{[b]}$, such that
$\dim(\tilde f^{-1}(a)) \geq d$.
Hence, \wloG $l = m$ and $Y = \K^m$.
\item
Hence, we are reduced to the situation $f: \K^n \to \K^m$, with $n = m + d$.
By Lemma~\ref{lem:imin-C1}, \wloG $f$ is $\Cone$.
Let $Y' := f(\K^n)$ and $\Sigma_f \subseteq Y$ 
be the set of singular values for~$f$.
If, for contradiction, $\dim(f^{-1}(a)) < d$ for every $a \in \K^m$,
then $Y' = \Sigma_f$.
Thus by Sard's Lemma, $\dim Y' < m$.
Therefore, by induction on~$m$, there exists $a \in Y'$ such that
$\dim(f^{-1})(a) \geq d$, a contradiction.
\qedhere
\end{itemize}
\end{proof}

\begin{lemma}\label{lem:dim-function}
Let $\K$ be \iminimal with \DSF, $X \subseteq \K^n$, $Y \subseteq \K^m$, and
$f: X \to Y$ be definable.
Then:
\begin{enumerate}
\item if $f$ is surjective, then $\dim Y \leq \dim X$;
\item if $f$ is injective, then $\dim Y \geq \dim X$;
\item if $f$ is bijective, then $\dim Y = \dim X$.
\end{enumerate}
\end{lemma}
\begin{proof}
1)
Assume, for contradiction, that $k := \dim X < d := \dim Y$.
By \DSF, there exists a definable function $g : Y \to X$,
such that $f \circ g = \id_Y$.
By Lemma~\ref{lem:imin-dimension-fiber}, there exists
$a \in X$ such that $\dim(g^{-1}(a)) \geq d - k \geq 1$, contradicting the
fact that $g$ is injective.

2) follows from 1) and definable choice, and 3) follows from 1) and 2).
\end{proof}


\section{D-minimal structures}\label{sec:dmin}
\begin{definizione}
$\K$ is d-minimal if for every $\K' \elem \K$, every definable subset of~$\K$
is the union of an open set and finitely many discrete sets.
\end{definizione}

\begin{remark}\label{rem:discrete}
Let $A \subseteq \K^n$ be definable.
If $A$ is a union of $N$ discrete sets, then $A$ is a union of $N$ definable
and discrete sets.%
\end{remark}

\begin{remark}
Every d-minimal structure is \ipminimal.
\end{remark}

\begin{definizione}
Let $d \leq n \in \Nat$, $\Pi(n,d)$ be the set of projections form $\K^n$ onto
$d$-dimensional coordinate spaces, and $\mu \in \Pi(n,d)$.  For every $p \in
\Nat$, let $\reg^p_\mu(A)$ and $\reg^p(A)$ be defined as in~\cite[Def.~8.4 and
\S3.4]{miller05}.  As in the case when $\K$ is an expansion of~$\Rbar$,
$\reg^p_\mu(A)$ is definable, open in $A$, and a $\Cp$-submanifold of $\K^n$
of dimension~$d$.

For every $A \subseteq \K^n$, let $\isol(A)$ be the set of isolated points
of~$A$. Notice that $\isol(A)$ is discrete.
\end{definizione}

\begin{lemma}\label{lem:regular}
Suppose that $\K$ is \iminimal, and let $A \subseteq \K^{m + n}$ be definable,
such that $B := \set{x \in \K^m: \isol(A_x) \neq \emptyset}$ has interior.
Then, for every $p \in \Nat$, $\reg^p_\pi(A) \neq \emptyset$, 
where $\pi := \Pi^{n + m}_m$.
\end{lemma}
\begin{proof}
Fix $p \in \Nat$; let $V \subseteq B$ be a non-empty open box, and
$C := \bigsqcup_{v \in V} \Pa{\set v \times \isol(A_v)}$.
Notice that $V \subseteq \pi(C)$.
By Definable Choice, there exists a definable function $f : V \to \K^n$ such
that $\pair{x, f(x)} \in C$ for every $x \in V$.
For every $x \in V$, define
\[\begin{aligned}
f^+(x) &:= \min \Pa{f(x) + 1, \inf\set{y \in A_x: y > f(x)}},\\
f^-(x) &:= \max \Pa{f(x) - 1, \sup\set{y \in A_x: y < f(x)}}.
\end{aligned}\]
Notice that $f^- < f < f^+$ on all~$V$.
By \iminimality, after shrinking~$V$, we can assume that $f$, $f^+$ and $f^-$
are $\Cp$ on~$V$.
It is easy to see that $\Gamma(f) \subseteq \reg^p_\pi(A)$.
\end{proof}

\begin{lemma}\label{lem:i-regular-2}
Suppose that every 0-dimensional definable subset of $\K$ has an isolated
point, and let $p \in \Nat$. Then:
\begin{enumerate}
\item $\K$ is \iminimal.
\item Let $\Afam$ be a finite collection of definable subsets of~$\K^n$.
Then, there is a $\Pi$-good partition  $\Part$ of~$\Afam$, compatible
with~$\Afam$, such that $P \setminus \reg^0_\mu(P)$ is
nowhere dense in $P$ for every projection $\mu$ and every $P \in \Part$ such
that $P$ is $\mu$-good.
\item $A \setminus \reg^p(A)$ is nowhere dense in~$A$, for every definable
set~$A$.
\end{enumerate}
\end{lemma}
\begin{proof}
(1). Let $X \subseteq \K$ be definable and with empty interior.
Suppose, for contradiction, that $\cl X$ contains a non-empty open
interval~$I$, and let $Y := X \cap I$.
Notice that $\dim Y = 0$ and $Y$ is dense in $I$, and therefore it has no
isolated points.

(2) and (3) have the same proof as ~\cite[Prop.~8.4]{miller05}
(using Thm.~\ref{thm:imin} and Lemma~\ref{lem:regular}).
\end{proof}

\begin{lemma}
The following are equivalent:
\begin{enumerate}
\item $\K$ is d-minimal;
\item for every $\K' \elem \K$, every subset of~$\K$
is the union of a definable open set and finitely many definable discrete sets;
\item for every $m \in \Nat$ and every definable $A \subseteq \K^{m+1}$ there
exists $N \in \Nat$ such that, for all $x \in \K^m$, either $A_x$ has
interior or is a union of $N$ definable discrete sets;
\item for every $m, n \in \Nat$ and definable $A \subseteq \K^{n + m}$ there
exists $N \in \Nat$ such that for every $x \in \K^m$, either $\dim A_x > 0$,
or $A_x$ is a union of $N$ definable discrete sets.
\end{enumerate}
\end{lemma}
\begin{proof}
($1 \Leftrightarrow 2$) follows from Remark~\ref{rem:discrete}.


$(2 \Leftrightarrow 3)$ is a routine compactness argument.

(3) is the case $n = 1$ of (4).

$(3 \Rightarrow 4)$.
Induction on~$n$.
The case $n = 1$ is the hypothesis.
Let $n > 1$, and assume we have already proved $(3)$ for each $n'< n$.
Let $C := \set{x \in \K^m: \dim(A_x) = 0}$.
Let $x \in C$.
For each $y \in K$, we have $\dim(A_{(x,y)}) = 0$, and therefore $A_{(x,y)}$
is a union of $N$ discrete definable sets (for some $N$ independent from $x$
and~$y$). 
Moreover, for each $x \in C$, the set
$D(x) := \set{y \in \K: A_{(x,y)} \neq  \emptyset}$ has empty interior,
because $\K$ is \iminimal. 
Thus, for every $x \in C$, $D(x)$~is a union of $M$ definable discrete sets
(for some $M$ independent from~$x$).
Hence, for every $x \in C$, $A_x$~is a union of $NM$ discrete sets, which can
be taken definable.
\end{proof}

\begin{definizione}[Cantor-Bendixson Rank]
Let $T$ be a Hausdorff topological space.
For every $X \subseteq T$, and every ordinal $\alpha$,
let $\CBd X 0 := X$,
$\CBd X {\alpha + 1}$ be the set of non-isolated points of~$\CBd X {\alpha}$,
and $\CBd X {\alpha} := \bigcap_{\beta < \alpha} \CBd X {\alpha}$ if $\alpha$
is a limit ordinal.
Let $\rkCB(X)$, the Cantor-Bendixson rank of~$X$,
be the smallest ordinal $\alpha$ such that
$\CBd X {\alpha} = \emptyset$ (or $\rkCB(X) = + \infty$ if such $\alpha$ does
not exist).
For every $a \in X$, let $\rkCB_X(a)$ be the supremum of ordinals $\alpha$ such
that $a \in \CBd X {\alpha}$.
\end{definizione}
Notice that each $\CBd X {\alpha} \setminus \CBd X {\alpha + 1}$ is discrete,
and that $X$ is a finite union of discrete sets iff $\rkCB(X) < \omega$.
Moreover,  $\CBd X {\rkCB(X)} = \emptyset$ (if $\rkCB (X) < + \infty)$.
Besides, $\rkCB(X) = 0$ iff $X$ is empty, $\rkCB(X) = 1$ iff $X$ is discrete
(and non-empty).

$\K$~is \aminimal iff every definable subset of $\K$ of dimension $0$
has rank at most 1 (to be proved).

\begin{remark}
If $X$ and $X'$ are subsets of a Hausdorff topological space~$T$, 
then $\rkCB(X \cup X') \leq \rkCB(X) \oplus \rkCB(X')$, where $\oplus$ is the
Cantor sum of ordinals.
Hence, a set $X$ is a union of $n$ discrete sets iff $\rkCB(X) \leq n$.
\end{remark}
\begin{proof}
By induction on $\alpha := \rkCB(X)$ and $\beta :=\rkCB(Y)$.
The basic case when $X$ is a singleton is obvious.
For contradiction, let $a \in \CBd{(X \cup Y)} {\alpha + \beta}$.
\Wlog, $a \in X$; let $\gamma := \rkCB_X(a) < \alpha$.

Let $V$ be an open neighbourhood of $a$ such that
$\CBd X{\gamma} \cap V = \set a$, and $X' := X \cap V \setminus \set a$.
Notice that $\alpha' := \rkCB(X') \leq \gamma < \alpha$.
Hence, by inductive hypothesis, $\rkCB(X' \cup Y) \leq \alpha' \oplus \beta <
\alpha \oplus \beta$.
Thus, $\rkCB(X' \cup Y \cup \set a) \leq \alpha' \oplus \beta \oplus 1 \leq
\alpha \oplus \beta$ .
Thus, $\rkCB_{X \cup Y}(a) < \alpha \oplus \beta$ for every $a \in X \cup Y$,
and we are done.
\end{proof}

\begin{corollary}
If $\K$ d-minimal iff for every $n \in \Nat$ and every definable
$X \subseteq \K^{n+1}$, there exists $d \in \Nat$ such that,
for all $a \in \K^n$, $\rkCB \Pa{X_a \setminus \interior(X_a)} \leq d$.
\end{corollary}

\begin{definizione}
A definable $d$-dimensional $\Cp$-submanifold $M$ is weakly $\mu$-special if
for each $y \in \mu M$ and $x \in M_y$ there exist $U \subseteq \K^d$ open box
around $y$ and $W \subseteq \K^{n-d}$ open box around~$x$, such that
$A \cap(U \times W) = \Gamma(f)$ for some (definable) $\Cp$-map $f: U \to W$.
$M$ is $\mu$-special if the box $U$ in the above definition does \emph{not}
depend on~$x$ (but only on~$y$).
$M$~is special if it is $\mu$-special for some $\mu \in \Pi(n,d)$.
\end{definizione}

\begin{lemma}\label{lem:d-special}
Assume $\K$ is d-minimal, let $p \in\Nat$, and $\Afam$ be a finite collection
of definable subsets of $\K^n$.
Then, there exists a finite partition of $\K^n$ into \emph{weakly} special
definable $\Cp$-submanifolds compatible with~$\Afam$.
\end{lemma}

\begin{proof}
As in the proof of~\cite[Thm.~3.4.1]{miller05}, we are reduced to show that if
$A \in \Afam$, with $0 < d := \dim A < n$; then $A$ can be partitioned into
special $\Cp$-manifolds.
Moreover, we can further assume that $A$ is a $\pi$-good, where $\pi :=
\Pi^n_d$, that $A \setminus \reg^p_\pi(A)$ is nowhere dense in~$A$,
and each $A_x$ is discrete, for every $x \in \K^d$.

Let $M := \reg^p_\pi(A)$.
Notice that $M$ is a weakly $\pi$-special $\Cp$-manifold.
It suffices to prove that $\pi(A \setminus M)$ is nowhere dense (in~$\K^d$)
to conclude the proof (since then $\fdim(A \setminus M) < \fdim(A)$).
Assume, for contradiction, that $B \subset \pi(A \setminus M)$ is a non-empty
open box, and let $N := \reg^p_\pi(A \setminus M)$.
By shrinking~$B$, we might assume that $\cll(M_x) = (\cll M)_x$ for every
$x \in B$. 
By Lemma~\ref{lem:i-regular-2}, $M$~is dense in~$A$, and therefore, for every
$x \in B$,
$\cll(A_x) = (\cll A)_x = (\cll M)_x = \cll(M_x)$, that is $M_x$ is dense
in~$A_x$.
However, $A_x$~is discrete, and thus $A_x = M_x$.
%
%
%
\end{proof}

\begin{conjecture}
Assume $\K$ is d-minimal, let $p \in\Nat$, and $\Afam$ be a finite collection of definable
subsets of $\K^n$.
Then, there exists a finite partition of $\K^n$ into special definable
$\Cp$-submanifolds compatible with~$\Afam$.
\end{conjecture}


\section{Dense pairs}\label{sec:dense}

\subsection{Cauchy completion}
For particular cases of some of the results in this section, see
also~\cite{LS}. 

\begin{definizione}
If $\Kb$ is an ordered field, we denote by $\KbC$ the \intro{Cauchy completion}
of~$\Kb$, that is, the maximal linearly ordered group such that $\Kb$ is dense
in~$\KbC$; notice that $\KbC$ is, in a canonical way, a real closed
field~\cite{scott}.

A \intro{cut} $\Lambda := (\Lambda^L, \Lambda^R)$ of $\K$ is a partition of
$\K$ into  two disjoint subsets $\Lambda^L, \Lambda^R$, such that $y < x$ for
every $y \in \Lambda^L$ and $x \in \Lambda^R$.
A cut $(\Lambda^L, \Lambda^R)$ is a \intro{gap} if $\Lambda^L$ 
is non-empty and has no maximum,
and $\Lambda^R$ is non-empty and has no minimum.
A cut $\Lambda$ is \intro{regular} if $\hat \Lambda = 0$, where
$\hat\Lambda := \inf \set {x - y: x, y \in \K \ \&\ y < \Lambda < x }$.  

\end{definizione}
The Cauchy completion of $\K$ is the disjoint union of $\K$ and the set of
regular gaps of~$\K$, with suitably defined order $<$ and operations $+$ and $\cdot$.

\begin{lemma}\label{lem:regular-type}
Let $\K$ be definably complete, $\K^* \succ \K$, $b \in \K^* \setminus \K$.
For every $X \subseteq \K^m$ definable, let $X^* \subseteq {\K^*}^m$ be the
interpretation of $X$ in $\K^*$.

Let $E \subset \K$ be closed, $\K$-definable, and with 
$n := \rkCB(E)  < \omega$, let $\Lambda$ be the cut of $\K$ determined by~$b$.
If $\Lambda$ is regular, then $b \notin E^*$.

If moreover $\K$ is d-minimal, and $D \subset \K$ is $\K$-definable and of
dimension~$0$, then $b \notin D$.
\end{lemma}
\begin{proof}
We will prove the lemma by induction on~$n := \rkCB(E)$.
\Wlog, $E$~is bounded (because $b$ is $\K$-bounded).
If $n = 0$, then $E = \emptyset$.
If $n = 1$, then $E$ is discrete; thus, $E$ is pseudo-finite; 
let $\delta := \delta(E)$; notice that $0 < \delta \in \K$. 
Since $\Lambda$ is regular, there exist $a', a'' \in \K$
such that $a' < b < a''$ and $a'' - a' < \delta$.
Thus, $(a', a'')^* \cap E^* = \set b$, and therefore $b$ is $\K$-definable, absurd.

If $n > 1$, let $G := \CBd E {n - 1}$; notice that $G$ is closed, discrete and
non-empty; thus, $\rkCB(G) = 1$, and $G$ is pseudo-finite.
By the case $n = 1$, $b \notin G^*$; let $a', a'' \in \K$ such that $a' < b <
a''$ and $G \cap [a', a''] = \emptyset$ ($a'$ and $a''$ exist by the proof of
the case $n = 1$).
Let $F := E \cap [a', a'']$.
Notice that $F$ is \dcompact.
Moreover, $\rkCB(F) < n$.
Therefore, by inductive hypothesis, $b \notin F^*$.
Hence, $b \notin E^*$.

If $\K$ is d-minimal, let $E$ be the closure of $D$ in~$\K$.
By d-minimality, $\rkCB(D) < \omega$; thus, $b \notin D^*$.
\end{proof}

From the proof of the above lemma, we can deduce the following.
\begin{lemma}\label{lem:regular-cut-dimension}
Let $\K$ be d-minimal, and $\K \prec \K^*$.
Let $c \in \K^* \setminus \K$, 
$\Lambda$~be the cut determined by $c$ over~$\K$,
and $X^*\subseteq \K^*$ be $\K$-definable.
If $\Lambda$ is a regular gap and $c \in X^*$, then there exists an open
interval $I$ with end-points in $\K$ such that $c \in I \subseteq X^*$.
\end{lemma}

\begin{lemma}\label{lem:regular-type-image}
Let $\K$ be definably complete, $\K^* \succ \K$, $b \in \K^* \setminus \K$,
$\Lambda$ be the cut of $\K$ determined by~$b$, 
$f: \K \to \K$ be definable,
$c := f^*(b) \in \K^*$, and $\Gamma$ be the cut determined by $c$ over~$\K$.
Assume that $f$ is continuous and  strictly monotone.
Then, $\Gamma$~is a regular gap, and $\K \gen c = \K \gen b$,
where $\K \gen b$ is the definable closure of $\K \cup \set b$.
\end{lemma}
\begin{proof}
Let $0 < \varepsilon \in \K$; by uniform continuity, 
there exists $0 < \delta \in \K$, such that, if
$x, y \in \K$ and $\abs{x - y} < \delta$, 
then $\abs{f(x) - f(y)} < \varepsilon$.
Let $y', y'' \in \K$
such that $y' < b < y''$ and $y'' - y' < \delta$ (they exist because $\Lambda$
is a regular cut).
Thus, $f(y') <  c < f(y'')$ and $f(y'') - f(y') < \varepsilon$, 
and therefore $\Gamma$ is  regular.
Moreover, $f$ is invertible, and therefore $b = f^{-1}(c) \in \K \gen c$.
\end{proof}

\begin{lemma}
Let $\K$ be a d-minimal structure, and $\K^* \succ \K$, such that $\K$ is
dense in $\K^*$.
Then, $\K^*$ has the exchange property relative to~$\K$;
that is, if $A \subset \K^*$, and $b, c \in \K^*$ satisfy
$c \in \K\gen{A,b} \setminus \K\gen A$, then $b \in \K\gen{A,c}$.
\end{lemma}
\begin{proof}
Let $b$ and $c$ be as in the hypothesis.
Let $\K' := \K \gen A$; \wloG, $\K = \K'$.
Then, since $\K$ has definable Skolem function,
$b = f^*(c)$ for some $K$-definable $f: \K \to \K$.
By Theorem~\ref{thm:imin}-\ref{en:i-monotonicity}
there exists $D \subseteq \K$ nowhere dense and
$K$-definable, such that $f$ is continuous and either constant or strictly
monotone on each sub-interval of $\K \setminus D$.
By Lemma~\ref{lem:regular-type}, $b \notin D^*$.
Therefore, by Lemma~\ref{lem:regular-type-image}, either $c \in \K$, 
or  $b \in \K \gen c$.
\end{proof}

Notice that the hypothesis that $\K$ is dense in $\K^*$ in the above lemma is
necessary: \cite[1.17]{DMS} shows that if $\K^*$ does not satisfy UF and it is
sufficiently saturated, then $\K^*$ does not satisfy the exchange property.

\begin{proposition}[Cauchy completion]\label{prop:dmin-Cauchy}
Let $\K$ be a d-minimal structure, expanding the  ordered field~$\Kb$,
$L = (0, 1, +, \cdot, <, \dotsc)$ be the language of~$\K$.
There exists a unique expansion of the Cauchy completion
$\KbC$ to an $L$-structure~$\KC$, such that
$\K$ is an elementary substructure of~$\KC$.
\end{proposition}

\begin{proof}
Let $\Sfam$ be set of elementary extensions $\K'$ of~$\K$, such that 
$\K$ is a dense in~$\K'$; order $\Sfam$ by elementary inclusion.
Let $\K'$ be a maximal element of~$\Sfam$ ($\K'$~exists by Zorn's lemma).
We claim that $\K' = \KC$ (and, thus, $\KC$~can be expanded to an $L$-structure).

Suppose not. \Wlog, $\K' = \K$.
Let $\Lambda$ be a regular gap of~$\K$, $\K^*$~be an elementary extension
of~$\K$, and $b \in \K^*$ filling the gap $\Lambda$.
Let $\K \gen b$ be the definable closure of $\K \cup \set b$ in~$\K^*$;
since $\K$ has definable Skolem functions (by definable choice),
$\K \prec \K \gen b \preceq \K^*$.

We claim that $\K$ is dense in $\K\gen b$, contradicting the maximality
of~$\K$.
In fact, let $c \in \K \gen b$; thus, $c = f(b)$ for some $\K$-definable
$f: \K^* \to \K^*$.
We have to prove that either $c \in \K$, or that $\Gamma$ is a regular gap,
where $\Gamma$ is the cut determined by $c$ over~$\K$.
By Theorem~\ref{thm:imin}-\ref{en:i-monotonicity}
there exists $D \subseteq \K$ nowhere dense and
definable, such that $f$ is continuous and either constant or strictly
monotone on each sub-interval of $\K \setminus D$.
By Lemma~\ref{lem:regular-type}, $b \notin D^*$.
Therefore, by Lemma~\ref{lem:regular-type-image}, either $c \in \K$, 
or $\Gamma$~is a regular gap.

It remains to prove that the $L$-structure on $\KbC$ is unique.
Let $\KC_1$ and $\KC_2$ be two expansion of $\KbC$ to elementary extensions
of~$\K$.
Let $\K'$ be a maximal common elementary substructure of $\KC_1$ and $\KC_2$
extending~$\K$.
Assume, for contradiction, that $\K' \neq \KC$; 
\wloG, we can assume that $\K = \K'$.
Let $b \in \KC \setminus \K$, and let $\phi(x)$ be an $L$-formula with
parameters in~$\K$.
In order to reach a contradiction, we must prove that $\KC_1 \models \phi(b)$
iff $\KC_2 \models \phi(b)$.
\Wlog, $\KC_1 \models \phi(b)$; let $X := \phi(\K)$.
Moreover, for every $Y \subseteq \K^n$ definable, 
let $Y^C_i$ be the interpretation of $Y$ in $\KC_i$, for $i = 1, 2$.
Notice that $X = U \cap D$, where $U := \inter X$ is open and definable, 
and $D := X \setminus \inter X$ is definable, with $\dim D = 0$.
Since $\K$ is d-minimal, $\rkCB(D) < \omega$.
Thus, by Lemma~\ref{lem:regular-type}, $b \notin D^C_1$.
Hence, $b \in U^C_1$; since $\K$ is dense in $\KC_1$, and $U^C_1$ is open,
there exist $y' < y'' \in \K$ such that $b \in (y', y'')^C_1 \subseteq U^C_1$.
Since $\K \preceq \KC_1$, $(y', y'') \subseteq U$,
and since $\K \preceq \KC_2$,
$b \in (y', y'')^C_2 \subseteq U^C_2 \subseteq X^C_2$.
\end{proof}

\begin{corollary}[of the proof]\label{cor:dmin-regular-cut}
Let $\K$ be a d-minimal structure, $\K^* \succ \K$,
$b \in \K^* \setminus \K$, and $\Lambda$ be the cut of $\K$ determined by~$b$.
If $\Lambda$ is a regular gap, then it uniquely determines the type of $b$
over~$\K$; moreover, $\K$ is dense in $\K \gen b$;
besides, for every $c \in \K \gen b \setminus \K$,
we have $\K \gen c = \K \gen b$. 
\end{corollary}

\begin{corollary}
An Archimedean locally o-minimal structure is o-minimal
\end{corollary}
\begin{proof}
If $\K$ is Archimedean, then $\KbC = \Rbar$; thus, $\K$ has an elementary
extension $\tilde \Real$ that is an expansion of~$\Real$.
Therefore, $\tilde \Real$ is o-minimal, and thus $\K$ is o-minimal.
\end{proof}

\subsubsection{Polish structures and theories}

\begin{lemma}\label{lem:Polish}
Let $\F$ be an ordered field.
Assume that $\F$ contains a countable dense subset (not necessarily
definable) and that $\F$ is Cauchy complete.
Then:
\begin{enumerate}
\item 
$\F$ has cofinality~$\omega$.
\item 
$\F$ is a Polish space (\ie, a Cauchy complete separable metric space).
\item
$\card{F} = 2^{\aleph_0}$.
\item 
If $X \subset \F^n$ is perfect, non-empty, and a $\Gd$ (in the topological
sense), then $\card{X} = 2^{\aleph_0}$.
\item
If $X \subset \F^n$ is closed and $\card{X} < 2^{\aleph_0}$, 
then $\isol(X)$, the set of isolated points of~$X$, is dense in~$X$.
\item
If $X \subseteq \F^n$ is a non-empty $\Gd$ (in the topological sense),
then it is a Baire space (again, in the topological sense), and it is even
strong Choquet~\cite{kechris}.
\end{enumerate}
\end{lemma}
\begin{proof}
(1) is obvious.

(2) requires us to define a metric.
If $\F$ is Archimedean, then $\F$ is homeomorphic to the reals, and we are
done.
Otherwise, let $v$ be the natural valuation on $\F$ induced by the ordering,
and $G$ be the value group of~$\F$.
\begin{claim}
The topology induced by $v$ on $\F$ is the same as the order topology.
\end{claim}
Notice that the claim is false if $\F$ is  Archimedean.

\begin{claim}
$G$ is countable.
\end{claim}
Thus, there exists a coinitial order-reversing embedding  $\iota$ of $(G, >)$ 
in $(\Real_+, <)$ as ordered sets (notice that $\iota$ ignores the group
structure).
For every $x, y \in \F$, define
\[
d(x, y) :=
\begin{cases}
\iota(v(x - y)) & \text{if } x \neq y;\\
0 & \text{otherwise}.
\end{cases}\]
\begin{claim}
$(\F, d)$ is a metric space.
\end{claim}
Actually, $(\F, d)$ satisfies the ultra-metric inequality.
\begin{claim}
$(\F,d)$ is homeomorphic to $(\F, v)$ (and hence to $(\F, <)$).
\end{claim}
\begin{claim}
If $(a_n)_{n \in \Nat}$ is a Cauchy sequence in $(\F, d)$,
then $(a_n)_{n \in \Nat}$ is a Cauchy sequence in $(\F, v)$.
\end{claim}
Hence, assume that $(a_n)_{n \in \Nat}$ is a Cauchy sequences in $(\F, d)$.
Since $(\F, v)$ is Cauchy complete (by assumption), $a_n \to a$ for some 
$a \in  \F$, according to the topology induced by~$v$.
However, $v$ and $d$ induce the same topology, and therefore $a_n \to a$
according also to~$d$; thus, $(\F, d)$ is a complete metric space.
Finally, $(\F, d)$ is separable by assumption.

$\card{\F} \leq 2^{\aleph_0}$ is easy.
The opposite inequality follows from (3).

Notice that a $\Gd$ non-empty subset of $\F^n$ is a Polish space~\cite[3.11]{kechris}.

(3) $X$ itself is a non-empty perfect polish space.
The conclusion follows from~\cite[6.2]{kechris}.

(4) Assume, for contradiction, that $\isol(X)$ is not dense in $X$;
let $B \subseteq \K^n$ be a closed box, such that $\inter{B} \cap X \neq
\emptyset$, and $Y := X \cap B$ contains no isolated points.
Hence, $Y$~satisfies the hypothesis of (3), absurd.

(5) Every Polish space is Baire and strong Choquet~\cite[8.17]{kechris}.
\end{proof}

\begin{definizione}
Let $T$ be a complete theory expanding the theory of ordered fields,
in a language~$\Lang$, expanding the language $\Langf$ of ordered fields.
We say that $T$ is a \intro{Polish theory} if, for every finite
language~$\Lang'$,  
such that $\Langf \subseteq \Lang' \subseteq \Lang$, the restriction of
$T$ to $\Lang'$ has a model which is separable and Cauchy complete.
If $T$ is not complete, we say that $T$ is a Polish theory if every completion
of $T$ is Polish.
\end{definizione}

\begin{lemma}
Let $T$ be a definably complete theory (expanding RCF).
If $T$ is Polish, $\K \models T$ and $X \subseteq \K^n$ is definable
and $\Gd$ (in the definable sense), then $X$ is definably Baire.
\end{lemma}
\begin{proof}
\Wlog, $\K$ is Cauchy complete and separable.
Hence, by Lemma~\ref{lem:Polish}, $X$ is topologically Baire, 
and \emph{a fortiori} definably Baire.
\end{proof}

\begin{lemma}
A \dminimal theory $T$ is Polish.
In particular, if $\K$ is \dminimal and $X \subseteq \K^n$ is definable,
then $X$ is Baire.
\end{lemma}
\begin{proof}
\Wlog, the language of $T$ is countable.
Let $\F'$ be a countable model of~$T$, and $\F$ be the Cauchy completion
of~$\F'$. 
$\F$ is a model of~$T$.
\end{proof}


\subsection{The Z-closure}
\label{sec:zclosure}

\begin{definizione}
Let $A \subseteq \K$., and $c \in \K$.
We define the \Zclosure of $A$ inside~$\K$
\[
\zclK(A) := \bigcup\set{C \subset \K: C \text{ nowhere dense and definable with parameters from } A}.
\]
If $\K$ is clear from the context, we drop the subscript~$\K$.
$A \subseteq \K$ is \Zclosed in $\K$ if $\zclK(A) = A$.
\end{definizione}
The notion above is most interesting when $A$ is an elementary substructure 
of~$\K$.

\begin{remark}
If $\K$ is o-minimal, then $\zcl = \dcl$.
\end{remark}

\begin{remark}
If $\K$ is o-minimal and $A \subseteq \K$, then $A$ is \Zclosed in $\K$ if
and only if $A$ is an elementary substructure of~$\K$.
\end{remark}

\begin{remark}
For every $A \subseteq \K$, $\dcl(A) \subseteq \zcl(A)$, and
$\zcl(\dcl(A)) = \zcl(A) = \zcl(\dcl(A))$.
\end{remark}

\begin{remark}
If $A$ has \DSF, then $\zclK(A) \preceq \K$.
\end{remark}

\begin{lemma}\label{lem:Z-interior-generator}
If $A \subseteq \K$ has non-emtpy interior, then $\dcl(A) = \K$, and therefore
$\zcl(A) = \K$.
\end{lemma}
\begin{proof}
Since $A \subseteq \dcl A = \dcl(\dcl(A))$, \wloG $A = \dcl A$.
Let $\varepsilon > 0$ and $a \in \K$
such that $B(a; \varepsilon) \subseteq A$.
Hence, $(-\varepsilon, \varepsilon) \subseteq A$.
Thus, $(1\varepsilon, +\infty) \subseteq A$.
Let $b \in \K$; we want to prove that $b \in A$; \wloG, $b > 0$.
Let $a := (1/\varepsilon)$ and $a' := b + 1/\varepsilon$;
notice that $a$ and $a'$ are in~$A$, and therefore $b = a' - a \in A$.
\end{proof}

\begin{remark}
Let $\K \prec \F$ be a dense substructure.
If $\K$ is \dminimal, then $\K$ is \Zclosed in~$\F$.
\end{remark}
\begin{proof}
By Lemma~\ref{lem:regular-type}.
\end{proof}

\begin{remark}
Given $A \subseteq \K$, $\dcl(A)$ does not depend on~$\K$:
that is, if $\K \preceq \K'$, then the definable closure of $A$ inside $\K$
and the definable closure of $A$ inside $\K'$ are the same set.
Instead,  $\zclK(A)$ may depend on~$\K$:
for instance, an infinite nowhere dense subset of $\K$ is $A$-definable
and if $\K'$ is a $\kappa$-saturated elementary extension of~$\K$,
then $\card{\zcl^{\K'}(A)} \geq \kappa$.
If $\K$ is \iminimal but not o-minimal, then $\dcl$ depends on~$\K$:
for instance there exists some $\K' \succ \K$ such that
$\zcl^{\K'}(\emptyset) \neq \zclK(\emptyset)$.
\end{remark}

\begin{remark}\label{rem:zcl-extension}
If $A \subseteq \K \preceq \K'$, then $\zclK(A) = \zcl^{\K'}(A) \cap \K$.
\end{remark}

\begin{definizione}
Let $f: X \app Y$ be a definable application
(\ie, a multi-valued partial function), with graph~$F$.
Assume that $\K$ is \iminimal.
For every $x \in X$, let $f(x) := \set{y\in Y: \pair{x,y} \in F}\subseteq Y$.
Such an application $f$ is a \intro{\Zapplication{}} if, for every $x \in X$, 
$\dim\Pa{f(x)} = 0$ (thus, the domain of $f$ is all~$X$);
it is a partial \Zapplication if for every $x \in X$, 
$\dim\Pa{f(x)} \leq 0$.
\end{definizione}

\begin{remark}
Let $A \subseteq \K$, and $b \in \K$.
Then, $b \in \zcl A$ iff there exists an $\emptyset$-definable
\Zapplication $f: \K^n \app \K$ and $\av \in A$,
such that $b \in f(\av)$.
Moreover, if $\cv \in \K^n$, then $b \in \zcl(A\cv)$ iff there exists
an $A$-definable \Zapplication $f: \K^n \to \K$, such that $b \in f(\cv)$.
\end{remark}
\begin{proof}
The ``if'' direction is clear: $f(\av)$ is nowhere-dense.
For the converse, let $Z \subset \K$ be nowhere-dense and $A$-definable,
such that $b \in Z$.
Let $\phi(x,\av)$ be the formula defining~$Z$.
Let $\psi(x,\y)$ be the formula ``($\psi(\x,\y)$ and $\psi(\K, \y)$ is 
nowhere-dense) or ($x = 0$ and $\phi(\K, \y)$ is somewhere-dense)''.
Then, $\psi$ defines a \Zapplication $f: \K^n \app \K$,
and $b \in f(\av)$.

The ``moreover'' part is clear.
\end{proof}

\begin{proviso}
For the rest of this subsection, $\K$ is \iminimal with \DSF.
\end{proviso}
We will show that, under the above condition, $\zclK$ is a matroid
(\aka combinatorial pregeometry); we will write $\zcl$ for $\zclK$.

\begin{definizione}
A formula $\phi(x,\y)$ is \xnarrow if, for every $\cv \in \K^n$,
$\phi(\K,\cv)$ is nowhere-dense.
\end{definizione}

\begin{remark}
For every $A \subseteq \K$,
\[\begin{aligned}
\zcl(A) &= 
\bigcup\set{C \subset \K: C \text{ is 0-dimensional and $A$-definable}} =\\
& = \bigcup\set{\phi(\K,\av): \phi(x,\y) \text{ is \xnarrow and } 
\av \subseteq A}.
\end{aligned}\]
If $\K$ is \dminimal, then
\[
\zcl(A) = 
\bigcup\set{C \subset \K: C \text{ is discrete and $A$-definable}}.
\]
\end{remark}

\begin{remark}\label{rem:Z-multi}
Let $A \subseteq \K$, and $D \subseteq \K^n$ be $A$-definable and
$0$-dimensional.
If $\bv \in D$, then each coordinate of $\bv$ is in $\zcl A$.
\end{remark}
\begin{proof}
It is enough to show that $b_1 \in \zcl A$.
Let $D_1 := \Pi^n_1 D$.
Then, $D_1$ is 0-dimensional, and $b_1 \in D_1$.
\end{proof}

\begin{lemma}
$\zcl$ is transitive.
That is, $\zcl(\zcl A) = \zcl A$.
\end{lemma}
\begin{proof}
Let $b \in \zcl(\zcl A)$.
Then, there exists $\cv \in (\zcl A)^n$ and an \xnarrow formula $\phi(x,\y)$,
such that $\K \models \phi(b, \cv)$.
For every $i \leq n$, let $Y_i$ be an $A$-definable nowhere-dense set
containing~$c_i$.
Let $\psi(x, \y)$ be the $L(A)$-formula
\[
\phi(x, \y) \wedge \bigwedge_{i \leq n} y_i \in Y_i
\]
and $Z := \phi(\K^{n + 1})$.

\begin{claim}
$\dim Z = 0$.
\end{claim}
By Lemma~\ref{lem:imin-dimension-fiber}.

Since $Z$~is $A$-definable, and $\pair{b, \cv} \in Z$,
Remark~\ref{rem:Z-multi} implies that  $b \in \zcl A$.
\end{proof}

\begin{lemma}
$\zcl$ has the exchange property.
That is, if $a \in \zcl(Bc) \setminus \zcl (B)$, then $c \in \zcl(Ba)$.
\end{lemma}
\begin{proof}
Assume $a \in \zcl(Bc)$ and $c \notin \zcl(Ba)$;
we want to conclude that $a \in \zcl(B)$.
Let $\phi(x,y)$ be an \xnarrow $L(B)$-formula, such that
$\K \models \phi(a,c)$.
Define $\psi(x,y)$ as
\[
\exists r > 0\ \forall y'\ \Pa{\abs{y - y'} < r \ \rightarrow \phi(x,y')};
\]
that is, $\K \models \phi(x,y)$ iff $y \in \interior(\phi(x,\K))$.
Notice that $\K \models \psi(a, c)$, because 
$\phi(a,\K) \setminus \psi(a,\K)$ has dimension~0.
Moreover, for every $a' \in \K$, $\psi(a', \K)$ is open:
let $\theta(x,y)$ be the formula
\[
\psi(x,y) \et y \text{ is the centre of a definably connected component of } \psi(x, \K).
\]
Notice that, for every $a' \in \K$, $\theta(a', \K)$ is discrete.
Let $d \in D$ be the centre of the connected component of $\psi(a, \K)$
containing~$c$, and $Z := \theta(\K^2)$.

\begin{claim}
$\dim Z = 0$.
\end{claim}
Let $\pi := \Pi^2_1$, and $\mu: \K^2 \to \K$ be the projection onto the second
coordinate.
If both $\pi(Z)$ and $\mu(Z)$ have dimension~0, then the claim is true.
If $\dim(\pi(Z)) = 0$, then, since $\dim(Z_x) \leq 0$ for every~$x$,
Lemma~\ref{lem:imin-dimension-fiber} implies that $\dim Z = 0$.
If, for contradiction, $\dim(\pi(Z)) > 0$, 
then, since $Z \subseteq \psi(\K^2)$, and $\psi(x, \K)$ is open for every~$x$,
\DSF and Theorem~\ref{thm:imin}(\ref{en:i-monotonicity}) imply that
there exist $g < h: I \to \K$ $B$-definable and continuous,
such  that, for every $x \in I$, $(g(x), h(x)) \subseteq \psi(x,\K)$.
Hence, $\psi(\K^2)$ contains an open set, and is contained in $\phi(\K^2)$,
contradicting the fact that $\phi(x,y)$ is \xnarrow.

Since $\pair{a,d} \in Z$ and $Z$ is $B$-definable, the claim and 
Remark~\ref{rem:Z-multi}  imply that $a \in \zcl(B)$.
\end{proof}

The above two lemmata imply that $\zcl$ is a matroid.
Hence, we can speak about \Zgenerating and \Zindependent sets, \Zbasis and
\Zdimension: $E$~is \Zgenerating set of $B/A$ if $\zcl(A E) = \zcl(B)$;
$E$~is \Zindependent over $A$ if, for every $e_1, \dotsc, e_{n+1} \in E$,
$e_{n+1} \notin \zcl(A e_1 \dots e_n)$;
$E$ is a \Zbasis if it is both \Zgenerating and \Zindependent.
The \Zdimension of $B/A$ is $\Zdim(B/A)$, 
the cardinality of some \Zbasis of $B/A$.
It is important to notice that the above notions do not depend on the ambient
space~$\K$: by Remark~\ref{rem:zcl-extension}
if $A \subseteq B \subseteq \K \preceq \K'$, and $E \subseteq \K$,
then $E$ is a \Zgenerating set of $B/A$ (\resp \Zindependent over~$A$,
\resp a \Zbasis of $B/A$) in $\K$ iff it is a \Zgenerating set of $B/A$ (\resp
\Zindependent over~$A$, \resp a \Zbasis of $B/A$) in $\K'$.
Hence, $\Zdim(B/A)$ is also independent from~$\K$.

Let $\monster \succ \K$ be ``the'' monster model.
Let $\kappa$ be an infinite cardinal, such that $\kappa > \card{T}$,
where $T$ is the theory of~$\K$; we will assume that
 assume that $\card{\K} < \kappa$ and $\monster$ is $\kappa$-saturated
and strongly $\kappa$-homogeneous.
We will say that $A$ is a proper subset of $\monster$ if
$A \subset \monster$ and $\card A < \kappa$.

\begin{lemma}
Let $A$ be a proper subset of~$\monster$, and $c \in \monster$.
Let $\Xi(c/A)$ be the set of conjugates of $c$ over~$A$.
\Tfae:
\begin{enumerate}
\item $c \in \zcl(A)$;
\item $\Xi(c/A)$ has empty interior;
\item $\Xi(c/A)$ is nowhere dense.
\end{enumerate}
If moreover $\monster$ is \dminimal, then $c \in \zcl(A)$ iff $\Xi(c/A)$ is discrete.
\end{lemma}
\begin{proof}
$(1 \Rightarrow 3 \Rightarrow 2)$ is clear.

$(2 \Rightarrow 1)$.
Let $\phi(x)$ be any $\Lang(A)$-formula, such that $\monster \models \phi(c)$.
Since $c \notin \zcl(A)$, $c \in \interior(\phi(\monster))$,
and therefore there exist $d, d' \in \monster$, such that
$d < c < d'$ and $(d, d') \subseteq \phi(\monster)$.
Let $\Gamma(v, v')$ be the set of $\Lang(A c)$-formulae
\[
v < c < v' \et (v, v') \subseteq \phi(\monster),
\]
where $\phi(x)$ varies in $\tp(c/A)$.
By what we said above, $\Gamma$ is consistent: hence, by saturation,
there exist $d, d' \in \monster$ satisfying $\Gamma$.
We claim that $(d, d') \subseteq \Xi(c/A)$.
In fact, if $c' \in (d, d')$, then by definition, $c'$ satisfies all
the $\Lang(A)$-formula satisfied by~$c$: therefore, $\tp(c'/A) = \tp(c/A)$,
and, by homogeneity, $c' \in \Xi(c/A)$.

Assume now that $\monster$ is \dminimal.
If $c \in \zcl(A)$, then there exists a discrete set~$X$, definable with
parameters from~$A$, such that $c \in X$; therefore, $\Xi(c/A) \subseteq X$,
and thus $\Xi(c/A)$ is discrete.
The converse is clear.
\end{proof}

\begin{corollary}
$\zclM$ is an existential matroid.
\Ie, let $\av \in \monster^n$, $\bv \in \monster^m$, $c \in \monster$,
and $\phi(x,\y,\z)$ be \xnarrow.
Assume that, for every conjugate $c'$ of $c$ over $\av$, 
$\monster \models \phi(c', \av, \bv)$.
Then, $c$ (and all its conjugates over~$\av$) is in $\zclM(\av)$.
\end{corollary}
\begin{proof}
It is \cite[Theorem~9.8]{fornasiero-matroids}.
The corollary is also a direct consequence of 
\cite[Lemma~3.22]{fornasiero-matroids}.
\end{proof}

\begin{definizione}
Given $A$, $B$, $C$ proper subsets of $\monster$, we say that $A$ and $C$
are \Zfree over~$B$, written $A \zind_B C$,
if some (every) \Zbasis of $A$ over $B$ remains \Zindependent over $BC$.
\end{definizione}

The above three lemmata imply the following result.

\begin{thm}\label{thm:Z-matroid}
$\zcl$ is an existential matroid.
Hence, $\zind$~is an independence relation in the sense of~\cite{adler}
(and in particular it is symmetric), and
satisfies $a \zind_B a$ iff $a \in \zclK(B)$, for every $a \in \K$,
$B \subseteq \K$.
\end{thm}

Notice that if $\K$ is o-minimal, then $\zcl = \dcl$, and therefore
$\zind = \tind$.
The converse is also true 
(remember the assumption that $\K$ is \iminimal with \DSF).
\begin{lemma}
\Tfae:
\begin{enumerate}
\item $\zclM = \dcl$;
\item $\zind = \tind$ (in the monster model);
\item $\K$ is o-minimal.
\end{enumerate}
\end{lemma}
\begin{proof}
($3 \Rightarrow 1$) is clear.
If (1) holds, then $\zind = \mind$; moreover, since $\zind$ is symmetric,
$\mind$ is also symmetric, and therefore $\mind = \tind$ \cite{adler}.

Assume that (2) holds.
Let $a \in \zcl(B)$.
Then, $a \zind_B a$, therefore $a \tind_B a$, and thus
$a \in \dcl B$: hence, (1) also holds.
We have to prove that $\monster$ is o-minimal.
Assume, for contradiction, that $A \subset \monster$ is definable with 
parameters~$\bv$, infinite and with empty interior.
Then, $A \subseteq \zclM(\bv) = \dcl(\bv)$.
However, since $A$ is infinite, $\card{A} \geq \kappa$, and therefore
$\card{\dcl(\bv)} \geq \kappa > \card T$, which is impossible.
\end{proof}

\begin{lemma}
The dimension induced by $\zind$ and the geometric notion of dimension
coincide.
That is, if $X \subseteq \monster^n$ is definable, then
$\dim X = \max \set{\RK^{Z}(\x): \x \in X}$.
\end{lemma}

Contrast the situation of $\zind$ to the notion of $M$-dividing independence
(defined in~\cite{adler}), where, 
$A \mind_B C$ iff, 
for every $\cv \subset \dcl(B C)$, 
\[
\dcl(A B \cv) \cap \dcl(B C) = \dcl(B \cv).
\]
\begin{lemma}\label{lem:asymmetry}
Assume that $T$ is \dminimal, but \emph{not} o-minimal.
Then, $\mind$~is not symmetric
(and therefore $\dcl$ does not have the Exchange Property).
However, $\mind$~does satisfy the existence and extension axioms, 
and therefore coincides with~$\tind$, the \th-forking relation.
Hence, $T$~is not rosy,
and in particular $\tind$ is not symmetric.
\end{lemma}
\begin{proof}
The fact that $\mind$ satisfy existence and extension and coincides with 
$\tind$ is immediate from \cite[Lemma~3.22]{fornasiero-matroids}.

Let $\K \prec \monster$ such that $\K$ is not Cauchy complete, and $\card{\K}
< \kappa$.
By expanding the language with less than $\kappa$ new constants, \wloG we can
assume that $\K$ is the prime model of~$T$.
Let $\pi$ be a regular gap of~$\K$; choose $c_0$ and $c_1$
such that $c_1 > \K$ and $c_0 \models \pi$.
\begin{claim}
There exist $a_0$ and $b$, and a finite set $C$ such that:
\begin{enumerate}
\item $c_0, c_1 \in C$;
\item $b \in \zcl(C)$;
\item $b \in \dcl(C a_0) \setminus \dcl(C)$;
\item $a_0 \notin \zcl(Cb)$;
\item $a_0 > 0$ and $a_0$ is infinitesimal \wrt~$\K$.
\end{enumerate}
\end{claim}
Let $X$ be a definable (with parameters) subset of $\monster$ which is
discrete and infinite (such a set exists by Corollary~\ref{cor:i-o-min},
because $T$ is \iminimal but not o-minimal), and let $C$ be any finite set
containing $c_0$, $c_1$, and the parameters of~$X$. 
By saturation, there exists $b \in X \setminus \dcl(C)$.
Let $I$ be an open interval containing~$b$, such that $I \cap X = \set b$.
By Lemma~\ref{lem:Z-interior-generator}, 
since $\Zrk(\monster /\K) \geq \kappa$, 
there exists $a_0' \in I \setminus \zcl(Cb)$.
\Wlog, $a_0' > 0$.
If $a_0'$ is infinitesimal (\wrt~$\K$), let $a_0 := a_0'$;
if $a_0'$ is finite but not infinitesimal, let $a_0 := a_0' / c_1$;
if $a_0'$ is infinite, let $a_0 := 1/ a_0'$.
Notice that $\zcl(C b) = \zcl(C)$.
\begin{claim}
With $b$ and $C$ as in the above Claim, there exists $a \in \monster$ such
that:
\begin{enumerate}
\item $a \models \pi$ (and therefore $a$ is in a Cauchy completion of~$\K$);
\item $a \notin \zcl(C b)$;
\item $b \in \dcl(Ca) \setminus \dcl(C)$.
\end{enumerate}
\end{claim}
Let $a := c_0 + a_0$.
Since $a_0$ is infinitesimal, $a \models \pi$.
Since $a_0 \notin \zcl(C)$ and $c_0 \in C$, the second point follows.
Since $\dcl(C a) = \dcl(C a_0) \ni b$, the third points follows.
\begin{claim}
$a \notind[M] C b$.
\end{claim}
In fact, let $C' := C$; then, $b \in \dcl(C' a) \cap \dcl(C b)$,
but $b \notin \zcl(C')$.
\begin{claim}
$C b \mind a$.
\end{claim}
In fact, let $A := \dcl(a) = \K \gen a$, and $A' \subseteq A$.
Define $Y := \dcl(C A' b) \cap A$; we have to prove that $Y = A'$.
Since $a$ satisfies a regular gap over~$\K$, $\dcl(\emptyset) = \K$ is dense
in~$A$; therefore, $\dcl^A$ satisfies EP.
Hence, either $A' = \K$, or $A' = A$.
If $A' = A$, the conclusion is obvious.
If $A' = \K$, then $a \notin Y$, because $Y \subset \zcl(C b)$, 
and $a \notin \zcl(C b)$; therefore, since $\dcl^A$ satisfies EP, $Y = \K$.
\end{proof}

We do not know if the above lemma extends to \iminimal theories with \DSF,
or to \ipminimal theories.

\begin{lemma}
Let $f: \K \app \K$ be a \Zapplication, definable with parameters~$\cv$.
Let $b \in \K$ and $U \subseteq \K$ be non-empty and open, such that, for
every $a \in U$, $b \in g(a)$.
Then, $b \in \zcl(\cv)$.
\end{lemma}

\subsection{Dense pairs of d-minimal structures}
\label{subsec:dense}

Dense pairs of o-minimal structures were studied in~\cite{vdd-dense}.
\begin{proviso}
$\K$ is an \ipminimal structure with \DSF, and $T := \Th(\K)$.
\end{proviso}

We have seen that the \Zclosure is an existential matroid on~$\K$.
Moreover, $A \subseteq \K$ is topologically dense iff it s dense \wrt to the
matroid $\zcl$, that is iff $X$ intersects every definable subset of $\K$ of
dimension~$1$.

We can apply the results in~\cite{fornasiero-matroids} to~$T$, and obtain
the following results.

\begin{thm}
Let $\Td$ be the theory of pairs $\Am \prec \Bm \models T$, 
such that $\Am$ is dense in $\Bm$ and $\zcl(\Am) = \Am$.
Then, $T$ is consistent and complete.
Besides, $\Bm$ is the open core of $\pair{\Bm, \Am}$.
If moreover $T$ is \dminimal, then $\Td$ is the theory of 
pairs $\Am \prec \Bm \models T$, such that $\Am$
is dense in $\Bm$ (the fact that $\Am$ is $\zcl$-closed in $\Bm$ follows).
\end{thm}
Similar results can be shown for dense tuples of models of $T$
\cite[\S13]{fornasiero-matroids}.

More results can be proved for $\Td$, \eg a form of elimination of quantifiers~\cite{fornasiero-matroids}.
If moreover $T$ is \dminimal, then also the results in~\cite[\S9]{fornasiero-matroids} apply to~$T$.

We will give some additional results and conjectures that are more psecific to
our situation.

\begin{thm}[{\cite[Theorem~2]{vdd-dense}}]
Let $\pair{\Bm, \Am} \models \Td$.
Given a set $Y \subset \Am^n$, \tfae:
\begin{enumerate}
\item 
$Y$ is definable in $\pair{\Bm, \Am}$;
\item 
$Y = Z \cap \Am^n$ for some set $Z \subseteq \Bm^n$ that is definable in~$\Bm$.
\end{enumerate}
If moreover $T$ is \ipminimal, then the above two conditions are equivalent
to:
\begin{enumerate}
\setcounter{enumi}{2}
\item 
$Y$ is definable in the structure $\pair{\Am, (\Am \cap (0,b))_{b \in \Bm}}$.
\end{enumerate}
\end{thm}
Remember that \ipminimal structures are \iminimal with \DSF. 

\begin{proof}
$(1 \Rightarrow 2)$ and, under the extra condition on~$T$, $(2 \Rightarrow 3)$
are as in \cite[Theorem~2]{vdd-dense}.
$(3 \Rightarrow 1)$ and $(2 \Rightarrow 1)$ are obvious (and true without the
extra condition on~$T$).
\end{proof}
We do not know whether we do really need the extra condition that $T$ is
\ipminimal to prove that (2) implies (3) in the above Theorem.

\begin{question}
Let $T'$ be a complete \aminimal theory.
Is there an existential matroid on~$T'$?
\cite[6.2]{DMS} and \cite{fornasiero-matroids} 
prove that if $T$ is equal to either
$T^d$ or $T^g$ (see \cite{DMS} for the definition of $T^g$) 
for some o-minimal theory~$T$, then $T'$ admits such a matroid
(in the case of $T^g$, the matroid is~$\acl$).
\end{question}

\begin{example}
We show that $\zcl$ does not satisfy Existence on~$\monster := \pair{\Bm, \Am}$, a monster model of~$\Td$. 
Write $x' \elem^2_C x$ if the $\Ltwo$-type of $x$ and $x'$ over $C$ are the
same, and let $\Xi^2(x / C) := \set{ x' \in \monster: x' \elem^2_C x}$.

Choose $a_1, a_2 \in \Am$ and $b \in \Bm \setminus \Am$, such that
$a_1$ and $a_2$ are \Zindependent over~$b$.
Let $c := a_1 b + a_2$, and $f: \Am^2 \to \Bm$ be the definable function
$f(x_1, x_2) := x_1 \cdot c + x_2$.
By hypothesis, $\zrk(a_1 a_2/ b) = 2$, and therefore
$\zrk(a_1, a_2 / b c) \geq 1$.
Thus, either $a_1 \notin \zcl(b c)$, or $a_2 \notin \zcl(b c)$;
\wloG, $a_1 \notin \zcl(b c)$.
However, $f$ is injective, and therefore 
$a_1$ and $a_2$ are $\Ltwo$-definable over $b c$,
hence, $\Xi^2(a_1/ b c) = \set{a_1} \subseteq \zcl(b c a_1)$.
If $\zcl$ did satisfy existence, then $a_1 \in \zcl(b c)$, absurd.
\end{example}

\subsubsection{The open core}
\begin{proviso}
For this subsection, we assume that $T$ is \dminimal.  
Let $\monster := \pair{\Bm, \Am} \models \Td$.
\end{proviso}

We have seen that $\Bm$ is the open core of~$\monster$.
Hence, since every $\Fs$ subset of $\monster^n$ is definable in the open core
of $\monster$, every such set is constructible.
We will prove some additional results about this topic.
$\scl$ is the \intro{small closure} on $\monster$ and $\sdim$ is the
corresponding dimension function, as defined in~\cite{fornasiero-matroids}.

\begin{lemma}\label{lem:cl-U}
Let $(X_t)_{t \in \monster}$ be a definable increasing family of subsets of $\monster^n$, and $X := \bigcup_t X_t$.
Let $d \leq n$ and assume that, for every $t \in \monster$,
$\sdim(X_b) \leq d$.
Then, $\sdim(X) \leq d$.
\end{lemma}
\begin{proof}
\cite[Lemma~3.71]{fornasiero-matroids}, applied to $\scl$.
\end{proof}

\begin{corollary}
$\monster$ is Baire.
\end{corollary}
\begin{proof}
By~\cite{fornasiero-matroids}, if $X \subseteq S$ is definable in
$\monster$ and nowhere-dense, then $\sdim X = 0$.

The conclusion then follows from Lemma~\ref{lem:cl-U}.
\end{proof}

\begin{lemma}
$\Td$ is a Polish theory.
\end{lemma}
\begin{proof}
\Wlog, the language of $T$ is countable.
Let $\Am'$ be a countable model of $T$ and $\Bm'$ be its Cauchy completion.
Notice that $\Bm' \neq \Am'$ and therefore $\pair{\Bm', \Am'}$
is a Cauchy complete and separable model of~$\Td$.
\end{proof}

\section{Types in locally o-minimal structures}
As usual, $\K$ is a definably complete structure.
Let $+ \infty$ be the partial 1-type over~$\K$, 
given by $\set{x > a: a \in \K}$.

\begin{remark}
\Tfae:
\begin{itemize}
\item $\K$ is locally o-minimal;
\item $+ \infty$ is a complete type;
\item for every cut $\Lambda$ of~$\K$,
if $\hat \Lambda = 0^+$, then $\Lambda$ is a complete type over $\K$.
\end{itemize}
\end{remark}
\begin{proof}
See Lemma~\ref{lem:lmin} and Corollary~\ref{cor:dmin-regular-cut}.
\end{proof}

\begin{lemma}
Let $\K$ be locally o-minimal, and $p \in S_1(\K)$.
$p$ is a definable type iff either $p = \pm \infty$,
 or $p = a^\pm$, or $p = a$, for some $a \in \K$.
\end{lemma}
\begin{proof}
By the above remark, the listed types are complete types.
The ``only if'' direction is clear.
For the other direction, let $p = + \infty$; we have to prove that $p$ is
definable (the other cases are similar).
Fix $c$ a realization of $p$ in some $\K' \succ \K$.
Let $\phi(x, \y)$ be a formula, 
and $D := \set{\y \in \K^n: \K \models \phi(c, \y)}$.
We have to prove that $D$ is a definable subset of~$\K$.
Let $F := \phi(\K')$ and $F_c := \phi(c, \K')$;
notice that $D = F_c \cap \K^n$.
Define
\[
G := \liminf F_t = \set{\av \in \K'^n: \exists t_0 \forall t > t_0\ \av \in F_t}.
\]
Notice that $G$ is definable without parameters; 
let $G' := G(\K) = G \cap \K^n$.
\begin{claim}
$D = G'$.
\end{claim}
The claim implies the conclusion.
Let us prove that $d \subseteq D$; let $d \in D$.
Thus, $d \in F_a \cap \K$, and therefore ``$d \in F_x$'' 
is in the type of $c$ over~$\K$.
Thus, $d \in F_t$ eventually, and therefore $d \in G$.
The opposite inclusion is trivial.
\end{proof}

\begin{question}
Give a characterization of definable $n$-types in locally o-minimal structure,
along the lines of Marker-Steinhorn's theorem.
\end{question}


\section{D-minimal open core}

We have seen that if $T$ is \dminimal, then $\Td$ has \dminimal open core.
We want to give some characterization of when a structure $\K$ as \dminimal
open core. 
Unfortunately, we were not able to prove what we wanted, so here is a
conjecture. 

\begin{conjecture}
Assume that, for every $\K' \succ \K$, every $\Fs$ subset of $\K'$ is the
union of an open set and finitely many discrete sets.
Then, for every $n \in \Nat$, every $\Fs$ subset of $\K^n$ is constructible.
Moreover, $\K$~is Baire.
\end{conjecture}
In the above situation, we say that $\K$ has \dminimal open core.
\begin{proof}[Idea of proof]
If $\K$ is \aminimal, then $\K$ has locally o-minimal open core, and \textit{a
  fortiori} \dminimal open core.
Hence, we may assume that $\K$ is not \aminimal, and therefore there exists a
\pN subset of~$\K$.

First, we prove that $\K$ is Baire.
If not, let $N \subseteq \K$ be a \pN subset of~$\K$, and
$\Cfam := \Pa{C_t}_{t \in N}$ be a definable family of \dcompact nowhere dense
subsets of $\K$, such that $\K = \bigcup_{t \in N} C_t$.
By assumption, there exists $M \in \Nat$ such that each $C_t$ is a finite
union of $M$ discrete sets.
Therefore, $\Cfam$ is a uniform family of (at most) pseudo-enumerable sets,
and therefore $\K$ is pseudo-enumerable, absurd.

Let $A \subseteq \K^n$ be an $\Fs$ set.
We have to prove that $A$ is constructible.
\Wlog, $\K$ is $\omega$-saturated.
We proceed by induction on $n$ and on $(d, k) := \fdim(A)$.
If $n = 1$, $A$ is constructible by hypothesis.
Hence, \wloG $n > 1$.
If $d = n$, let $B := A \setminus \inter A$.
Thus, $A = \inter A \cup B$.
Since $\dim(B) < d$, by induction on~$d$, $B$~is constructible, and
therefore $A$ is constructible.
If $d = 0$, for $i = 1, \dotsc, n$, let $A_i$ be the projection of $A$ onto
the $i$th coordinate axis.
By assumption on~$A$, $A_i$ is an $\Fs$ with empty interior, and therefore, by
hypothesis, $A_i$ is a finite union of discrete sets.
Moreover, $A \subseteq A_1 \times \dots \times A_n$,therefore $A$ is
a finite union of discrete sets, and hence $A$ is constructible.

Hence, we may assume that $0 < d < n$.
Let $\pi := \Pi^n_d$; \wloG, $\pi(A)$ has non-empty interior.
Let $C := \nlc A = A \setminus \lc(A)$.
\begin{claim}
\Wlog, $A$ is bounded.
\end{claim}
In fact, it suffices to prove that $\phi(A)$ is constructible, where $\phi$ is
a homeomorphism between $\K^n$ and $(0,1)^n$.

Given $X \subseteq \K^d$, let $A_X := A \cap \pi^{-1}(X)$.

\begin{claim}
\Wlog, $\pi(C) = \pi(A) = B$ for some closed box~$B$ with non-empty interior.
\end{claim}
If $\fdim(C) < \fdim(A)$, then, by inductive hypothesis, $C$ is constructible,
and therefore $A$ is also constructible.
Thus, $\pi(C)$ contains a closed box~$B' \subseteq B$ with non-empty interior.
If we prove that $A_{B'}$ is constructible for any such~$B'$, then, 
since $\K$ is $\omega$-saturated, there exists $M \in \Nat$ such that
$\inlc {A_B'} M = \emptyset$ for every such box~$B'$.
Hence, $\inlc A M \cap \pi^{-1}(B') = \emptyset$ for every closed box 
$B' \subseteq \pi(C)$ with non-empty interior, 
and therefore $\fdim(\inlc A M) < \fdim A$.
Thus, by inductive hypothesis, $\inlc A M$ is constructible, and thus $A$ is
also constructible. 

By hypothesis, there exists an increasing definable family
$\Pa{A(t)}_{t \in \K}$ of \dcompact subsets of~$\K^n$, such that 
$A = \bigcup_t A(t)$.
\begin{claim}
\Wlog, $\pi(A(t)) = B$ for every $t \in \K$.
\end{claim}
In fact, since $\K$ is Baire, there exists $t_0 \in \K$ and $B' \subseteq B$
closed box with non-empty interior, such that $B' \subseteq \pi(A(t_0))$.
If we prove that $A_{B'}$ is constructible for any such box~$B'$ contained in
some~$A(t)$,
then, since $\K$ is $\omega$-saturated, there exists $M \in \Nat$ such that
$\inlc {A_B'} M = \emptyset$ for every such box~$B'$.
Since every $x \in B$ is contained in some $B'$ as above,
$\fdim(\inlc A M) < \fdim A$, and therefore, by inductive hypothesis, $A$ is
constructible. 
Thus, \wloG $B \subseteq \pi(A(t_0))$ for some $t_0 \in \K$.
Define $A'(t) := A(t_0)$ if $t \leq t_0$, and $A'(t) = A(t)$ if $t > t_0$.
Then, $\pi(A'(t)) = B$, each $A'(t)$ is \dcompact, and 
$A = \bigcup_t A'(t)$.

\begin{claim}
\Wlog, $\dim(A_x) = 0$ for every $x \in B$.
\end{claim}
In fact, let $D := \set{x \in \K^d: \dim(A_x) > 0}$.
By Lemma~\ref{lem:dim-fiber}, $D$ is an $\Fs$ of dimension less than~$d$.
Hence, by induction on $\fdim(A)$, $A_D$~is constructible.
Thus, it suffices to prove that $A \setminus A_D$ is constructible.

By assumption, and since $\K$ is $\omega$-saturated,
there exists $M \in \Nat$, such that, for every $x \in \K^d$, $A_x$ is the
union of at most $M$ discrete sets.

\begin{claim}
\Wlog, for every $x \in B$, $A_x$ is discrete.
\end{claim}
\textbf{To be done}.

\begin{claim}
$\lc(A) \neq \emptyset$.
\end{claim}
Assume, for contradiction, that $\lc(A)$ is empty.
Let $\tilde A := A(0)$; remember that $\tilde A$ is \dcompact and $\pi(\tilde
A) = B$; let $f: B \to \K$, $f(x) := \min(\tilde A_x)$.
To simplify the notation, assume that $n = d + 1$.
By~\cite{FS}, $\Dis(f)$ is a meager $\Fs$ subset of~$\K^d$;
thus, by inductive hypothesis, $\Dis(f)$ is nowhere dense.
Hence, after shrinking~$B$, we may assume that $f$ is continuous.
We want to prove that, after possibly further shrinking~$B$, 
$\Gamma(f) \subseteq \lc(A)$.
Let $A^+ := \set{\pair{x, y}: (y \in A_x \et y > f(x)) \vee y = f(x) + 1}$.
Since $f$ is continuous and $A$ is an~$\Fs$, $A^+$ is an $\Fs$~sets. 
Define $f^+: B \to \K$ as $f^+(x) := \min (A^+_x)$, and define in a symmetric
fashion $f^-: B \to \K$.
Notice that $f^- < f < f^+$ (because each $A_x$ is discrete).
We claim that $f^+$ is continuous outside a nowhere dense subset of~$B$.
Let $A_+ := \bigcup_{t \in N} C(t)$, where $N$ is a \pN subset of~$\K$,
and $\Pa{C(t)}_{t \in N}$ is a definable increasing family of \dcompact sets, 
such that $B \times \set 1 \subseteq C(t)$ for every~$t$.
Let $f_t(x) := \min(C(t))$, and $D(t) := \Dis(f_t)$.
Each $D(t)$ is a meager $\Fs$ subset of~$\K^d$, hence, by inductive
hypothesis, each $D(t)$ is nowhere dense.
Hence, $D := \bigcup_{t \in N} D(t)$ is meager, and therefore, by inductive
hypothesis, nowhere dense; let $B' := B \setminus \cl D$, 
and $g := f^+ \rest {B'}$.
Since $B'$ is open and definable, $B'$ is Baire; moreover, $g$ is the
pointwise limit of the continuous functions $f_t \rest{B'}$;
thus, by Lemma~\ref{lem:first-class} and inductive hypothesis, 
$g$~is continuous outside a nowhere dense subset of~$B'$.
Reasoning in the same way for $f^-$, we see that there exists $B''$ open dense
subset of~$B$, such that both $f^+$ and $f^-$ are continuous on~$B''$.
Hence, $\Gamma(f) \cap \pi^{-1}(B'') \subseteq \lc(A)$ 
and therefore the latter is non-empty.

\begin{claim}
$\lc(A)$ is dense in~$A$.
\end{claim}
For every $U \subseteq \K^n$ open box, such that $A \cap U$ is non-empty, the
above claim implies that $\lc(A) \cap U = \lc(A \cap U)$ is non-empty.

Let $A' := \lc(A)$.
By definition, $A' = \cl{A'} \cap U$, for some open definable set~$U$.
By Lemma~\ref{lem:bad-2}, $\B_d(U)$ is an $\Fs$ meager subset of~$\K^d$, and
therefore, by inductive hypothesis, $\B_d(U)$ is constructible, and therefore
nowhere dense.
%
%
Hence, after shrinking~$B$, we may assume that $\B_d(U)$ is empty.


\begin{claim}
$\B_d(A')$ is empty.
\end{claim}
Again, by Thm.~\ref{thm:lc-fiber}.

Since $A' \subseteq A \subseteq \cl{A'}$, Remark~\ref{rem:bad-3} implies that
$\B_d(A) \subseteq \B_d(A') = \emptyset$.

Therefore, $A' = \lc(A) \subseteq A$, $\cl{A'} = \cl{A}$, and
$\cll(A'_x) = (\cl{A'})_x = (\cl A)_x = \cll(A_x)$ for every $x \in B$;
hence, $A'_x$ is dense in $A_x$ for every $x \in B$.
Since, for every $x \in B$, $A_x$~is discrete, this means that $A_x = A'_x$;
thus, $A' = A$, and hence $A$ is constructible.
\end{proof}




\section*{Acknowledgments}
Some of the people to whom I asked questions and have provided useful answers
(more or less relating to this article):
H.~Adler, A.~Berarducci, D.~Ikegami, C.~Miller, T.~Servi, K.~Tent, M.~Ziegler.


\bibliographystyle{alpha}	
\bibliography{tame}		

\begin{thebibliography}{vdD98b}

\bibitem[Adl05]{adler}
Hans Adler.
\newblock {\em Explanation of Independence}.
\newblock PhD thesis, Albert-Ludwigs-Universit\"at, Freiburg im Breisgau, June
  2005.

\bibitem[All96]{allouche96}
Jean-Paul Allouche.
\newblock Note on the constructible sets of a topological space.
\newblock In {\em Papers on general topology and applications ({G}orham, {ME},
  1995)}, volume 806 of {\em Ann. New York Acad. Sci.}, pages 1--10. New York
  Acad. Sci., New York, 1996.

\bibitem[DM01]{DM}
Randall Dougherty and Chris Miller.
\newblock Definable {B}oolean combinations of open sets are {B}oolean
  combinations of open definable sets.
\newblock {\em Illinois J. Math.}, 45(4):1347--1350, 2001.

\bibitem[DMS10]{DMS}
Alfred Dolich, Chris Miller, and Charles Steinhorn.
\newblock Structures having o-minimal open core.
\newblock {\em Trans. Amer. Math. Soc.}, 362:1371--1411, 2010.

\bibitem[For10]{fornasiero-matroids}
Antongiulio Fornasiero.
\newblock Dimensions, matroids, and dense pairs of first-order structures.
\newblock Submitted, 2010.

\bibitem[FS09]{FS}
Antongiulio Fornasiero and Tamara Servi.
\newblock Definably complete and baire structures.
\newblock Submnitted, 2009.

\bibitem[Kec95]{kechris}
Alexander~S. Kechris.
\newblock {\em Classical descriptive set theory}, volume 156 of {\em Graduate
  Texts in Mathematics}.
\newblock Springer-Verlag, New York, 1995.

\bibitem[Kur66]{kuratowski}
Kazimierz Kuratowski.
\newblock {\em Topology. {V}ol. {I}}.
\newblock New edition, revised and augmented. Translated from the French by J.
  Jaworowski. Academic Press, New York, 1966.

\bibitem[LS95]{LS}
Michael~C. Laskowski and Charles Steinhorn.
\newblock On o-minimal expansions of {A}rchimedean ordered groups.
\newblock {\em J. Symbolic Logic}, 60(3):817--831, 1995.

\bibitem[Mil96]{miller96}
Chris Miller.
\newblock A growth dichotomy for o-minimal expansions of ordered fields.
\newblock In {\em Logic: from foundations to applications ({S}taffordshire,
  1993)}, Oxford Sci. Publ., pages 385--399. Oxford Univ. Press, New York,
  1996.

\bibitem[Mil01]{miller}
Chris Miller.
\newblock Expansions of dense linear orders with the intermediate value
  property.
\newblock {\em J. Symbolic Logic}, 66(4):1783--1790, 2001.

\bibitem[Mil05]{miller05}
Chris Miller.
\newblock Tameness in expansions of the real field.
\newblock In {\em Logic {C}olloquium '01}, volume~20 of {\em Lect. Notes Log.},
  pages 281--316. Assoc. Symbol. Logic, Urbana, IL, 2005.

\bibitem[MS99]{MS99}
Chris Miller and Patrick Speissegger.
\newblock Expansions of the real line by open sets: o-minimality and open
  cores.
\newblock {\em Fund. Math.}, 162(3):193--208, 1999.

\bibitem[Oxt71]{oxtoby}
John~C. Oxtoby.
\newblock {\em Measure and category. {A} survey of the analogies between
  topological and measure spaces}.
\newblock Springer-Verlag, New York, 1971.
\newblock Graduate Texts in Mathematics, Vol. 2.

\bibitem[Pil87]{pillay87}
Anand Pillay.
\newblock First order topological structures and theories.
\newblock {\em J. Symbolic Logic}, 52(3):763--778, 1987.

\bibitem[Poi85]{poizat85}
Bruno Poizat.
\newblock {\em Cours de th\'eorie des mod\`eles}.
\newblock Bruno Poizat, Lyon, 1985.
\newblock Une introduction {\`a} la logique math{\'e}matique contemporaine.

\bibitem[Rob74]{robinson74}
Abraham Robinson.
\newblock A note on topological model theory.
\newblock {\em Fund. Math.}, 81(2):159--171, 1973/74.
\newblock Collection of articles dedicated to Andrzej Mostowski on the occasion
  of his sixtieth birthday, II.

\bibitem[Sco69]{scott}
Dana Scott.
\newblock On completing ordered fields.
\newblock In {\em Applications of {M}odel {T}heory to {A}lgebra, {A}nalysis,
  and {P}robability ({I}nternat. {S}ympos., {P}asadena, {C}alif., 1967)}, pages
  274--278. Holt, Rinehart and Winston, New York, 1969.

\bibitem[vdD98a]{vdd-dense}
Lou van~den Dries.
\newblock Dense pairs of o-minimal structures.
\newblock {\em Fund. Math.}, 157(1):61--78, 1998.

\bibitem[vdD98b]{vdd}
Lou van~den Dries.
\newblock {\em Tame topology and o-minimal structures}, volume 248 of {\em
  London Mathematical Society Lecture Note Series}.
\newblock Cambridge University Press, Cambridge, 1998.

\end{thebibliography}

\end{document}